\documentclass[a4paper, notitlepage, 11pt]{amsart}
\usepackage[english]{babel}
\usepackage[applemac]{inputenc}
\usepackage{amscd, amsfonts, amsmath, amssymb, amsthm, array, bm, graphicx, manfnt, dsfont, geometry, bm, pstricks,epstopdf}
\usepackage{mathsymbols, maththeorems}
\usepackage[all,cmtip]{xy}

\allowdisplaybreaks

\usepackage{maththeorems}
\usepackage{mathsymbols}
\newcommand{\vol}{\mathrm{vol}}

\newcommand{\Betadistr}{\mathrm{Beta}}
\newcommand{\Beta}{\mathrm{B}}
\newcommand{\Res}{\mathrm{Res}}
\newcommand{\bfal}{\boldsymbol{\alpha}}
\newcommand{\bp}{\mathbf{P}}
\newcommand{\be}{\mathbf{E}}

\begin{document}

\begin{abstract}
We discuss the spectral asymptotics of some open subsets of the real line with random fractal boundary and of a random fractal, the continuum random tree. In the case of open subsets with random fractal boundary we establish the existence of the second order term in the asymptotics almost surely and then determine when there will be a central limit theorem which captures the fluctuations around this limit. We will show examples from a class of random fractals generated from Dirichlet distributions as this is a relatively simple setting in which there are sets where there will and will not be a central limit theorem. The Brownian continuum random tree can also be viewed as a random fractal generated by a Dirichlet distribution. The first order term in the spectral asymptotics is known almost surely and here we show that there is a central limit theorem describing the fluctuations about this, though the positivity of the variance arising in the central limit theorem is left open. In both cases these fractals can be described through a general Crump-Mode-Jagers branching process and we exploit this connection to establish our central limit theorems for the higher order terms in the spectral asymptotics. Our main tool is a central limit theorem for such general branching processes which we prove under conditions which are weaker than those previously known.
\end{abstract}

\title[Central limit theorems for the spectra of random fractals]{Central limit theorems for the spectra\\
of classes of random fractals}
\author{Philippe H.\ A.\ Charmoy, David A.\ Croydon and Ben M.\ Hambly}
\address{Mathematical Institute, Radcliffe Observatory Quarter, Oxford OX2 6GG, United Kingdom}
\address{Department of Statistics, University of Warwick, Coventry CV4 7AL, United Kingdom}
	
\date{\today.}

\maketitle

\section{Introduction}

Let $D$ be a non-empty bounded open subset of $\R^d$ for $d\geq 1$ and let $\Delta$ be the Dirichlet Laplacian on $D$.
Then the spectrum $\Lambda$ of $- \Delta$ is discrete and forms a positive increasing sequence
$$
0 < \lambda_1 \leq \lambda_2 \leq \cdots,
$$
where the eigenvalues are repeated according to their multiplicity. Interest in the geometric information about
$D$ encoded by $\Lambda$ started a little over 100 years ago and was crystallised by Kac in his paper
\cite{Kac1966} entitled `Can one hear the shape of a drum?' Or more precisely, does $\Lambda$ determine
$D$ up to isometry? The answer to that question is no in general, as shown in \cite{Gordonetal1992, Milnor1964};
see also \cite{Buseretal1994} for a concise presentation of a family of counterexamples.

However some geometric information about $D$ can be recovered. Weyl's theorem shows that the eigenvalue
counting function $N$ defined by
$$
N(\lambda) = \#\{\lambda_i:\lambda_i \leq \lambda\}
$$
has asymptotic expansion
$$
N(\lambda) = c_1(d) \vol_d(D)\lambda^{d/2} + o(\lambda^{d/2}),
$$
as $\lambda \to \infty$, for some constant $c_1(d)$ depending only on $d$, where $\vol_d$ denotes the
$d$-dimensional Lebesgue measure. Aside from prompting Kac's question this result has led to a large body
of work on the behaviour of the eigenvalue counting function and we now give a very brief description of the
results that have motivated the work we will present here.

As a first extension it is natural to ask about the second order term in this expansion. If $\partial D$ is smooth, then under some assumptions, that there are not too many periodic geodesics, the expansion has a second order term
$$
N(\lambda) = c_1(d) \vol_d(D) \lambda^{d/2} - c_2(d) \vol_{d-1} (\partial D) \lambda^{(d-1)/2} + o(\lambda^{(d-1)/2}),
$$
as $\lambda \to \infty$, for some other constant $c_2(d)$ depending only on $d$. The reader is referred
to \cite{ivrii80, Lai1981, Seeley1978, Seeley1980, Vassiliev1990} and references therein for more information. This means
that, under some regularity conditions, we can recover the size of the domain and that of the boundary from the
spectral asymptotics; in particular, using the isoperimetric inequality, we can determine whether or not $D$ is an open ball.

Interest in the second term of the expansion of $N$ grew further when Berry studied the spectral asymptotics of
domains with a fractal boundary in \cite{Berry1979, Berry1980}. He conjectured that the Hausdorff dimension of
$\partial D$ should drive the second order term. Brossard and Carmona in
\cite{BC1986} studied the associated partition function, a smoothed version of the eigenvalue counting function,
and showed that the Minkowski dimension, $d_M$, was the relevant notion of dimension for the second order
term in the short time expansion of this function.
For the counting function itself a general result of Lapidus \cite{lap1991} shows that, if $d-1<d_M\leq d$, the second order term is of
order $O(\lambda^{d_M/2})$ provided the Minkowski content of the boundary is finite. In general it is difficult to determine
the precise order of growth for the second order term for arbitrary boundaries, however for one-dimensional
domains \cite{LP1993} it was shown that the Minkowski dimension captures the order of growth of the second term
in the asymptotics and the Minkowski content, the constant, when they exist.

The problem of determining the spectral asymptotics has also been considered for sets which are themselves fractal. For some
classes of fractal, such as the Sierpinski gasket, or more generally p.c.f. self-similar sets \cite{Kigami2001} or
generalised Sierpinski carpets \cite{BB1999}, a Laplacian can be defined and shown to have a discrete spectrum.
The exponent for the leading order growth rate in the eigenvalue counting function is called the spectral dimension
and differs from the Hausdorff or Minkowski dimension of the set. If the fractal has enough symmetry, such as for instance
the Sierpinski gasket, then a Weyl type theorem is no longer true \cite{FS1992}, \cite{BarK1997} in that the rescaled limit of
the eigenvalue counting function does not converge. However the Weyl limit does exist for `generic' deterministic p.c.f. self-similar
sets \cite{KL1993} and also for random Sierpinski gaskets \cite{Hambly2000}
and it is natural to ask about the growth of the second order term in these settings.

Our aim is to consider some random fractals where we anticipate more generic behaviour of the counting function.
We will consider both domains with fractal boundaries and fractal sets here. Firstly we will consider the case of open subsets
with fractal boundaries in the
one-dimensional case of a so called fractal string. Our second case will be an example where the set itself is a fractal,
the continuum random tree. In both cases the first order terms in the spectral asymptotics due to the fractal structure
are understood and we will focus on the behaviour of the second order terms.

A fractal string is a set obtained as the complement of a Cantor set in the
unit interval, so can be thought of as a sequence of intervals of decreasing length \cite{LvF2013}. The Dirichlet Laplacian
is then the union of the Dirichlet Laplacians on each interval. Some discussion of the spectral asymptotics of random
fractal strings can be found in \cite{HL2006} where it is shown that for Cantor sets constructed via random iterated
function systems, the second order term due to the boundary exists almost surely. We will consider a suitable
subset of these random fractal strings and determine when the order of the fluctuations about
the boundary term is given by a central limit theorem (CLT).

This turns out to be a subtle question and the existence of a CLT is determined by the rate of convergence in an associated
renewal theorem. We will give a precise statement after introducing all the terminology in Theorem~\ref{thm::spectralAsympString}.
We will then show that when the fractal is generated using a Dirichlet distribution, the existence of a central limit theorem depends
on the particular Dirichlet distribution considered.

An example of what we are able to show is the following. Let $S_{\gamma,\alpha}$, for
$\gamma \in (0,1), \alpha \in \mathbb{N}$, be the random fractal string obtained as the complement of the random Cantor
set generated by subdividing any interval of length $\ell$ into three, retaining two intervals of size $T_1^{1/\gamma}\ell,
T_2^{1/\gamma}\ell$, and removing one of length $\ell(1-T_1^{1/\gamma}-T_2^{1/\gamma})$, where the pair
$(T_1,T_2)$ is independent for each interval and distributed as Dirichlet($\alpha,\alpha$) (that is a Beta($\alpha,\alpha$) distribution
in this simple case) and $0<\gamma<1$. We write $\bp$ for the probability law for the random fractal string and $\be$ for
expectation with respect to $\bp$.
We note that $\gamma$ will be the Minkowski dimension of the random Cantor set $\bp$-almost surely,
that is the dimension of the boundary of the string. We write $N_{\gamma,\alpha}(\lambda)$ for the associated eigenvalue
counting function.

\begin{theorem}\label{thm:ex1} (i) For all $\alpha \in \mathbb{N}$ and $\gamma\in (0,1)$ there is a strictly positive deterministic
constant $C(\gamma,\alpha)$ such that as $\lambda\to\infty$
\[ \lambda^{-\gamma/2}\left(\frac1{\pi} \lambda^{1/2} - N_{\gamma,\alpha}(\lambda)\right) \to C(\gamma,\alpha) \;\;
\mbox{$\bp$-almost surely.} \]
(ii) If $\alpha\leq 59$, then there exists a strictly positive deterministic constant $\sigma(\alpha)$ such that
as $\lambda\to\infty$
\[ \lambda^{\gamma/4}\left(\lambda^{-\gamma/2}\left(\frac1{\pi} \lambda^{1/2} -
N_{\gamma,\alpha}(\lambda)\right) - C(\gamma,\alpha)\right) \to Z, \;\;\mbox{in distribution} \]
where $Z$ is normally distributed with mean 0 and variance $\sigma(\alpha)^2\in (0,\infty)$.\\
(iii) There exists an $\tilde{\alpha}>80$ and a $\gamma\in (0,1)$ such that: if $59<\alpha<\tilde{\alpha}$, then there exists a not-identically-zero periodic function $p_{\gamma,\alpha}(x)$ such that
\[ \be  N_{\gamma,\alpha} (\lambda) = \frac1{\pi} \lambda^{1/2} - C(\gamma,\alpha)\lambda^{\gamma/2}
+ p_{\gamma,\alpha}(\log \lambda) \lambda^{\gamma\eta(\alpha)/2} + o(\lambda^{\eta(\alpha)}), \]
where $\eta(\alpha)=\max\{\Re(\theta_0)\in (-\infty,1):P(\theta_0)=0\}$,
\[ P(\theta):=\prod_{i=0}^{\alpha-1}(\alpha+\theta+i) - \frac{(2\alpha)!}{\alpha!}\]
and, for this range of $\alpha$ we have $1/2< \eta(\alpha) < 1$. In particular
\[ \lambda^{\gamma/4}\left(\lambda^{-\gamma/2}\left(\frac1{\pi} \lambda^{1/2} - N_{\gamma,\alpha}
(\lambda)\right) - C(\gamma, \alpha)\right) \]
does not converge in distribution as $\lambda\to\infty$.
\end{theorem}

\begin{remark}{\rm
(1) The first result gives the almost sure behaviour of the second term in the counting function asymptotics and is
true for random fractal strings constructed using a wide class of distributions on the simplex.\\
(2) In part (iii) we conjecture that it is possible to take $\tilde{\alpha}=\infty$ and any $\gamma \in (0,1)$.
Indeed, towards proving the above
result, we first provide conditions under which a CLT holds (see Theorem \ref{thm::spectralAsympString} and
Section \ref{subsec::numericalExample}), and explain when one will not (see Remark
\ref{rmk::sharpSpeedConvergence}). This distinction is determined by the rate of convergence in a related renewal
theorem and depends on the values of the roots of $P(\theta)=0$,
which we solve numerically (we can also solve this equation analytically for small values of $\alpha$). These computations
demonstrate that we can take $\tilde{\alpha}$ to be at least $81$. Furthermore, although we are not able to prove it
rigorously, the monotonicity of the results suggests that $\tilde{\alpha}$ can be taken arbitrarily large.\\
(3) We also conjecture that, in the case where there is no CLT, i.e.\ $\alpha>59$, the size of the second order term is
determined by $\eta(\alpha)$, in that, $\bp$ almost surely for $\epsilon>0$,
\[ N_{\gamma,\alpha}(\lambda) =  \frac1{\pi} \lambda^{1/2} - C(\gamma,\alpha) \lambda^{\gamma/2} +
O(\lambda^{\gamma\eta(\alpha)/2+\epsilon}), \]
where $1/2< \eta(\alpha)< 1$ and $\eta(\alpha)\to 1$ as $\alpha\to\infty$.\\
(4) The proof of the above result shows that the period of $p_{\gamma,\alpha}$ is given by $4\pi/\gamma|\mathfrak{I}(\theta_0)|$,
where $\theta_0$ is one of the complex conjugate pair of roots whose real part gives $\eta(\alpha)$.}
\end{remark}

Observe that, as $\alpha$ increases, the Beta$(\alpha,\alpha)$ distribution becomes closer to the
distribution given by a delta measure at the point (1/2,1/2). If we take $\gamma = \ln 2/\ln 3$, then we anticipate that
our random fractal string should converge (in a suitable sense) to the Cantor string (the string formed as the complement of
the classical ternary Cantor set) as $\alpha$ goes to infinity. It is known that for the Cantor string there is a non-constant periodic
function that appears in the second order term in the counting function asymptotics \cite{LvF2013}. Thus our result suggests
that there is a non-trivial transition in the parameter space from the case where there is `enough randomness' for a CLT about
the second order term, to the case where there is not, through to the limit, where there is not even a strong law of large numbers
for this term.

We will also consider the case of the Brownian continuum random tree, a random self-similar fractal. It was shown
in \cite{CH2008} that there was a Weyl limit for the counting function in this case. It was also shown that the
second order term for this fractal set was of order 1 in mean -- which would be anticipated as the boundary of the
tree is just two points, a 0-dimensional set. In this paper we show that there is a CLT about the almost sure
asymptotics. However at this point we have not shown strict positivity of the variance due to the complexity of
the correlation structure in the variance of the limit of the rescaled counting function. We conjecture that there
will be a non-trivial CLT for this counting function. This will show that the randomness in the construction means the
second order term in the spectral asymptotics is determined by the fluctuations about the leading order term, as
these are much greater than the effects due to the boundary of the set.

The main technical tool we develop is a central limit theorem for the general Crump-Mode-Jagers branching process.
In our setting the random fractal sets, the random Cantor set boundary of the string, or the continuum random tree, can
be encoded as general branching processes. We are able to use a characteristic associated with these processes to
determine the behaviour of the counting function. In this case there may be dependence on the offspring of an
individual and we obtain a CLT in this more general setting, extending the work of \cite{JN1984}. We also remark that
the techniques used here can easily be applied to geometric counting functions or other functions associated with
heat flow, such as the partition function or heat content of the set. We anticipate similar behaviour in the fluctuations
of these quantities about their almost sure limits.

The paper is organised as follows. In Section \ref{bpcltsec}, we recall the definition of the general branching process
and some laws of large numbers for such processes. We then prove our central limit theorem for the general branching
process using a Taylor expansion proof. In Section \ref{sec::dirichletWeights}, we restrict ourselves to general branching
processes where a suitable function of the birth times is chosen to lie on an $n$-dimensional simplex, which will ensure
that the limit of the usual branching process martingale is a constant. We will call such processes $\Delta_n$-GBPs and
discuss extensively how to establish the conditions required for the central limit theorem in this setting
as this will allow us to illustrate when we do and do not have a central limit theorem for the associated general branching process.
In Section \ref{sec::CantorStrings}, we define a family of open subsets $U$ of $[0,1]$ whose random boundary is
a statistically self-similar Cantor set built using scale factors on the simplex. We are then able to show our main result which gives
conditions for the existence of a central limit theorem. In Section 5 we consider some examples where the law of
the $\Delta_n$-GBP is given by a Dirichlet distribution. We show
that, for some Dirichlet weights, the eigenvalue counting function of the set $U$ satisfies a central limit theorem. As
a consequence we will be able to establish Theorem~\ref{thm:ex1}. In Section~6 we turn to the
continuum random tree. We recall that this tree can be viewed as a random self-similar set and how to construct a
Laplace operator on it. We then show that the conditions for the general branching process central limit theorem hold
and hence there is a CLT in the spectral asymptotics.

\subsection*{Notation}

For convenience, we will use the shorthand notation $c_i$ with $i \in \N$ to mean some positive constant whose value is fixed for the length of a proof or a subsection.

\section{A central limit theorem for general branching processes}\label{bpcltsec}

\subsection{General branching processes}\label{bpsubsec}

In this subsection, we introduce the general or C-M-J branching process. The presentation is inspired by
\cite{Hambly2000, Jagers1975, Nerman1981}, to which the reader is referred for further information.

In the general branching process, the typical individual $x$ is born at time $\sigma_x$, has offspring whose \emph{birth times}
are determined by a point process $\xi_x$ on $(0, \infty)$, a \emph{lifetime} modelled as a non-negative random variable $L_x$,
and a (possibly random) c\`adl\`ag function $\phi_x$ on $\R$ called a \emph{characteristic}.

We index the individuals of the general branching process using the address space
\begin{equation}\label{eq::addressSpace}
I = \bigcup_{k \geq 0} \N^k, \quad \text{where} \quad  \N^0 = \emptyset.
\end{equation}
The ancestor $\emptyset$ is born at time $\sigma_\emptyset = 0$, and individual $x$ has $\xi_x(0, \infty)$ offspring whose
birth times $\sigma_{x, i}$ satisfy
$$
\xi_x = \sum_{i = 1}^{\xi_x(\infty)} \delta_{\sigma_{x,i} - \sigma_x},
$$
where $\delta$ is the Dirac measure and $x,i$ is the concatenation of $x$ and $i$. The trace of the underlying Galton-Watson
process is a random subtree of $I$ which we denote by $\Sigma$.
We write $\partial \Sigma$ for the set of infinite lines of
descent in the process. For $x,y\in\Sigma$ we also use the notation $x\leq y$ if there exists a sequence $(z_1,\dots,z_k)$ with
$z_i\in \N, i=1,\dots,k$ with $k\in \N$ such that $y=(x,z_1,\dots,z_k)$. Similarly, for $x\in\Sigma$, $y\in\partial\Sigma$ we write $x\leq y$ if there exists a sequence $(z_1,z_2,\dots)$ with $z_i\in \N, i=1,2,\dots$ such that $y=(x,z_1,z_2,\dots)$. A cut-set $\cC$ of $\Sigma$ is a collection of $x \in \Sigma$
such that $x\not\leq x'$ and $x'\not\leq x$ for all $x'\neq x\in \cC$ and $\forall y \in \partial \Sigma$ there is an $x\in\cC$
such that $x\leq y$.

It is customary to assume that the triples $(\xi_x, L_x, \phi_x)_x$ are i.i.d.\ but we allow $\phi_x$ to depend on the progeny of $x$; we
also do \emph{not} make any assumptions on the joint distribution of $(\xi_x, L_x, \phi_x)$. When discussing a generic individual, it is
convenient to drop the dependence on $x$ and write $(\xi, L, \phi)$. We will write $\bp$ for the associated probability law and $\be$
for its expectation.

We define
$$
\xi(t) = \xi((0, t]), \quad \nu(dt) = \bE \xi(dt), \quad \xi_\gamma(dt) = e^{-\gamma t} \xi(dt), \quad \text{and} \quad \nu_\gamma(dt) = \bE \xi_\gamma(dt),
$$
for $\gamma \in (0,\infty)$. Furthermore, we will always assume that the general branching process has \emph{Malthusian growth}, i.e.\ that there exists a \emph{Malthusian parameter} $\gamma \in (0, \infty)$ for which
$\nu_\gamma(\infty) = 1$. This implies, in particular, that the general branching process is super-critical.

We denote the moments of the probability measure $\nu_\gamma$ by
\begin{equation}\label{firstmoment}
\mu_k = \int_0^\infty s^k \nu_\gamma(ds).
\end{equation}
In all cases of interest to us, $\mu_1$ will be finite. Note, however, that some convergence results still hold when that is not the case, as explained in \cite{Nerman1981}.

The presence of the characteristic $\phi$ in the population is captured using the \emph{characteristic counting process} $Z^\phi$ defined as
\begin{equation}
Z^\phi(t) = \sum_{x \in \Sigma} \phi_x (t - \sigma_x) = \phi_\emptyset(t) + \sum_{i = 1}^{\xi_\emptyset(\infty)} Z_i^\phi(t- \sigma_i), \label{eq:ccp}
\end{equation}
where the $Z^\phi_i$ are i.i.d.\ copies of $Z^\phi$. An important example in the study of random fractals is the characteristic $\phi(t) = (\xi(\infty)- \xi(t))\mathbf{1}_{[0,\infty)}(t)$,
whose corresponding counting process $Z^\phi$ has the property that $Z^\phi(t)$ is the number of offspring born after time $t$ to parents born up to time $t$. Later, we will define characteristics that count eigenvalues of the Dirichlet Laplacian.

There are two central elements in the study of the asymptotics of the counting process. The first is that the functions
$$
z^\phi(t) = e^{-\gamma t} \bE Z^\phi(t) \quad \text{and} \quad u^\phi(t) = e^{-\gamma t} \bE \phi(t),
$$
satisfy the well-studied renewal equation
\begin{equation}
z^\phi(t) = u^\phi(t) + \int_0^\infty z^\phi(t-s) \nu_\gamma(ds);\label{eq:zrenewal}
\end{equation}
see \cite{Feller1968} for a classic exposition and \cite{Jagers1975, Karlin1955, LV1996} for alternative results.

The second is the process defined by
\[M_t = \sum_{x \in \Lambda_t} e^{-\gamma \sigma_x},\]
where
\[\Lambda_t = \{x\in\Sigma:\:x = (y, i) \mbox{ for some }y\in\Sigma,i\in\N,\mbox{ and }\sigma_{y} \leq t < \sigma_x\}\]
is the set of individuals born after time $t$ to parents born up to time $t$. The process $M$ is a non-negative c\`adl\`ag $\cF_t$-martingale with unit expectation, where
$$
\cF_t = \sigma(\cF_x, \sigma_x \leq t) \quad\text{and} \quad \cF_x = \sigma(\{(\xi_y, L_y): \sigma_y \leq \sigma_x\});
$$
we call it the \emph{fundamental martingale} of the general branching process.

The martingale convergence theorem shows that $M_t \to M_\infty$ as $t \to \infty$, almost-surely, for some random
variable $M_\infty$. Furthermore, under an $x \log x$ condition standard in the theory of branching processes, $M$ is
uniformly integrable. More precisely, in \cite{Doney1972, Doney1976}, Doney proved the following result.

\begin{theorem}[Doney]
	The following are equivalent:
	\begin{enumerate}
		\item $\bE \left[\xi_\gamma(\infty) (\log \xi_\gamma(\infty))_+ \right] < \infty$;
		\item $\bE M_\infty > 0$;
		\item $\bE M_\infty = 1$;
		\item $M_\infty > 0$ almost surely on the set where there is no extinction;
		\item $M$ is uniformly integrable.
	\end{enumerate}
	Otherwise, $M_\infty = 0$ almost surely.
\end{theorem}

For technical reasons, it is often easier to apply renewal theory under the assumption that $\phi$ vanishes for negative times. When that is not the case, we can set
\begin{equation}\label{eq::oneSidedToTwoSided}
\chi_x(t) = \phi_x(t) \bone_{t \geq 0} + \sum_{i = 1}^{\xi_x(\infty)} Z_{x,i}^\phi(t- \sigma_i) \bone_{0 \leq t < \sigma_i}
\end{equation}
so that
$$
Z^\phi \bone_{[0, \infty)}(t) = \chi_\emptyset(t) + \sum_{i=1}^{\xi_\emptyset(\infty)} Z_i^\phi(t- \sigma_i) \bone_{t- \sigma_i \geq 0}.
$$
This means that $Z^\phi \bone_{[0, \infty)} = Z^\chi$, the counting process of the characteristic $\chi$, and we can then work with $Z^\chi$ instead of $Z^\phi$ because $\chi$ vanishes for negative times and $Z^\chi$ and $Z^\phi$ obviously have the same asymptotics as $t \to \infty$.

\subsection{Application to statistically self-similar fractals}
\label{subsec::appGBPtoFractals}

As discussed in \cite{Falconer1986a, Graf1987, MW1986}, the general branching process provides a natural way to encode statistically self-similar sets. We outline this connection now.

To build a statistically self-similar set $K$, we start with the address space $I$ defined in \eqref{eq::addressSpace} and a non-empty compact set $K_\emptyset$. To each $x \in I$, we associate a random collection $(N_x, \Phi_{x, 1}, \dots \Phi_{x, N_x})_{x \in I}$, where $N_x$ is a natural number and $\Phi_{x,i}$ are contracting similitudes whose ratios we write $R_{x, i}$. We assume that the collection is i.i.d.\ in $x$.

The random numbers $(N_x, x \in I)$ generate a random subtree $\Sigma$ of $I$ defined by $\emptyset \in \Sigma$ and
$$
\quad y = y_1, \dots, y_n \in \Sigma \iff y_1, \dots, y_{n-1} \in \Sigma \text{ and } y_n \leq N_{y_1,\dots, y_{n-1}}.
$$
For $x = x_1, \dots, x_n \in \Sigma$, define
$$
K_x = \Phi_{x_1} \circ \dots \circ \Phi_{x_1, \dots, x_n}(K_\emptyset) \quad \text{and} \quad K = \bigcap_{n = 1}^\infty \bigcup_{|x| = n} K_x,
$$
where $|x|$ is the length of the word $x$. The set $K$ has the intuitive property that it can be written as scaled i.i.d.\ copies of itself, namely,
$$
K = \bigcup_{i=1}^N \Phi_i(K_i),
$$
where $K_1, \dots, K_N$ are i.i.d.\ copies of $K$.

Let us emphasise that the choice of $K_\emptyset$ is not unique in general. However, for technical reasons discussed in \cite{Falconer1986a}, we make the following two assumptions. First, we assume that the sets $(\interior K_x, x \in \Sigma)$ \emph{form a net}, i.e.\
$$
x \leq y \implies \interior K_y \subset \interior K_x
$$
and also
$$
\interior K_x \cap \interior K_y = \emptyset \text{ if neither $x \leq y$ nor $y \leq x$};
$$
the analogue of the \emph{open set condition} for self-similar sets. Second, we assume that the construction of $K$ is proper.  i.e.\ that every cut-set $\cC$ of $\Sigma$ satisfies the condition: for every $x \in \cC$, there exists a point in $K_x$ that does not lie in any other $K_y$ with $y \in \cC$.

The Hausdorff dimension of statistically self-similar sets is almost surely constant on the event that it is not empty and was calculated in \cite{Falconer1986a, Graf1987, MW1986}. It is given in the following result by a formula, the random analogue of that due to Moran \cite{Moran1946} and Hutchinson \cite{Hutchinson1981} familiar from the deterministic setup.

\begin{theorem}
	Let $K$ be a statistically self-similar set. Write $(N, R_1, \dots, R_N)$ for the number of similitudes and their ratios. Then, on the event that the set $K$ is not empty,
	$$
	\dim K = \inf\left\{ s : \bE \left(\sum_{i=1}^N R_i^s \right) \leq 1 \right\} \text{ a.s.}
	$$
\end{theorem}

To specify a general branching process corresponding to the random set $K$, we set
$$
\xi_x = \sum_{i=1}^{N_x} \delta_{-\log R_{x,i}},
$$
and $L_x = \sup_i \sigma_{x,i} - \sigma_x$. With this parametrisation, the set $K_x$ in the construction of $K$ corresponds to an individual born at time $\sigma_x$ and has size $e^{-\sigma_x}$. Furthermore, since
$$
\bE\int_{0}^\infty e^{-s x} \xi(dx) = \bE \left( \sum_{i=1}^N R_i^s \right),
$$
the Malthusian parameter $\gamma$ is equal to the Hausdorff dimension of $K$ by definition.

\subsection{Laws of large numbers}
Before we can prove our central limit theorem for the general branching process, we state Nerman's laws of
large numbers, proved in \cite{Nerman1981}. They are proved for non-negative characteristics with progeny
dependence. In applications, if this is not the case, it suffices to write the characteristic as the difference of its
positive and negative parts.

We start with the weak law of large numbers.
Recall that a measure is said to be lattice if its support is contained in a discrete
subgroup of $\R$ and non-lattice otherwise.

\begin{theorem}\label{thm::NermanWLLN}
	Let $(\xi_x, L_x, \phi_x)_x$ be a general branching process with Malthusian parameter $\gamma$, where $\phi \geq 0$ and $\phi(t) = 0$ for $t < 0$. Assume that $u^\phi$ is directly Riemann integrable and that $\nu_\gamma$ is non-lattice. Assume further that, for every $t$,
	$$
	\bE \left[ \sup_{u \leq t} \phi(u) \right] < \infty.
	$$
	Then,
	$$
	z^\phi(t) \to z^\phi(\infty) = \mu_1^{-1}{\int_0^\infty u^\phi(s) ds},
	$$
	where $\mu_1$ is defined in \eqref{firstmoment}, and
	$$
	e^{-\gamma t} Z^\phi(t) \to z^\phi(\infty) M_\infty, \text{ in probability},
	$$
	as $t \to \infty$, where $M_\infty$ is the almost sure limit of the fundamental martingale of the general branching process. Furthermore, if $M$ is uniformly integrable, then the convergence also takes place in $L^1$.
\end{theorem}

The strong law of large numbers requires the following additional regularity condition.

\begin{condition}\label{cond::ghFunCondition}
There exist non-increasing bounded positive integrable c\`adl\`ag functions $g$ and $h$ on $[0, \infty)$ such that
$$
\bE \left[ \sup_{t \geq 0} \frac{ \xi_\gamma(\infty) - \xi_\gamma(t)}{g(t)} \right] < \infty \quad \text{and} \quad \bE \left[ \sup_{t \geq 0} \frac{e^{-\gamma t} \phi(t)}{h(t)} \right] < \infty.
$$	
\end{condition}

This first part of the condition is satisfied if there exists a non-increasing bounded positive function $g$ such that $\nu_\gamma(1/g)$ is finite, because then
$$
\frac{\xi_\gamma(\infty)- \xi_\gamma(t)}{g(t)} \leq \int_t^\infty \frac{1}{g(s)}\xi_\gamma(ds) \leq \int_0^\infty \frac{1}{g(s)} \xi_\gamma(ds),
$$
which has finite expectation. As such, this can be thought of as a moment condition that is weaker than imposing that $\nu_\gamma$ have a finite second moment; take $g(t) = t^{-2} \wedge 1$.

In particular, if the expected number of offspring is finite, this part of the condition is satisfied since, with the latter choice of $g$,
$$
\bE \int_0^\infty g(t)^{-1} e^{-\gamma t} \xi(dt) \leq \sup_{t \geq 0} \{(1 \vee t^{2}) e^{-\gamma t}\} \bE \xi(\infty) < \infty.
$$

We can now state the strong law of large numbers.

\begin{theorem}\label{thm::NermanSLLN}
	Let $(\xi_x, L_x, \phi_x)_x$ be a general branching process with Malthusian parameter $\gamma$, where $\phi \geq 0$ and $\phi(t) = 0$ for $t < 0$. Assume that $\nu_\gamma$ is non-lattice. Assume further that Condition \ref{cond::ghFunCondition} is satisfied. Then,
	$$
	z^\phi(t) \to z^\phi(\infty) = \mu_1^{-1}{\int_0^\infty u^\phi(s) ds},
	$$
where $\mu_1$ is defined at \eqref{firstmoment}, and
	$$
	e^{-\gamma t} Z^\phi(t) \to z^\phi(\infty) M_\infty, \text{ a.s.},
	$$
	as $t \to \infty$, where $M_\infty$ is the almost sure limit of the fundamental martingale of the general branching process. Furthermore, if $M$ is uniformly integrable, then the convergence also takes place in $L^1$.
\end{theorem}

Similar results have been proved by Gatzouras in the lattice case. We will not use them here and refer the reader to \cite{Gatzouras2000}.

\subsection{The central limit theorem}

In \cite{JN1984}, Jagers and Nerman proved a central limit theorem for the general branching process under the assumptions that the characteristics are i.i.d. We now give a Taylor expansion proof of a similar result, but continue to allow $\phi_x$ to depend on the progeny of $x$. We start by introducing some additional notation.

Consider the general branching process $(\xi_x, L_x, \bar \zeta_x)_x$ with Malthusian parameter $\gamma$. We assume that $\bar \zeta$ is such that
$$
\bar Z(t) := Z^{\bar \zeta}(t)
$$
has zero expectation. In applications, $\bar \zeta$ is typically a suitably centred version of some characteristic $\phi$; we will discuss examples in Section \ref{sec::dirichletWeights}.

We will use the rescaled version $\tilde Z$ of $\bar Z$ defined by
\begin{equation}\label{eq::definitionTildeZ}
\tilde Z(t) = e^{-\gamma t/2} \bar Z(t) = \tilde \zeta_\emptyset(t) + \sum_{i=1}^{\xi(\infty)} e^{-\gamma \sigma_i/2} \tilde Z_i(t -\sigma_i),
\end{equation}
where $\tilde \zeta(t)  = e^{-\gamma t/2} \bar \zeta(t)$.

Finally, to have a proxy for the variance, we define
\begin{equation}\label{eq::varianceBranchingProcess}
V(t) = \bar Z(t)^2 = \rho_\emptyset(t) + \sum_{i = 1}^{\xi(\infty)} V_i(t- \sigma_i),
\end{equation}
where
\[
\rho_\emptyset(t) = \bar\zeta_\emptyset(t)^2 + 2 \bar \zeta_\emptyset(t) \sum_{i = 1}^{\xi(\infty)}
\bar Z_i(t- \sigma_i) + 2 \sum_{i = 1}^{\xi(\infty)} \sum_{j < i} \bar Z_i(t- \sigma_i) \bar Z_j( t- \sigma_j).
\]

As $V$ satisfies an equation of the form \eqref{eq:ccp} which leads to the renewal equation \eqref{eq:zrenewal}, the
functions $v$ and $r$ defined by
\begin{equation}\label{vdefrdef}
v(t) = e^{-\gamma t} \bE V(t) \quad \text{and} \quad r(t) = e^{-\gamma t} \bE \rho(t)
\end{equation}
satisfy the renewal equation
\begin{equation}\label{vrenewal}
v(t) = r(t) + \int_0^\infty v(t-s) \nu_\gamma(ds).
\end{equation}

Our central limit theorem requires two technical conditions which we discuss now.

\begin{condition}\label{cond::integrabilityCondCLT}
	There exists $\epsilon \in (0, 1/2)$ such that
	$$
	e^{-\gamma t/2} \sum_{\sigma_x \leq \epsilon t} \bar \zeta_x(t- \sigma_x) \to 0, \text{ in probability},
	$$
	as $t \to \infty$.
\end{condition}

This is a regularity condition on $\bar \zeta$. In applications, we typically expect $\bar \zeta$ to satisfy a weak law of large numbers. Therefore, the sum should grow like $e^{\gamma \epsilon t}$ and we can expect that the condition is satisfied.

\begin{condition}\label{cond::momentCondCLT}
	There exists $\kappa \in (0, \infty)$ such that
	$$
	\sup_{t \in \R}\bE \{| \tilde Z(t)|^{2+ \kappa}\} < \infty.
	$$
\end{condition}

This is a moment condition. In applications, it is convenient to check it for the third moment, i.e.\ when $\kappa = 1$, because that can be done using renewal arguments.

\begin{theorem}\label{thm::CLT}
	Let $(\xi_x, L_x, \bar \zeta_x)_x$ be a general branching process with Malthusian parameter $\gamma$, where $\bar \zeta$ is such that $\bE\bar Z(t)=0$ for every $t$. Assume that $v$ is bounded and that
	$$
	v(t) \to v(\infty),
	$$
	some finite constant, as $t \to \infty$. Assume further that Conditions \ref{cond::integrabilityCondCLT} and \ref{cond::momentCondCLT} hold.
	Then,
	$$
	\tilde Z(t) \to \tilde Z_\infty, \text{ in distribution},
	$$
	as $t \to \infty$, where the distribution of $\tilde Z_\infty$ is characterised by
	$$
		\bE \left[e^{i \theta \tilde Z_\infty} \right] = \bE\left[e^{- \frac 12 \theta^2 v(\infty) M_\infty}\right].
	$$
\end{theorem}

In the proof, we will use that if $z_1, \dots, z_n$ and $w_1, \dots, w_n$ are complex numbers whose modulus is bounded by $C$, then
\begin{equation}\label{eq::technicalEstimateDurrett}
\left| \prod_{i = 1}^n z_i - \prod_{i=1}^n w_i \right| \leq C^{n-1} \sum_{i = 1}^n|z_i-w_i|.
\end{equation}
A proof of this may be found in
\cite[Lemma 3.4.3]{Durrett2010}.

\begin{proof}
	For $\epsilon \in (0,1/2)$, iterating from \eqref{eq::definitionTildeZ}, the definition of $\tilde Z$, we get
	\begin{equation}\label{eq::tildeZSplitExpression}
	\tilde Z (t) = \sum_{\sigma_x \leq \epsilon t} e^{-\gamma \sigma_x/2} \tilde \zeta_x(t- \sigma_x) + \sum_{x \in \Lambda_{\epsilon t}} e^{-\gamma \sigma_x/2} \tilde Z_x(t-\sigma_x).
	\end{equation}
	The first sum appearing in \eqref{eq::tildeZSplitExpression} can be rewritten as
	$$
	e^{-\gamma t/2} \sum_{\sigma_x \leq \epsilon t} \bar \zeta_x(t- \sigma_x)
	$$
which, by Condition \ref{cond::integrabilityCondCLT}, converges to 0 in probability	 as $t \to \infty$ if we choose $\epsilon$ appropriately small. For the rest of the proof, we fix such a choice of $\epsilon$.

We now consider the other sum appearing in \eqref{eq::tildeZSplitExpression}, and show that it converges in distribution to $\tilde Z_\infty$ as $t \to \infty$. The result will then follow from Slutsky's lemma. In other words, for $\theta \in \R$, we want to show that
	\begin{equation}\label{eq:charFunConvergence}
		\bE\left[e^{i \theta \sum_{x \in \Lambda_{\epsilon t}} e^{-\gamma \sigma_x/2}\tilde Z(t- \sigma_x)} - e^{-\frac 12 \theta^2 v(\infty) M_\infty}\right] \to 0,
	\end{equation}
	as $t \to \infty$. To do this, write, for an $x_0\in(0,1)$ that will be chosen below,
	$$
	A_{\epsilon, t}= \left\{\sup_{x \in \Lambda_{\epsilon t}} |\theta e^{-\gamma \sigma_x/2}\tilde Z_x(t- \sigma_x)| \leq x_0\right\},
	$$
	and split \eqref{eq:charFunConvergence} as
	\begin{align}
		\label{eq::splitFirstBit}
		\bE & \left[ e^{i \theta \sum_{x \in \Lambda_{\epsilon t}} e^{-\gamma \sigma_x/2} \tilde Z_x(t- \sigma_x) } - e^{- \frac 12 \theta^2 v(\infty) M_\infty}; A_{\epsilon, t}^c \right]\\
		\label{eq::splitSecondBit}
		 	& + \bE \left[ e^{i \theta \sum_{x \in \Lambda_{\epsilon t}} e^{-\gamma \sigma_x/2} \tilde Z_x(t- \sigma_x) } - e^{- \frac 12 \theta^2 v(\infty) \sum_{x \in \Lambda_{\epsilon t}} e^{-\gamma \sigma_x}}; A_{\epsilon, t}\right]\\
		\label{eq::splitThirdBit}
			& + \bE \left[e^{- \frac 12 \theta^2 v(\infty) \sum_{x \in \Lambda_{\epsilon t}} e^{-\gamma \sigma_x}} - e^{- \frac 12 \theta^2 v(\infty) M_\infty}; A_{\epsilon, t}\right].
	\end{align}
	We will show that each of these terms converge to $0$ as $t \to \infty$.
	
	Fix $\theta \in \R$ and $\delta \in (0,1)$. Let $x_0=x_0(\delta) \in (0,1)$ be such that
	\[
	\left| e^z - 1- z \right| \leq \delta |z| \quad \text{and} \quad \left| e^z- 1 - z - \frac{z^2}{2} \right| \leq \delta |z|^2,\]
	whenever $z \in \C$ satisfies $|z| \leq x_0$. And let $\tau=\tau(\delta,\theta) \in (0, \infty)$ be such that for $t \geq \tau$,
	\begin{gather}\nonumber
	x_0^{-(2+ \kappa)} |\theta|^{2+ \kappa} \sup_{u \in \R} \bE \{| \tilde Z(u)|^{2 + \kappa}\}  e^{-\gamma \epsilon \kappa t/2} \leq \delta,\\
	\theta^2 \|v\|_\infty e^{-\gamma \epsilon t} \leq x_0\nonumber
	\end{gather}
	and
	\[		|v(\infty) - v(t/2)| \leq \delta,\]
	where $\kappa$ is given by Condition \ref{cond::momentCondCLT}.
		
	Let us start by dealing with \eqref{eq::splitFirstBit}. For $t \geq \tau$,
	\begin{equation*}
	\begin{aligned}
		\bP (A_{\epsilon,t}^c) & = \bP \left(\sup_{x \in \Lambda_{\epsilon t}} |\theta e^{-\gamma \sigma_x/2}\tilde Z_x(t- \sigma_x)|^{2+\kappa} > x_0^{2+\kappa} \right) \\
			& \leq  \bP\left( \sum_{x \in \Lambda_{\epsilon t}} |\theta|^{2+\kappa} e^{-\gamma \sigma_x(1+\kappa/2)} | \tilde Z_x(t- \sigma_x)|^{2+\kappa} \geq x_0^{2 + \kappa}
		\right),
	\end{aligned}
	\end{equation*}
which, by Markov's inequality, is bounded by
\begin{equation}\label{eq::controlProbability}
	\begin{aligned}
	 x_0& ^{-(2+ \kappa)} |\theta|^{2+ \kappa} e^{-\gamma \epsilon \kappa t/2} \bE \left[ \sum_{x \in \Lambda_{\epsilon t}} e^{-\gamma \sigma_x} |\tilde Z_x(t- \sigma_x)|^{2+\kappa} \right]\\
			& \leq  x_0^{-(2+ \kappa)} |\theta|^{2+ \kappa}e^{-\gamma \epsilon \kappa t/2} \bE \left[ \sum_{x \in \Lambda_{\epsilon t}}e^{-\gamma \sigma_x} \bE\{| \tilde Z_x(t- \sigma_x)|^{2+\kappa}| \cF_{\epsilon t}\} \right]\\
			& \leq  x_0^{-(2+ \kappa)} |\theta|^{2+ \kappa} \sup_{u \in \R} \bE \{|\tilde Z(u)|^{2 + \kappa}\} e^{-\gamma \epsilon \kappa t/2}\\
			& \leq  \delta,
	\end{aligned}
	\end{equation}
	where we have used that $\sigma_x \in \cF_{\epsilon t}$, that $\tilde Z_x$ is independent of $\cF_{\epsilon t}$ and that $\bE M_{\epsilon t} = 1$. (This is where we need to control the moment of order $2+ \kappa$ for some $\kappa \in (0,\infty)$.) Therefore, \eqref{eq::splitFirstBit} is dominated by
	\begin{equation*}
	2 \bP(A_{\epsilon, t}^c) \leq 2\delta.
	\end{equation*}
	Since $\delta$ is arbitrary,
it follows that
	$$
	\bE \left[ e^{i \theta \sum_{x \in \Lambda_{\epsilon t}} e^{-\gamma \sigma_x/2} \tilde Z_x(t- \sigma_x) } - e^{- \frac 12 \theta^2 v(\infty) M_\infty}; A_{\epsilon, t}^c \right] \to 0,
	$$
	as $t \to \infty$, as required.
	
	To deal with \eqref{eq::splitThirdBit}, recall that
	$$
	\sum_{x \in \Lambda_{\epsilon t}} e^{-\gamma \sigma_x} = M_{\epsilon t} \to M_\infty \text{, a.s.},
	$$
	as $t \to \infty$. Therefore, dominated convergence implies that
	$$
	\bE \left[ e^{- \frac 12 \theta^2 v(\infty) \sum_{x \in \Lambda_{\epsilon t}}e^{-\gamma \sigma_x}}  - e^{- \frac 12 \theta^2 v(\infty) M_\infty}; A_{\epsilon, t}\right] \to 0,
	$$
	as $t \to \infty$, as required.
	
	We will spend the rest of the proof dealing with \eqref{eq::splitSecondBit}, which we rewrite as
	$$
	\bE\left[ \prod_{x \in \Lambda_{\epsilon t}}  \bE\left\{ \left.  e^{i \theta e^{-\gamma \sigma_x/2}\tilde Z_x(t- \sigma_x)} \bone_{A_{\epsilon, t}^x} \right| \cF_{\epsilon t} \right\} - \prod_{x \in \Lambda_{\epsilon t}}\bE\left\{ \left.  e^{- \frac 12 \theta^2v(\infty) e^{-\gamma \sigma_x}} \bone_{A_{\epsilon, t}^x}  \right| \cF_{\epsilon t} \right\} \right]
	$$
	using the conditional independence built into the branching structure, where
	$$
	A^x_{\epsilon,t} = \{|\theta e^{-\gamma \sigma_x/2} \tilde Z_x(t- \sigma_x)| \leq x_0\}.
	$$

	By \eqref{eq::technicalEstimateDurrett}, the term inside the expectation satisfies
\begin{equation}\label{eq::trickDurrett}
	\begin{aligned}
	 & \left|\prod_{x \in \Lambda_{\epsilon t}}  \bE\left\{ \left.  e^{i \theta e^{-\gamma \sigma_x/2}\tilde Z_x(t- \sigma_x)} \bone_{A_{\epsilon, t}^x} \right| \cF_{\epsilon t} \right\} - \prod_{x \in \Lambda_{\epsilon t}}\bE\left\{ \left.  e^{-\frac 12 \theta^2 v(\infty) e^{-\gamma \sigma_x}} \bone_{A_{\epsilon, t}^x}  \right| \cF_{\epsilon t} \right\}\right| \\
		& \quad \leq \sum_{x \in \Lambda_{\epsilon t}}\left|  \bE\left\{ \left.  e^{i \theta e^{-\gamma \sigma_x/2}\tilde Z_x(t- \sigma_x)} \bone_{A_{\epsilon, t}^x} \right| \cF_{\epsilon t} \right\} - \bE\left\{ \left.  e^{- \frac 12 \theta^2 v(\infty) e^{-\gamma \sigma_x}} \bone_{A_{\epsilon, t}^x}  \right| \cF_{\epsilon t} \right\} \right|.
	\end{aligned}
\end{equation}
	A Taylor expansion and the second order exponential estimate yield that
	\[
	\begin{aligned}
		& \left|  \bE\left\{ \left.  e^{i \theta e^{-\gamma \sigma_x/2}\tilde Z_x(t- \sigma_x)} \bone_{A_{\epsilon, t}^x} \right| \cF_{\epsilon t} \right\} \right.\\
		 	& \quad \quad - \left. \bE\left\{ \left. \left( 1 + i \theta e^{-\gamma \sigma_x/2} \tilde Z_x(t- \sigma_x) - \frac 12 \theta^2 e^{-\gamma \sigma_x} \tilde Z_x(t- \sigma_x)^2 \right) \bone_{A_{\epsilon,t}^x } \right| \cF_{\epsilon t} \right\} \right| \\
		& \quad \leq \delta \theta^2 \|v\|_\infty e^{-\gamma \sigma_x}.
	\end{aligned}
	\]
	Furthermore,
	\[
		\begin{aligned}
			&\left| \bE\left\{ \left. \left( 1 + i \theta e^{-\gamma \sigma_x/2} \tilde Z_x(t- \sigma_x) - \frac 12 \theta^2 e^{-\gamma \sigma_x} \tilde Z_x(t- \sigma_x)^2 \right) \bone_{A_{\epsilon,t}^x } \right| \cF_{\epsilon t} \right\} \right.\\
			& \quad \quad  - \left. \bE\left\{ \left.  e^{- \frac 12 \theta^2 v(\infty) e^{-\gamma \sigma_x}} \bone_{A_{\epsilon, t}^x}  \right| \cF_{\epsilon t} \right\} \right|\\
			& \quad \leq \left| \bE\left\{ \left. \left(i \theta e^{-\gamma \sigma_x/2} \tilde Z_x(t- \sigma_x) \right) \bone_{A_{\epsilon,t}^x } \right| \cF_{\epsilon t} \right\} \right|\\
			& \quad \quad + \left| \bE\left\{ \left. \left( 1 - \frac 12 \theta^2 e^{-\gamma \sigma_x} \tilde Z_x(t- \sigma_x)^2 - e^{- \frac 12 \theta^2 v(\infty) e^{-\gamma \sigma_x}}
\right) \bone_{A_{\epsilon,t}^x } \right| \cF_{\epsilon t} \right\}\ \right|.
		\end{aligned}
	\]
	Now, since
	$$
	\bE \left[ \left. \tilde Z_x(t-\sigma_x) \right| \cF_{\epsilon t} \right] = 0,
	$$
	reasoning as in \eqref{eq::controlProbability} shows that
	\begin{equation*}
	\begin{aligned}
		&\left| \bE\left\{ \left. \left(i \theta e^{-\gamma \sigma_x/2} \tilde Z_x(t- \sigma_x) \right) \bone_{A_{\epsilon,t}^x } \right| \cF_{\epsilon t} \right\} \right|\\
		& \quad = \left| \bE\left\{ \left. \left(i \theta e^{-\gamma \sigma_x/2} \tilde Z_x(t- \sigma_x) \right) \bone_{\Omega \setminus A_{\epsilon,t}^x } \right| \cF_{\epsilon t} \right\} \right| \\
		& \quad \leq x_0 \bE\left\{ \left.  x_0^{-1} |\theta| e^{-\gamma \sigma_x/2} |\tilde Z_x(t- \sigma_x)| \bone_{\Omega \setminus A_{\epsilon,t}^x } \right| \cF_{\epsilon t} \right\}\\
		& \quad \leq x_0 \bE\left\{ \left. x_0^{-(2+ \kappa)} |\theta|^{2+\kappa} e^{-\gamma \sigma_x(1+ \kappa/2)} |\tilde Z_x(t- \sigma_x)|^{2+\kappa} \bone_{\Omega \setminus A_{\epsilon,t}^x } \right| \cF_{\epsilon t} \right\}\\
		& \quad \leq \delta e^{-\gamma \sigma_x},
	\end{aligned}
	\end{equation*}
	for $t \geq \tau$. And, since for $x \in \Lambda_{\epsilon t}$ we have
	$$
	|v(\infty) \theta^2 e^{-\gamma \sigma_x}/2| \leq \theta^2 \|v\|_\infty e^{-\gamma \epsilon t} \leq x_0,
	$$
	a Taylor expansion and the first moment exponential estimate yield that
	\[
		\begin{aligned}
			&\left| \bE\left\{ \left. \left( 1 - \frac 12 \theta^2 e^{-\gamma \sigma_x} \tilde Z_x(t- \sigma_x)^2 - e^{-\frac 12 \theta^2 v(\infty) e^{-\gamma \sigma_x}}\right) \bone_{A_{\epsilon,t}^x } \right| \cF_{\epsilon t} \right\}\ \right|\\
			& \quad \leq \left| \bE\left\{ \left.  - \frac 12 \theta^2 e^{-\gamma \sigma_x} \tilde Z_x(t- \sigma_x)^2 + \frac 12 \theta^2 v(\infty) e^{-\gamma \sigma_x} \right| \cF_{\epsilon t} \right\} \right| + \frac 12 \delta \theta^2 |v (\infty)| e^{-\gamma \sigma_x} \\
			& \quad \leq \theta^2 \{ |v(t- \sigma_x) - v(\infty) | + \delta \|v\|_\infty\} e^{-\gamma \sigma_x}.
		\end{aligned}
	\]
	Notice that, for $t\geq \tau$,
	\begin{equation}\label{eq:splittingRenewalTerm}
	\begin{aligned}
		\sum_{x \in \Lambda_{\epsilon t}} & \theta^2 \{ |v(t- \sigma_x) - v(\infty)| +\delta \|v\|_\infty\} e^{-\gamma \sigma_x}\\
		& \leq \theta^2 \sum_{x \in \Lambda_{\epsilon t} \setminus \Lambda_{t/2}} \delta ( 1 + \|v\|_\infty) e^{-\gamma \sigma_x} + \theta^2 \sum_{x \in \Lambda_{\epsilon t} \cap \Lambda_{t/2}} ( 2 + \delta) \|v\|_\infty e^{-\gamma \sigma_x}\\
		& \leq \delta \theta^2(1 + \|v\|_\infty) M_{\epsilon t} + \theta^2 (2 + \delta) \|v\|_\infty \sum_{x \in \Lambda_{\epsilon t} \cap \Lambda_{t/2}} e^{-\gamma \sigma_x}.
	\end{aligned}
\end{equation}
	
	Together, \eqref{eq::trickDurrett} to \eqref{eq:splittingRenewalTerm} show that \eqref{eq::splitSecondBit} is dominated by
	\begin{equation}\label{eq::cleanBoundSecondBit}
	\bE\left[\delta [ 1 + \theta^2( 1 + 2 \|v\|_\infty)] M_{\epsilon t}\right] + \bE\left[\theta^2 (2 + \delta) \|v\|_\infty \sum_{x \in \Lambda_{\epsilon t} \cap \Lambda_{t/2}} e^{-\gamma \sigma_x}\right].	
\end{equation}
	Fix $c \in(0,\infty)$ large. For $t$ large enough, by Lemma 3.5 of \cite{Nerman1981},
	$$
	\bE\left[ \sum_{x \in \Lambda_{\epsilon t} \cap \Lambda_{t/2}} e^{-\gamma \sigma_x} \right] \leq \bE\left[ \sum_{x \in \Lambda_{\epsilon t} \cap \Lambda_{\epsilon t + c}} e^{-\gamma \sigma_x} \right] \to \mu^{-1}{\int_c^\infty(1- \nu_\gamma(s))ds},
	$$
	as $t \to \infty$. The limiting expression can be made arbitrarily small by choosing $c$ large enough. Therefore, the second term in \eqref{eq::cleanBoundSecondBit} converges to 0.
 Since $\bE M_{\epsilon t} = 1$, the first term can also be made arbitrarily small by adjusting $\delta$. The result follows.
\end{proof}

\subsection{A word on the lattice case}\label{lattice}

The statement of Theorem \ref{thm::CLT} requires us to show that $v(t) \to v(\infty)$. This is typically done using renewal theory. Recall that if $\nu_\gamma$ is lattice, supported on
$b \Z$, say,
then this convergence does not occur. Instead we typically get that, for every $y$,
$$
v(y + bn) \to g(y),
$$
as $n \to \infty$, where $g(y)$ is a $b$-periodic function. In this case, the proof presented above shows that, for every $y$,
$$
\tilde Z (y + bn) \to \tilde Z_\infty(y), \text{ in distribution},
$$
as $n \to \infty$, where the distribution of $\tilde Z_\infty(y)$ is characterised by
$$
\bE \left[ e^{i \theta Z_\infty(y)} \right] = \bE \left[ e^{-\frac 12 \theta^2 g(y) M_\infty}\right].
$$

\section{The central limit theorem for $\Delta_n$-general branching processes}
\label{sec::dirichletWeights}

In this section, we discuss applications of the central limit theorem to the general branching process under the assumption that the number of offspring $\xi(\infty)$ is constant and equal to $n$, say, and that the birth times $\sigma_1,\dots,\sigma_n$ are distributed such that, for some fixed constant $\gamma\in (0,\infty)$, we have $(e^{-\gamma \sigma_1}, \dots, e^{-\gamma \sigma_{n}})$ is a distribution on the simplex $\Delta_n$ in $\mathbb{R}^n$, i.e.
$$
\sum_{i = 1}^n e^{-\gamma \sigma_i} = 1.
$$
The latter assumption ensures that the fundamental martingale $M \equiv 1$. We call a branching process with offspring distribution defined on the simplex in this way a $\Delta_n$-general branching process. For simplicity, we also suppose throughout this section that $\nu_\gamma$ is non-lattice, though as discussed in Section \ref{lattice}, we expect subsequential versions of the results below to hold if this is not the case. A basic example that we will return to is the case where the distribution of $(e^{-\gamma \sigma_1}, \dots, e^{-\gamma \sigma_{n}})$ is Dirichlet$(\alpha_1,\dots,\alpha_n)$.

In this case our general branching process $(\xi_x, L_x, \phi_x)_x$, with Malthusian parameter $\gamma$, will satisfy
a weak law of large numbers, i.e.\
$$
e^{-\gamma t} Z^\phi(t) \to z^\phi(\infty), \text{ in probability},
$$
as $t \to \infty$. We wish to describe the random fluctuations around the limit. To do this, we study the expression
\begin{equation}\label{eq::splitApplicationCLTDirichlet}
e^{\gamma t/2}  \left[ e^{-\gamma t} Z^\phi(t) - z^\phi(\infty) \right] =
e^{-\gamma t/2}\left[Z^\phi(t) - e^{\gamma t} z^\phi(t)\right] + e^{\gamma t/2} [z^\phi(t) - z^\phi(\infty)].
\end{equation}

Our aim is to apply Theorem \ref{thm::CLT} to the first part of this expression. We will see that conditions making
this possible ensure that the second term converges to 0. This will then produce a result on the fluctuations of $Z^\phi(t)$ thanks to Slutsky's lemma.

An outline of this section is that we begin by centring our characteristic in order to apply the central limit theorem
of Section~2. We then show that in order to control the centred characteristic we need to control the rate of
convergence in the associated renewal theorem.

\subsection{Centring the process}

To start, notice that, in the notation of Section \ref{bpcltsec},
$$
Z^\phi(t) - e^{\gamma t} z^\phi(t) = \bar Z(t) = Z^{\bar \zeta}(t)
$$
is a centred characteristic counting process, if we define $\bar \zeta$ by
\begin{equation}\label{eq::defCentring}
\begin{aligned}
\bar \zeta_\emptyset(t) &= \phi_\emptyset(t) + \sum_{i = 1}^n e^{\gamma( t- \sigma_i)} z^\phi(t- \sigma_i) - e^{\gamma t}z^\phi(t)\\
&= \phi_\emptyset (t) + \sum_{i = 1}^n e^{\gamma( t- \sigma_i)} [z^\phi(t- \sigma_i) - z^\phi(t) ],
\end{aligned}
\end{equation}
where we have used that $e^{-\gamma \sigma_1} + \cdots + e^{-\gamma \sigma_n} = 1$. This representation can be
used to get the following control on $\bar \zeta$ under an assumption on $|z^\phi(t- \sigma_i) - z^\phi(t)|$.

\begin{lemma}\label{lem::convergenceRate}
Let $(\xi_x, L_x, \phi_x)_x$ be a $\Delta_n$-general branching process. Assume that
$$
|z^\phi(t) - z^\phi(\infty)| \leq c_1 e^{- \beta_1 t} \wedge c_2,
$$
for some positive constants $c_1,c_2$ and $\beta_1 \in [0, \gamma]$. Then,
$$
|\bar \zeta(t)| \leq |\phi(t)| + 2n \left( c_1 e^{(\gamma - \beta_1) t} \wedge c_2 e^{\gamma t} \right).
$$
\end{lemma}

\begin{proof}
Notice that
\begin{align*}
|\bar \zeta_\emptyset(t)| &\leq |\phi_\emptyset(t)| + \sum_{i = 1}^n e^{\gamma(t- \sigma_i)} \{|z^\phi(t- \sigma_i) - z^\phi(\infty)| + |z^\phi(t) - z^\phi(\infty)|\}\\
& \leq |\phi_\emptyset(t)| + \sum_{i = 1}^n \left(c_1 e^{(\gamma - \beta_1)(t- \sigma_i)} \wedge c_2 e^{\gamma (t- \sigma_i)}\right) + n \left( c_1 e^{(\gamma- \beta_1)t} \wedge c_2  e^{\gamma t} \right)\\
& \leq |\phi_\emptyset(t)| + \left(c_1 \sum_{i = 1}^n e^{(\gamma - \beta_1)(t- \sigma_i)} \wedge c_2 \sum_{i = 1}^n e^{\gamma(t - \sigma_i)}\right) + n \left( c_1 e^{(\gamma- \beta_1)t} \wedge c_2 e^{\gamma t} \right).
\end{align*}
The result follows upon noticing that
$$
\sum_{i = 1}^n e^{(\gamma - \beta_1)(t- \sigma_i)} = e^{(\gamma - \beta_1) t} \sum_{i = 1}^n e^{-(\gamma - \beta_1) \sigma_i} \leq n e^{(\gamma - \beta_1) t},
$$
and proceeding similarly for the other sum.
\end{proof}

This lemma and the decomposition in \eqref{eq::splitApplicationCLTDirichlet} show that understanding the rate of convergence of $z^\phi(t)$ to its limit in the renewal theorem is helpful for estimating the fluctuations of $Z^\phi$.

\subsection{Convergence rate in the renewal theorem}

Consider the renewal equation
\begin{equation}\label{eq::genericRenewalEquation}
z(t) = u(t) + \int_0^\infty z(t-s) F(ds),
\end{equation}
where $F$ is a non-lattice distribution function on $[0, \infty)$. The key to solving the renewal equation is the renewal measure
$$
H = \sum_{n = 0}^\infty F^{*n},
$$
where $F^{*n}$ denotes the $n$-fold convolution of $F$ with itself, because $z$ is typically given by
\begin{equation*}
z(t) = \int_0^\infty u(t-y) H(dy);
\end{equation*}
see \cite{Feller1968, Karlin1955} among others. The renewal theorem of \cite{Feller1968} states that if $F$ has a finite mean $\mu_1$, then
\begin{equation}\label{renewal}
H(t) \sim \mu_1^{-1} t,
\end{equation}
where $f(x) \sim g(x)$ means that $f(x)/g(x) \to 1$ as $x \to \infty$. Under some regularity conditions on $u$, e.g.\ if $u$ is directly Riemann integrable, this can be used to show that
$$
z(t) \to z(\infty) = \frac{1}{\mu_1} \int_{-\infty}^\infty u(s) ds,
$$
as $t \to \infty$; see \cite{Feller1968}.

The error in the linear approximation of the renewal function
\begin{equation}
G(t) = H(t) - \mu_1^{-1} t \label{eq:defnG}
\end{equation}
has been studied extensively. When $F$ has a finite second moment $\mu_2$,
\begin{equation}\label{eq::secondMomentConvergence}
G(t) \to \frac{\mu_2}{2 \mu_1^2},
\end{equation}
as $t \to \infty$.

The rate of convergence of $G(t)$ to its limit in this case can be studied using the Fourier transform $f$ of $F$ defined
by
$$
f(w) = \int_{-\infty}^\infty e^{ws} F(ds), \quad w \in \C.
$$
We refer the reader to \cite{Feller1968, Leadbetter1964, Stone1965a, Stone1965} and references therein for proofs; see also Appendix B of \cite{Kigami2001} for a discussion of the lattice case. In particular, Stone proved the following theorem in \cite{Stone1965a}.

\begin{theorem}[Stone]\label{thm::stoneExpConvergenceRate}
	Suppose that there exists $r_1 \in (0, \infty)$ such that $f(w)$ is analytic and $\neq 1$ when $\Re w \in (0, r_1)$. Then, for every $r \in(0, r_1)$,
	$$
	G(t) - \frac{\mu_2}{2 \mu_1^2} = O(e^{-rt}),
	$$
	as $t \to \infty$.
\end{theorem}

Since we aim to apply Lemma \ref{lem::convergenceRate}, we will be particularly interested in exponential rates of convergence of $z(t)$ to its limit. This rate of convergence is connected to that in Theorem \ref{thm::stoneExpConvergenceRate} by the following lemma adapted from \cite{CH2010} which we include for convenience.

\begin{lemma}\label{lem::convergenceGtoZ}
Let $z$, $u$ and $F$ satisfy the renewal equation \eqref{eq::genericRenewalEquation}. Suppose that
$$
z(t) = \int_0^\infty u(t - y)H(dy) \to \mu_1^{-1} \int_{-\infty}^\infty u(y) dy,
$$
as $t \to \infty$. Then,
$$
z(\infty) - z(t) = \mu_1^{-1} \int_0^\infty u(t+y) dy - \int_0^\infty u(t-y) G(dy).
$$
\end{lemma}

\begin{proof}
	It suffices to notice that
	\begin{align*}
	z(\infty) - z(t) = \mu_1^{-1}\int_{0}^\infty u(t+y) dy + \mu_1^{-1} \int_0^\infty u(t-y) dy - \int_0^\infty u(t-y) H(dy),	
	\end{align*}
	and use the definition of $G$.
\end{proof}

\subsection{Checking the conditions of the central limit theorem}

Discussing the conditions of the central limit theorem involves finding growth estimates for $\bar \zeta$ which can be
verified using Lemma \ref{lem::convergenceRate}. To be compatible with the framework of Nerman's law of large
numbers, we will focus on the situation where the characteristic vanishes for negative times; extensions will be discussed
where needed. We start by looking at Condition \ref{cond::integrabilityCondCLT}.

\begin{lemma}\label{lem::checkIntegrabilityCondition}
Let $(\xi_x, L_x, \phi_x)_x$ be a $\Delta_n$-general branching process and let $\bar \zeta$ be defined
as in \eqref{eq::defCentring}. Assume that $\phi(t) = 0$ for $t < 0$ and that there exists a constant $c_1>0$ such that
$$
|\bar \zeta(t)| \leq c_1 e^{\beta_1 t},
$$
for some $\beta_1 \in (0, \gamma/2)$. Then, Condition \ref{cond::integrabilityCondCLT} is satisfied.
\end{lemma}

\begin{proof}
	Let $\epsilon \in (0, \gamma/2 - \beta_1)$. Then, for $t \geq 0$,
	$$
	\left|e^{-\gamma t/2} \sum_{\sigma_x \leq \epsilon t} \bar \zeta(t- \sigma_i) \right| \leq e^{(\beta_1 + \epsilon - \gamma/2) t} e^{-\epsilon t} \sum_{\sigma_x \leq \epsilon t}c_1.
	$$
	By Theorem \ref{thm::NermanWLLN},
	$$
	e^{-\epsilon t} \sum_{\sigma_x \leq \epsilon t} 1 \to \frac{1}{\mu_1} \int_0^\infty e^{-\gamma t} dt = \frac{1}{\gamma \mu_1} \text{ in probability},
	$$
	as $t \to 0$. So the result follows since $\beta_1 + \epsilon < \gamma/2$.
\end{proof}

Coupled with Lemma \ref{lem::convergenceRate}, this result hints that we should not expect to have a central limit theorem when $z^\phi(t)- z^\phi(\infty)$ does not decay at least as fast as $e^{-\gamma t/2}$, the threshold for the second term in the right-hand side of \eqref{eq::splitApplicationCLTDirichlet} to converge to 0. We will explain more precisely why this is sharp in Remark \ref{rmk::sharpSpeedConvergence}.

Let us now discuss the moment condition for the central limit theorem. The next lemmas discuss a set of sufficient assumptions for Condition \ref{cond::momentCondCLT} to be satisfied for the $\Delta_n$-GBP. We introduce the function
\[ \psi(\theta) = \bE \sum_{i=1}^n e^{-\theta\gamma\sigma_i}, \]
and note that $\psi(1)=1$ and that $\psi(\theta)$ is strictly decreasing in $\theta$.

\begin{lemma}\label{lem::birthTimesSumControl}
Let $(\xi_x, L_x, \phi_x)_x$ be a $\Delta_n$-general branching process. Then, we have
$$
\bE \left[ \sum_{x \in \Sigma} e^{- y \gamma \sigma_x} \right] = \sum_{k = 0}^\infty \psi(y)^k,
$$
and the sum is finite if $y \in (1, \infty)$.
\end{lemma}

\begin{proof}
By monotone convergence,
$$
\bE \left[ \sum_{x \in \Sigma} e^{- y \gamma \sigma_x} \right] = \sum_{k = 0}^\infty \bE \left[ \sum_{|x| = k}
e^{- y \gamma \sigma_x} \right].
$$
Further, conditioning on the birth times, we get that
$$
\bE \left[ \sum_{|x| = k}e^{- y \gamma \sigma_x}  \right] = \bE \left[ \sum_{|x| = k-1} \sum_{i = 1}^n e^{- y \gamma \sigma_x }
e^{- y \gamma (\sigma_{x,i} -\sigma_x)}  \right]= \psi(y) \bE \left[ \sum_{|x| = k-1}e^{- y \gamma \sigma_x}  \right].
$$
Iterating this and summing over $k$ proves the equality with the infinite sum. Noting that $\psi(y) < 1$ for $y \in (1, \infty)$
completes the proof.
\end{proof}

Relying on this, we can produce the following estimate on the third moment of the scaled process $\tilde Z$.

\begin{lemma}\label{lem::thirdMomentPositiveTimes}
	Let $(\xi_x, L_x, \phi_x)_x$ be a $\Delta_n$-general branching process. Assume that $\phi(t) = 0$ for $t < 0$, that
	$$
	|\bar \zeta(t)| \leq c_1 e^{\gamma t/ 2} \text{ for }  t \geq 0
	$$
	and that $v$ is bounded. Then,
	$$
	\sup_{t \geq 0} \bE |\tilde Z(t)|^3 < \infty.
	$$
\end{lemma}

\begin{proof}
	Notice that
	\begin{equation}\label{z3w}
	\bar Z(t)^3 =  W_\emptyset(t) + \sum_{i=1}^n \bar Z_i(t- \sigma_i)^3,
	\end{equation}
	where
	\begin{align*}
	W_\emptyset(t) = & \,\bar \zeta_\emptyset(t)^3 + 3 \bar \zeta_\emptyset(t)^2 \sum_{i=1}^n \bar Z_i(t- \sigma_i) + 3 \bar \zeta_\emptyset(t) \sum_{i,j= 1}^{n} \bar Z_i(t- \sigma_i) \bar Z_j (t- \sigma_j)\\
	 		& + \sum_{\substack{i,j,k= 1\\\text{not all equal}}}^3 \bar Z_i(t- \sigma_i) \bar Z_j (t- \sigma_j) \bar Z_k (t- \sigma_k).
	\end{align*}
	Therefore, it is clear that $|\bar Z(t)|^3$ is bounded by
	\begin{align}
	  \sum_{x \in \Sigma} & |\bar \zeta_x(t- \sigma_x)|^3 \label{eq::firstBitThridMoment}\\
			& + 3 \sum_{x \in \Sigma} |\bar \zeta_x(t- \sigma_x)|^2 \sum_{i = 1}^n |\bar Z_{x,i}(t- \sigma_{x,i})|\label{eq::secondBitThridMoment}\\
			& + 3 \sum_{x \in \Sigma}| \bar \zeta_x(t- \sigma_x)| \sum_{i,j= 1}^{n}| \bar Z_{x,i}(t- \sigma_{x,i})| \,| \bar Z_{x,j} (t- \sigma_{x,j})|\label{eq::thirdBitThirdMoment}\\
			& + \sum_{x \in \Sigma}\sum_{\substack{i,j,k= 1\\\text{not all equal}}}^n |\bar Z_{x,i}(t- \sigma_i)|\,| \bar Z_{x,j} (t- \sigma_j)|\, | \bar Z_{x,k} (t- \sigma_k)|.\label{eq::fourthBitThirdMoment}
	\end{align}
	To prove the lemma, it is sufficient to check that $e^{-3 \gamma t/2}$ times the expectation of each of these terms is bounded, which we do now.
	
	To deal with \eqref{eq::firstBitThridMoment}, note that
	\begin{align*}
	e^{-3 \gamma t/2} \bE \sum_{x \in \Sigma} 	|\bar \zeta_x(t- \sigma_x)|^3 \leq c_1^3 \bE \sum_{x \in \Sigma} e^{- \frac 32 \gamma  \sigma_x} = c_1^3 \sum_{k = 0}^\infty \psi(3/2)^k < \infty,
	\end{align*}
	thanks to Lemma \ref{lem::birthTimesSumControl}.
	
	Notice that our assumption on the boundedness of $v$ implies that
	$$
	\bE \bar Z(t)^2 \leq c_2 e^{\gamma t}.
	$$
	Therefore, we can control the term corresponding to \eqref{eq::secondBitThridMoment} using that
	\begin{align*}
		e^{-3 \gamma t/2} &\bE \sum_{x \in \Sigma} |\bar \zeta_x(t- \sigma_x)|^2 \sum_{i = 1}^n |\bar Z_{x,i}(t- \sigma_{x,i})|\\
		 & \leq c_1^2e^{-3 \gamma t/2}\bE \left[\sum_{x \in \Sigma} e^{\gamma(t- \sigma_x)} \sum_{i = 1}^n \bE [\bar Z_{x,i}(t- \sigma_{x,i})^2 | \cF_x]^{1/2} \right] \\
		 & \leq c_1^2 c_2 \bE \left[ \sum_{ x \in \Sigma} e^{- \gamma \sigma_i} \sum_{i = 1}^n e^{-\gamma \sigma_{x,i} /2} \right]\\
		 & \leq n c_1^2 c_2 \sum_{k = 1}^\infty \psi(3/2)^k\\
		 & < \infty,
	\end{align*}
	using that $\sigma_{x,i} \geq \sigma_x$ and Lemma \ref{lem::birthTimesSumControl}.
	
	To deal with \eqref{eq::thirdBitThirdMoment}, we proceed similarly to get that
	\begin{align*}
	e^{-3 \gamma t/2} & \bE 	\sum_{x \in \Sigma}| \bar \zeta_x(t- \sigma_x)| \sum_{i,j= 1}^{n}| \bar Z_{x,i}(t- \sigma_{x,i})| \,| \bar Z_{x,j} (t- \sigma_{x,j})|\\
		& \leq n e^{-3 \gamma t/2} \bE 	\sum_{x \in \Sigma}| \bar \zeta_x(t- \sigma_x)| \sum_{i= 1}^{n}| \bar Z_{x,i}(t- \sigma_{x,i})|^2\\
		& \leq n^2 c_1 c_2 \sum_{k = 0}^\infty \psi(3/2)^k.
	\end{align*}	
	
	The similar argument needed to bound the term corresponding to \eqref{eq::fourthBitThirdMoment} relies on the observation that, if $i,j$ and $k$ are not all equal, then, assuming without loss of generality that $i$ is different, we have
	\begin{align*}
 \bE&\{|\bar Z_i(t- \sigma_i)|\,| \bar Z_j (t- \sigma_j)|\, | \bar Z_k (t- \sigma_k)| | \cF_ \emptyset\}\\
 & = \bE\{|\bar Z_i(t- \sigma_i)|| \cF_ \emptyset\} \bE \{ | \bar Z_j (t- \sigma_j)|\, | \bar Z_k (t- \sigma_k)| |\cF_ \emptyset\}\\
 &  \leq (\bE[ \bar Z_i(t- \sigma_i)^2 | \cF_ \emptyset])^{1/2} (\bE[ \bar Z_j(t- \sigma_j)^2| \cF_ \emptyset])^{1/2} (\bE[ \bar Z_k(t- \sigma_k)^2|\cF_ \emptyset])^{1/2}.
	\end{align*}
	Using this and reasoning as above then shows that
	\begin{align*}
	e^{-3 \gamma t/2} & \bE \sum_{x \in \Sigma}\sum_{\substack{i,j,k= 1\\\text{not all equal}}}^n |\bar Z_{x,i}(t- \sigma_{x,i})|\,| \bar Z_{x,j} (t- \sigma_{x,j})|\, | \bar Z_{x,k} (t- \sigma_{x,k})|\\
	 & \leq n^3 c_2 \sum_{k = 0}^\infty \psi(3/2)^k.
	\end{align*}
	The proof is complete.
\end{proof}

These results are easily combined to produce a version of the CLT for the case with weights on the simplex.

\begin{theorem}
	Let $(\xi_x, L_x, \phi_x)_x$ be a $\Delta_n$-general branching process such that $\phi(t) = 0$ for $t <0$. Assume that
	$$
	|z^\phi(t) - z^\phi(\infty)| \leq c_1 e^{-\beta_1 t},
	$$
	for some $\beta_1 \in (\gamma/2, \infty)$, that
	$$
	|\phi(t)| \leq c_2e^{\beta_2 t},
	$$
	for some $\beta_2 \in (0, \gamma/2)$, and that $v$ is bounded with $v(t)\to v(\infty)$ as $t\to\infty$. Then the CLT holds in that
	\[ e^{-\gamma t/2} Z^{\bar{\zeta}}(t) \to \tilde{Z}_{\infty}, \text{ in distribution,} \]
	as $t \to \infty$, where the distribution of $\tilde Z_\infty$ is normal with mean 0 and variance $v(\infty)$.
\end{theorem}

\begin{proof}
	It follows from Lemmas \ref{lem::convergenceRate}, \ref{lem::checkIntegrabilityCondition} and \ref{lem::thirdMomentPositiveTimes}
	that Conditions \ref{cond::integrabilityCondCLT} and \ref{cond::momentCondCLT} are satisfied. This, with our other assumptions,
	gives the required conditions for Theorem~\ref{thm::CLT}.
\end{proof}

\section{Spectrum of random self-similar Cantor strings}\label{sec::CantorStrings}

In this section, we discuss the spectral asymptotics of a family of open subsets of $[0,1]$ whose boundary is a random self-similar
Cantor set generated using a distribution on the simplex, called $\Delta_n$-random Cantor strings. Spectral asymptotics for a
variety of Cantor strings have been studied extensively in
\cite{HL2006, LP1993, LP1996, LvF2000} and references therein. We specialise the discussion to $\Delta_n$-random
Cantor strings here so that we can study the fluctuations of the spectrum using the results of the previous section.

\subsection{Construction}

Choose $\gamma \in (0,1)$, $n \geq 2$ and consider the random vector
$(T_1, \dots, T_n)$ with a probability law on the simplex.
Start with the unit interval $K_0 = [0,1]$ and replace it by $n$ equally spaced intervals of length $T_1^{1/\gamma}, \dots,
T_n^{1/\gamma}$. Replace each of these intervals by $n$ intervals created independently with the same procedure. Iterating
indefinitely, we obtain a decreasing sequence of compact sets $(K_j, j \geq 0)$, and $K = \cap_j K_j$ is a statistically self-similar
Cantor set. Indeed, in the notation of Subsection \ref{subsec::appGBPtoFractals}, it suffices to set
$$
(N, R_1, \dots, R_N) = (n, T_1^{1/\gamma}, \dots, T_n^{1/\gamma});
$$
the maps $(\Phi_1, \dots, \Phi_N)$ can easily be deduced from this. The corresponding general branching process $(\xi, L)$
(no characteristic just yet) is obtained as described in Subsection \ref{subsec::appGBPtoFractals}. By construction
this is a $\Delta_n$-general branching process and its Malthusian parameter is $\gamma$.
Figure~\ref{fig::cantorString} depicts the first 4 iterations in the construction of the set $K$ with the distribution $\Dirichlet(1, 1, 1)$.

The set in which we are interested in this chapter is $U = [0,1]\setminus K$, whose boundary is $K$ by construction. Thanks to the
Lindel\"of property, $U$ is a countable union of intervals. As such, $U$ is a \emph{random string} in the sense of
\cite{LP1993, LP1996, LvF2000} and references therein.

\subsection{Fractal dimension}

Consider the set $U$ defined above and recall that the Hausdorff dimension of $\partial U = K$ is $\gamma$. Using Theorem \ref{thm::NermanSLLN}, we now show that the Minkowski dimension of $\partial U$ exists and is also equal to $\gamma$ almost surely.

\begin{theorem}
The Minkowski dimension of the $\Delta_n$-random Cantor string $K$ is almost surely equal to $\gamma$.
\end{theorem}

\begin{proof}
	Consider the characteristic function defined by
	$$
	\phi(t) = \xi(\infty) - \xi(t).
	$$
	The corresponding counting process $Z^\phi$ counts the number of offspring born after time $t$ to parents born up to time $t$. As such, $Z^\phi(t)$ is an upper bound for the covering number $N(e^{-t}, K)$ of $K$ with balls of radius $e^{-t}$.
	
	By the strong law of large numbers, we have
	$$
	e^{-\gamma t} Z^\phi(t) \to \frac{1}{\mu_1} \int_0^\infty \bE e^{-\gamma s}\phi(s) ds \in(0, \infty), \text{ a.s.},
	$$
	as $t \to \infty$. This is easily used to check that $\overline{\dim}_M K \leq \gamma$ almost surely. The result follows since $\dim K = \gamma$ almost surely.
\end{proof}

\begin{figure}
\includegraphics[width = 7cm]{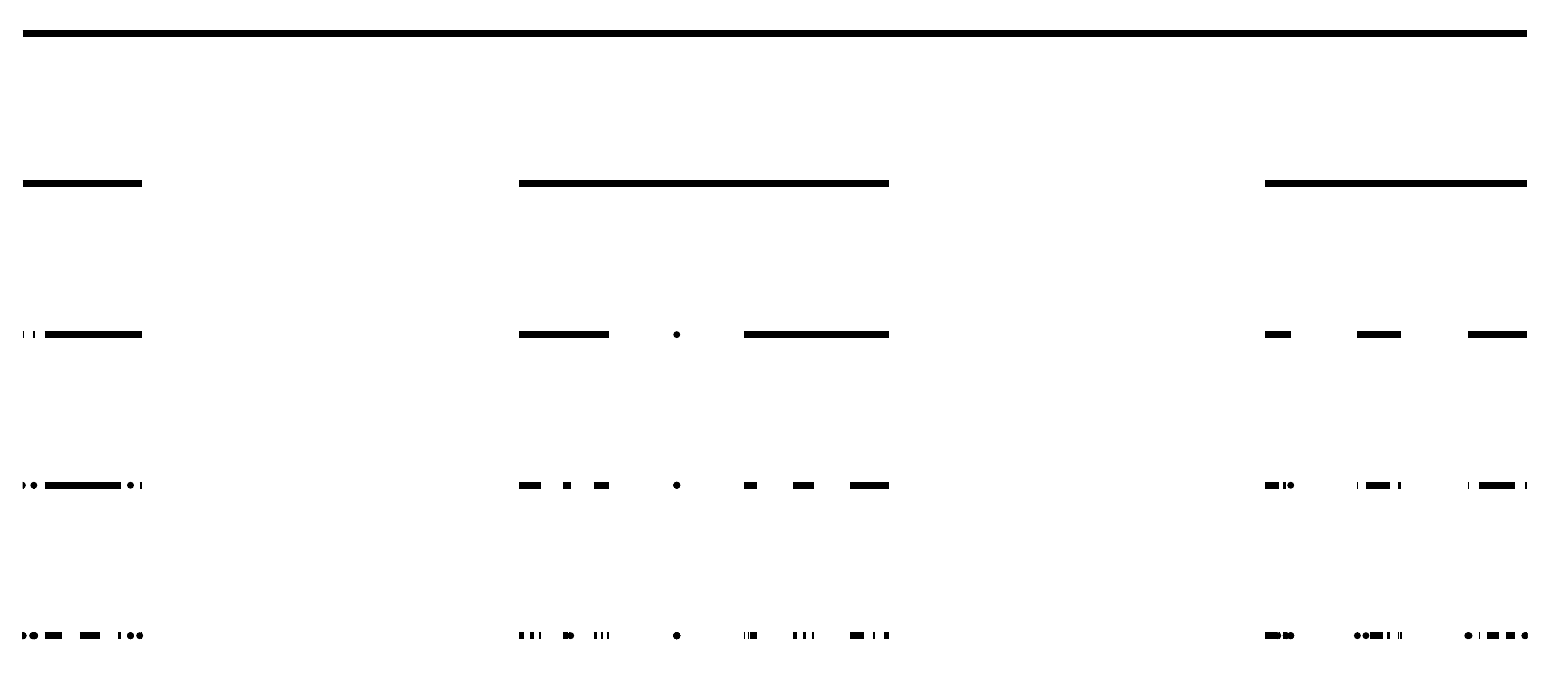}
\caption{First 4 iterations of the construction of $K$ with the distribution $\Dirichlet(1, 1, 1)$ and $\gamma = 0.6$.}
\label{fig::cantorString}
\end{figure}

\subsection{Spectrum of the Dirichlet Laplacian}

Recall that the eigenvalue counting function for a domain $D$ (or a countable union of domains) of $\R$ is defined by
$$
N_D(\lambda) = \#\{\text{eigenvalues of } - \Delta \leq \lambda\}.
$$
Following \cite{LV1996}, we define
$$
\bar N_D(\lambda) = \frac{1}{\pi}\vol_1(D)\lambda^{1/2} - N_D(\lambda).
$$
The function $\bar{N}_D$ has the property that if $D_1$ and $D_2$ are disjoint, then
\begin{equation}\label{eq::separationN}
\bar N_{D_1 \cup D_2}(\lambda) = \bar N_{D_1}(\lambda) + \bar N_{D_2}(\lambda).
\end{equation}
Furthermore, for $r\in (0, \infty)$, a change of variables shows that
\begin{equation}\label{eq::scalingN}
\bar N_{r D}(\lambda ) = \bar N_D(r^2 \lambda).
\end{equation}

In our applications of the central limit theorem, we will rely on the following assumption, which was discussed in the previous section.

\begin{assumption}\label{ass::convergenceRate}
The rate of convergence in the renewal theorem satisfies
$$
\left|G(t) - \frac{\mu^2 + \sigma^2}{2 \sigma^2}\right| \leq c_1 e^{- \beta_1 t}
$$
for large $t$, for some $\beta_1 \in (\gamma/2, \infty)$
\end{assumption}

We may now use the general branching process to study $\bar N_U$.

\begin{theorem}\label{thm::spectralAsympString}
Let $K$ be a $\Delta_n$-random Cantor string with dimension $\gamma$ and consider the string $U = [0,1] \setminus K$. Then,
$$
\lambda^{-\gamma/2}\bar N_U(\lambda) \to \mathfrak{N}, \text{ a.s.\ and in } L^1,
$$
as $\lambda \to \infty$, for some strictly positive constant $\mathfrak{N}$.

Furthermore, if Assumption \ref{ass::convergenceRate} holds, then
$$
\lambda^{\gamma/4} (\lambda^{-\gamma/2}\bar N_U(\lambda) - \mathfrak{N}) \to Z \text{ in distribution},
$$
as $\lambda \to \infty$, where $Z$ has a normal distributions with mean 0 and variance $\sigma^2$ for some strictly
positive constant $\sigma$.
\end{theorem}

\begin{proof}
Define the random variable $S$ by
$$ S = (1- R_1 - \cdots - R_n)/(n-1). $$
By construction and the properties in \eqref{eq::separationN} and \eqref{eq::scalingN}, we have
$$
	\bar N_U(\lambda) = (n-1) \bar N_{S[0,1]}(\lambda) + \sum_{i =1}^n \bar N_{R_i U_i}(\lambda) = (n-1)
	\bar N_{[0,1]}(S^2 \lambda) + \sum_{i=1}^n \bar N_{U_i} (R_i^2 \lambda),
$$
where $U_i$ are i.i.d.\ copies of $U$.
	
Now recall that the eigenvalues of $-\Delta$ for the unit interval $[0,1]$ are $(n\pi)^2$. Therefore,
$$
	\bar N_{[0,1]}(\lambda) = \pi^{-1} \lambda^{1/2} - \lfloor \pi^{-1} \lambda^{1/2} \rfloor,
$$
which is bounded by $1 \wedge ( \pi^{-1} \lambda^{1/2})$.
	
To use the general branching process, set
$$
	\phi(t) = (n-1) \bar N_{[0,1]}(S^{2} e^{2t}) \quad \text{and} \quad Z^\phi(t) = \bar N_U(e^{2t})
$$
so that
$$
	Z^\phi(t) = \phi_\emptyset(t) + \sum_{i=1}^n Z_i^\phi(t- \sigma_i),
$$
where $Z_i^\phi$ are i.i.d.\ copies of $Z^\phi$ and $Z^\phi$ is the counting process of the characteristic $\phi$. Furthermore,
notice that
	\begin{equation}\label{eq::technicalBounds1}
	0 \leq \phi(t) \leq (n-1) e^{t} \bone_{t < 0} + c_1 \bone_{t \geq 0} \quad \text{and} \quad Z^\phi(t) \bone_{t < 0} \leq c_2 e^{t}
	 \bone_{t < 0},
	\end{equation}
	for some positive constants $c_1$ and $c_2$.
	
	To establish the first statement of the theorem, we use \eqref{eq::oneSidedToTwoSided} and set
	$$
	\chi(t) = \phi(t) \bone_{t\geq 0} + \sum_{i =1}^n Z_i^\phi(t- \sigma_i) \bone_{0 \leq t < \sigma_i},
	$$
	which is bounded thanks to \eqref{eq::technicalBounds1}.

 Thanks to Theorem \ref{thm::NermanSLLN}, this implies that
	$$
	e^{-\gamma t} Z^\chi(t) \to \mu_1^{-1} \int_0^\infty e^{-\gamma s} \bE \chi(s) ds \in (0, \infty), \text{ a.s.\ and in } L^1,
	$$
	as $t \to \infty$. By definition, this means that
	$$
	\lambda^{-\gamma/2} \bar N_{U}(\lambda) \to \mu_1^{-1} \int_0^\infty e^{-\gamma s} \bE \chi(s) ds,  \text{ a.s.\ and in } L^1,
	$$
	as $\lambda \to \infty$, as required.
	
	Let us now prove the second part of the theorem under Assumption \ref{ass::convergenceRate}. Consider $\bar \zeta^\phi$ and $\bar \zeta^\chi$ defined as in \eqref{eq::defCentring} using $\phi$ and $\chi$. Notice that
	$$
	\bar Z^\chi : = Z^{\bar \zeta^\chi} (t) \quad \text{and} \quad \bar Z^\phi : = Z^{\bar \zeta^\phi} (t)
	$$
	are equal for $t \geq 0$. In particular, the corresponding variance functions $v^\chi(t)$ and $v^\phi(t)$ are equal for $t \geq 0$. Furthermore, applying the fact that $(\bar\zeta^\phi_x)_x$ is i.i.d.\ together with the bounds from Lemma~\ref{lem::convergenceRate}, it follows that there are positive constants $c,\tau=(2\beta_1-\gamma)\wedge \gamma$ such that
	$$
	r^\phi(t) = e^{-\gamma t} \bE |\bar \zeta^\phi(t)|^2 \leq ce^{-\tau|t|}.
	$$
	These observations and the renewal theorem of \cite{LV1996} imply that
	\begin{equation}\label{eq::finiteLimitingVariance}
	\lim_{t \to \infty} v^\chi(t) = \lim_{t \to \infty} v^\phi(t) = \mu_1^{-1} \int_0^\infty e^{-\gamma t} \bE \bar \zeta^\phi(s)^2 ds  = v(\infty)\in (0, \infty),
	\end{equation}
	say.

	Since $\chi$ is bounded and
	\begin{equation}\label{zchi}
	|z^\chi(t) - z^\chi(\infty)| \leq c_2 e^{-\beta_1 t}
\end{equation}
by Assumption \ref{ass::convergenceRate} and Lemma \ref{lem::convergenceGtoZ}, the conditions of Theorem \ref{thm::CLT} are satisfied. This implies that
$$ e^{\gamma t/2} (e^{-\gamma t} Z^\chi(t) - \mathfrak{N}) \to N(0, v(\infty)), \text{ in distribution},	$$
as $t \to \infty$. Using the definition of $Z^\chi$, the decomposition in \eqref{eq::splitApplicationCLTDirichlet},
(\ref{zchi}) and Slutsky's lemma completes the proof.
\end{proof}

\begin{remark} \label{rmk::sharpSpeedConvergence}	{\rm
The arguments in the proof can be used to show that there cannot be a central limit theorem when the rate of convergence in the
renewal theorem is not fast enough. Suppose that
\[  z^{\phi}(t)- z^{\phi}(\infty) = c_1 e^{-\beta_1 t} + o(e^{-\beta_1 t}), \]
as $t \to \infty$, for some real constant $c_1$. Then notice that the centring $\bar \zeta^\phi$ introduced in \eqref{eq::defCentring} and used in the proof of Theorem \ref{thm::spectralAsympString} satisfies
	\begin{align*}
	\bar \zeta_\emptyset^\phi(t) &= \phi_\emptyset(t) + \sum_{i = 1}^n e^{\gamma(t- \sigma_i)} [z^\phi(t- \sigma_i) - z^\phi(\infty) + z^\phi(\infty) - z^\phi(t)]\\
	& = \phi_\emptyset(t) + c_1 e^{(\gamma- \beta_1) t} \left( \sum_{i = 1}^n e^{-(\gamma - \beta_1) \sigma_i} - e^{-\gamma\sigma_i} \right) + o(e^{(\gamma- \beta_1) t})\\
	& = \phi_\emptyset(t) + c_1 e^{(\gamma- \beta_1) t} R + o(e^{(\gamma- \beta_1) t}),
	\end{align*}
	say, as $t \to \infty$ (where the remainder is deterministic). Notice that $R$ is a strictly positive random variable. In particular, for some $\epsilon_0 \in (0, \infty)$, we have $\bP(R \geq \epsilon_0) > 0$. Therefore, there exists $t_0$ such that, for $t \geq t_0$,
	$$
	|\bar \zeta^\phi_\emptyset(t)| ^2 \geq |\bar \zeta_\emptyset^\phi(t)|^2 \bone_{R \geq \epsilon_0} \geq \frac 12 c_1^2 \epsilon_0^2 e^{2(\gamma - \beta_1) t}\bone_{R \geq \epsilon_0}.
	$$
	This implies that, for $t\geq t_0$,
	$$
	r^\phi(t) = e^{-\gamma t} \bE |\bar \zeta_\emptyset^\phi(t)|^2 \geq \frac12 c_1^2 \epsilon_0^2 e^{(\gamma - 2 \beta_1)t}\bP(R \geq \epsilon_0).
	$$
	If $\beta_1 \in (0, \gamma/2]$, then, by renewal theory, we cannot have a finite limiting variance in \eqref{eq::finiteLimitingVariance} and in particular, no central limit theorem.
	}
\end{remark}

\section{Dirichlet distributions}\label{exs}

We now consider the case where the distribution on the simplex is the Dirichlet distribution and perform some explicit calculations
to show the range of behaviour that is possible in our set up. In this set up the birth times are distributed as
$$
(e^{-\gamma \sigma_1}, \dots, e^{-\gamma \sigma_{n}}) \sim \Dirichlet(\alpha_1, \dots, \alpha_n),
$$
where $\Dirichlet$ denotes the Dirichlet distribution. We will describe this situation by saying that the $\Delta_n$-general branching
process has Dirichlet weights $\bfal = (\alpha_1, \dots, \alpha_n)$ and write
$$
\alpha_0 = \alpha_1 + \cdots + \alpha_n,
$$
as usual.

\subsection{Explicit calculations with Dirichlet weights}

In general, it is difficult to determine the solutions to $f(w) = 1$ needed to use Theorem \ref{thm::stoneExpConvergenceRate}.
In some cases with Dirichlet weights which we discuss now, however, we can use the properties of the Gamma function to study
$f$ and deduce convergence rates for the renewal theorem.

Let $X \sim \Dirichlet(\bfal)$ be a random vector in $\R^n$. It is well-known that, for every $i$,
$$
X_i \sim \Betadistr(\alpha_i, \alpha_0 - \alpha_i),
$$
where $\Betadistr$ denotes the Beta distribution. Recall that if $Y \sim \Betadistr(\beta_1, \beta_2)$ then
$$
\bE Y^\theta = \frac{\Beta(\theta + \beta_1, \beta_2)}{\Beta(\beta_1, \beta_2)},
$$
where $\Beta$ is the Beta function, i.e.\
$$
\Beta (x,y) = \frac{\Gamma(x) \Gamma(y)}{\Gamma(x + y)}.
$$
Therefore, we get that
\begin{equation}\label{phidef}
\psi(\theta) = \bE \left[\sum_{i = 1}^n X_i^\theta \right] = \frac{\Gamma(\alpha_0)}{\Gamma(\alpha_0 + \theta)} \sum_{i=1}^n \frac{\Gamma(\alpha_i + \theta)}{\Gamma(\alpha_i)},
\end{equation}
where the equation defines $\psi$. For the general branching process with Dirichlet weights defined above, it follows that
\begin{equation}\label{eq::defFourierTansform}
f(w) = \int_{-\infty}^\infty e^{ws} \nu_\gamma(ds) = \int_{-\infty}^\infty e^{(w - \gamma) s} \nu (ds) = \bE \left[ \sum_{i = 1}^n e^{-\gamma \sigma_i (1 - w/\gamma)} \right] = \psi(1- w/\gamma).
\end{equation}

If $\alpha_0 - \alpha_i \in \Z$ for every $i$, we can use that $\Gamma(w+1) = w \Gamma(w)$ to reduce the function $\psi$ to a rational function which may be simpler to analyse. Notice that this assumption implies in particular that there exist some $a \in \R$ and  $\ell_i \in \Z$ such that, for every $i$, we have $\alpha_i = a + \ell_i$. Therefore,
$$
\alpha_0 = a + \ell_0 = \alpha_1 + \cdots + \alpha_n = n a + \ell_1 + \cdots + \ell_n,
$$
from which it follows that $(n-1) a \in \Z$.

By definition of $G$ in \eqref{eq:defnG} and writing $F=\nu_{\gamma}$, we have
\begin{align*}
G * F(t) 	& = H*F(t) - \mu_1^{-1} \int_0^\infty (t-s) F(ds)\\
			& = H(t) - \bone_{t \geq 0} - \mu_1^{-1} t + 1\\
			& = G(t) + \bone_{t < 0}.
\end{align*}
Denoting by $g$ the Fourier transform of $G$ and $f$ that of $F=\nu_{\gamma}$, we thus get that
$$
g(w) = \frac{1}{1-f(w)}, \quad w \in \C.
$$
Applying \eqref{phidef} and \eqref{eq::defFourierTansform}, for the general branching process with Dirichlet weights $\bfal$ satisfying $\alpha_0 - \alpha_i \in \Z$ for every $i$, the function $f$ can be written
$$
f(w) = \psi(1- w/\gamma) = \sum_{i = 1}^n \frac{1}{P_i(w)},
$$
where $P_i$ is a polynomial of degree $\alpha_0 - \alpha_i$. It follows that
\begin{align*}
g(w) & = \frac{\prod_{i = 1}^n P_i (w) }{\prod_{i = 1}^n P_i (w) - \sum_{i = 1}^n \prod_{ j \neq i} P_j (w) } \\
& = 1 + \frac{\sum_{i = 1}^n \prod_{ j \neq i} P_j (w) }{\prod_{i = 1}^n P_i (w) - \sum_{i = 1}^n \prod_{ j \neq i} P_j (w) }\\
& = 1 + \frac{R(w)}{Q(w)},
\end{align*}
say. It is easy to see that
$$
\deg R < (n-1) \alpha_0 = \deg Q.
$$
Now, decompose $g$ into partial fractions and write
$$
g(w) = 1 + \sum_{i = 1}^q \frac{Q_i(w)}{(w- \rho_i)^{m_i}},
$$
where $(\rho_i, i \leq q)$ are the roots of $Q$ with corresponding multiplicities $m_i$ and $Q_i$ are polynomials with $\deg Q_i < m_i$ for every $i$; in particular,
$$
m_1 + \dots + m_q = (n-1) \alpha_0.
$$

Recall that, for $k \in \Z_+$ and $\Re(\lambda-r)<0$,
$$
\int_{-\infty}^\infty e^{\lambda t} t^k e^{-rt} \bone_{t \geq 0} dt = \int_{0}^\infty t^k e^{(\lambda -r) t} dt = \frac{ k!}{(r-\lambda)^{k+1}}
$$
and therefore that
\begin{equation}\label{eq::polynomialCoefficients}
\int_0^\infty e^{\lambda t} \frac{d^k}{dt^k}\left( t^k e^{-rt} \right) dt = \frac{ k! \lambda^k}{(r-\lambda )^{k+1}}.
\end{equation}
Using this, it is easy to check that
$$
G(dt) = \delta_0(t) dt + \sum_{\Re \rho_i \leq 0} \tilde Q_i(t) e^{\rho_i t} \bone_{t < 0} dt + \sum_{\Re \rho_i >0} \tilde Q_i(t) e^{-\rho_i t} \bone_{t \geq 0} dt,
$$
where the $\tilde Q_i$ are polynomials determined using \eqref{eq::polynomialCoefficients} and satisfying $\deg \tilde Q_i <m_i$. Of course, since $F$ is supported on $[0, \infty)$, so is $H$ and therefore, by definition of $G$, we have
$$
G(t) \bone_{t < 0} = - \mu_1^{-1} t \bone_{t < 0}.
$$
Putting this together shows that
\begin{equation}\label{eq::fromOfG}
G(dt) = \delta_0(t) dt - \mu_1^{-1} \bone_{t < 0} dt + \sum_{\Re \rho_i > 0} \tilde Q_i(t) e^{-\rho_i t} \bone_{t \geq 0} dt,
\end{equation}
which we can integrate to study the asymptotics of $G$. A particular example which will guide us below is given in the following lemma.

\begin{lemma}\label{lem:simpleroots}
	Assume that all the roots of $Q$ are simple. Then,
	$$
	G(t) = - \mu_1^{-1} \bone_{t < 0} t + \frac{\mu_2}{2 \mu_1^2} \bone_{t \geq 0} + \sum_{\Re \rho_i > 0 } \frac{c_i}{\rho_i} e^{- \rho_i t} \bone_{t \geq 0},
	$$
	where $c_i = \Res(g; \rho_i)$, the residue of $g$ at $\rho_i$.
\end{lemma}

Since in this case none of the singularities can be removed and all have order one, all the $c_i$ are non zero. Furthermore, since if $\rho_i$ is a root with residue $c_i$ then $\bar \rho_i$ is a root with residue $\bar c_i$ as $g(\bar w) = \overline{g(w)}$, roots with the same real part cannot cancel out. In particular, in this case, the result of Theorem \ref{thm::stoneExpConvergenceRate} is sharp.

\begin{proof}
Since all the roots of $Q$ are simple, we have
$$
g(w) = 1 + \sum_{i = 1}^q \frac{c_i}{w- \rho_i},
$$
where we must have $c_i = \Res(g; \rho_i)$. Integrating \eqref{eq::fromOfG} then shows that
$$
G(t) = \bone_{t \geq 0}- \mu_1^{-1} \bone_{t < 0} t - \sum_{\Re \rho_i > 0} \frac{c_i}{\rho_i} \bone_{t \geq 0} + \sum_{\Re \rho_i > 0 } \frac{c_i}{\rho_i} e^{- \rho_i t} \bone_{t \geq 0}.
$$
Since $F$ has a second moment, the result now follows from \eqref{eq::secondMomentConvergence} after letting $t \to \infty$.
\end{proof}

\subsection{Examples} \label{subsec::numericalExample}

Let us first discuss how the observations above enable us to establish the desired rate of convergence for some simple cases of Dirichlet weights.\\[.1in]
\textit{Example 1.}

\begin{lemma}\label{lem::casesWhereWorks}
Consider the general branching processes with Dirichlet weights $\bfal$ described above. Assume that
$$
\alpha_1 = \cdots = \alpha_n = \frac{k}{n-1}, \quad k \in \{1, 2,3,4\}, \quad n\geq 2.
$$
Then the Fourier transform $f(w)$ of $\nu_\gamma$ is analytic and $\neq 1$ when $\Re w \in (0, \gamma]$. In particular,
$$
G(t) - \frac{\mu_2}{2 \mu_1^2} = O(e^{-\gamma t}),
$$
as $t \to \infty$.
\end{lemma}

\begin{proof}
Letting $\alpha = k/(n-1)$ a direct calculation gives
\[ \psi(\theta) = \prod_{i=1}^k \frac{\alpha +i}{\theta+\alpha +i-1}, \;\; \theta>-\alpha. \]
There is always a solution to $\psi(\theta)=1$ at $\theta=1$ and all we require is that the other solutions
are less than 0 to establish, via \eqref{eq::defFourierTansform}, that $f(w)$ is analytic and $\neq 1$ on
$\Re w \in (0, (1 +\alpha) \gamma)$.

For $k = 1$, the only solution to $\psi (\theta) = 1$ is $\theta = 1$.

For $k = 2$,  the other solution to $\psi (\theta) = 1$ is given by $\theta=-2(\alpha+1)$.

For $k=3$, the other solutions to $\psi (\theta) = 1$ are
\[ \theta = -\frac{3\alpha+4}{2} \pm \frac12\sqrt{-3\alpha^2-12\alpha-8}. \]
and $k=4$ has solutions to $\psi (\theta) = 1$ at
\[ \theta=-2\alpha -4, -\frac32-\alpha \pm \frac12 \sqrt{-4\alpha^2-20\alpha-15}. \]
Thus the real parts of all the solutions are less than zero and we have the required analyticity.

The rest of the statement follows from Theorem \ref{thm::stoneExpConvergenceRate}.
\end{proof}

Analytic solutions for the solutions to the equation $\psi(\theta)=1$
do not appear to be available for larger values of $k$.\\[.1in]
 \textit{Example 2.}
\vspace*{.1in}

\begin{figure}
	\includegraphics[width = 4 cm]{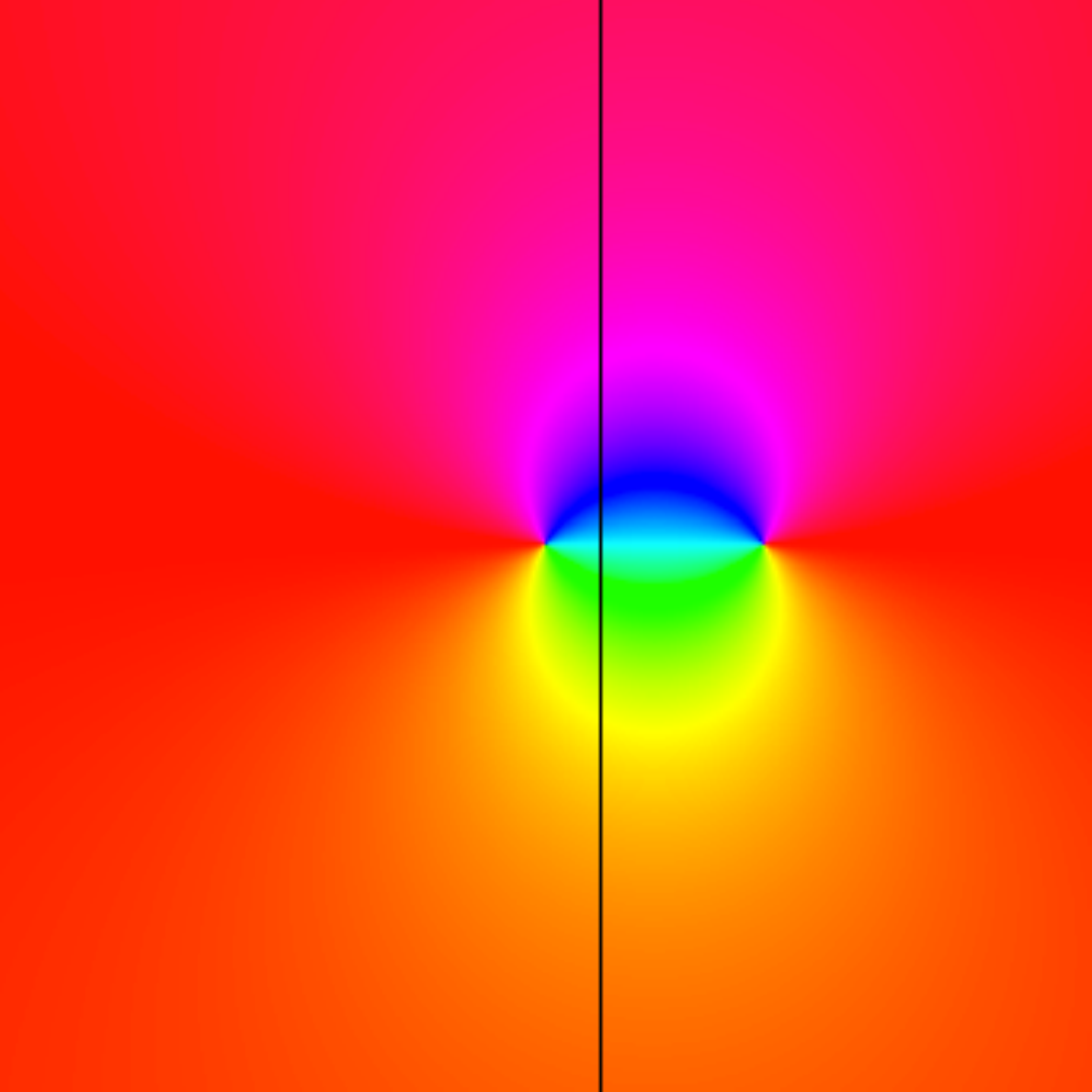}\hspace{1 cm}
	\includegraphics[width = 4 cm]{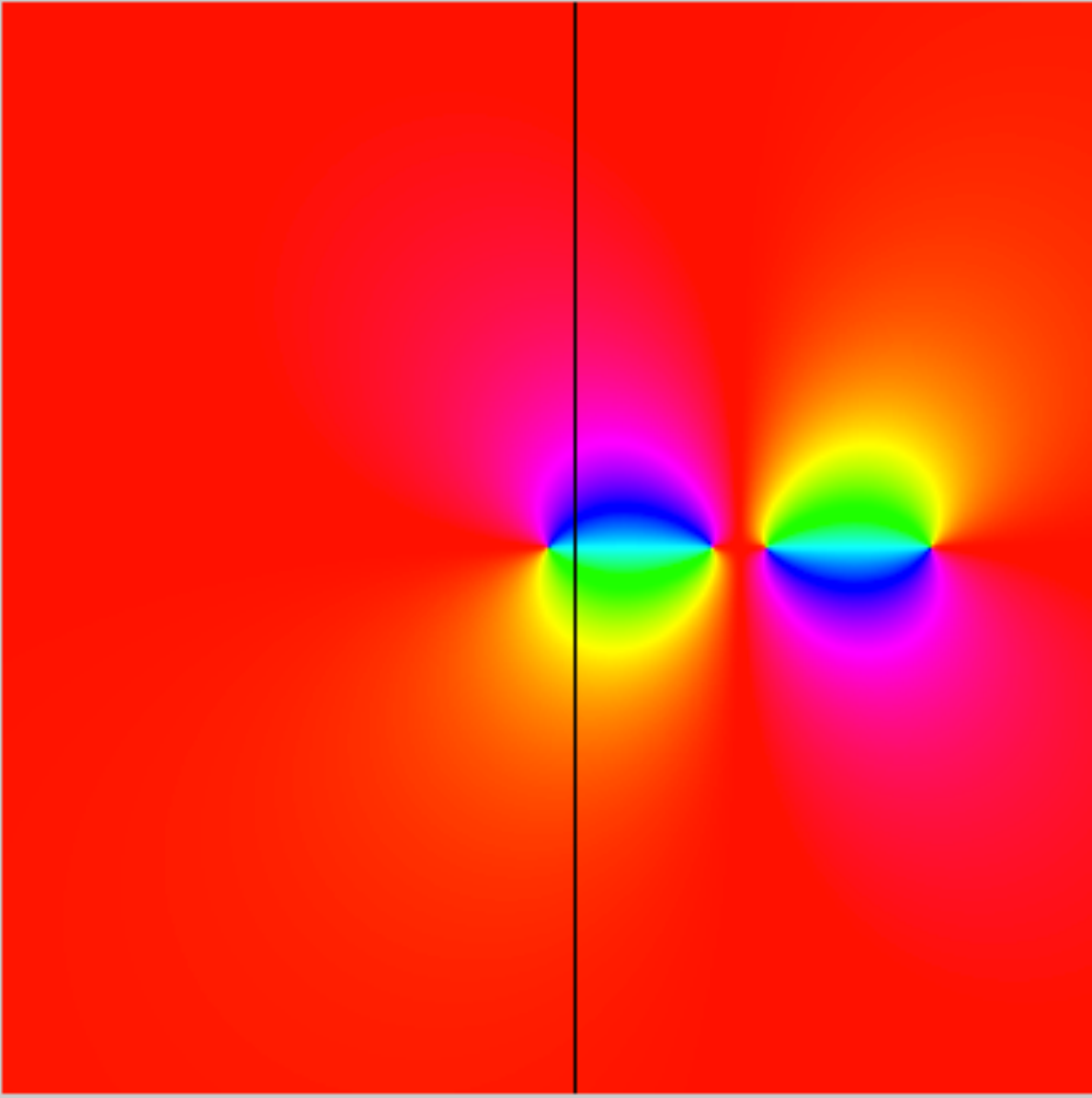}\hspace{1 cm}
	\includegraphics[width = 4 cm]{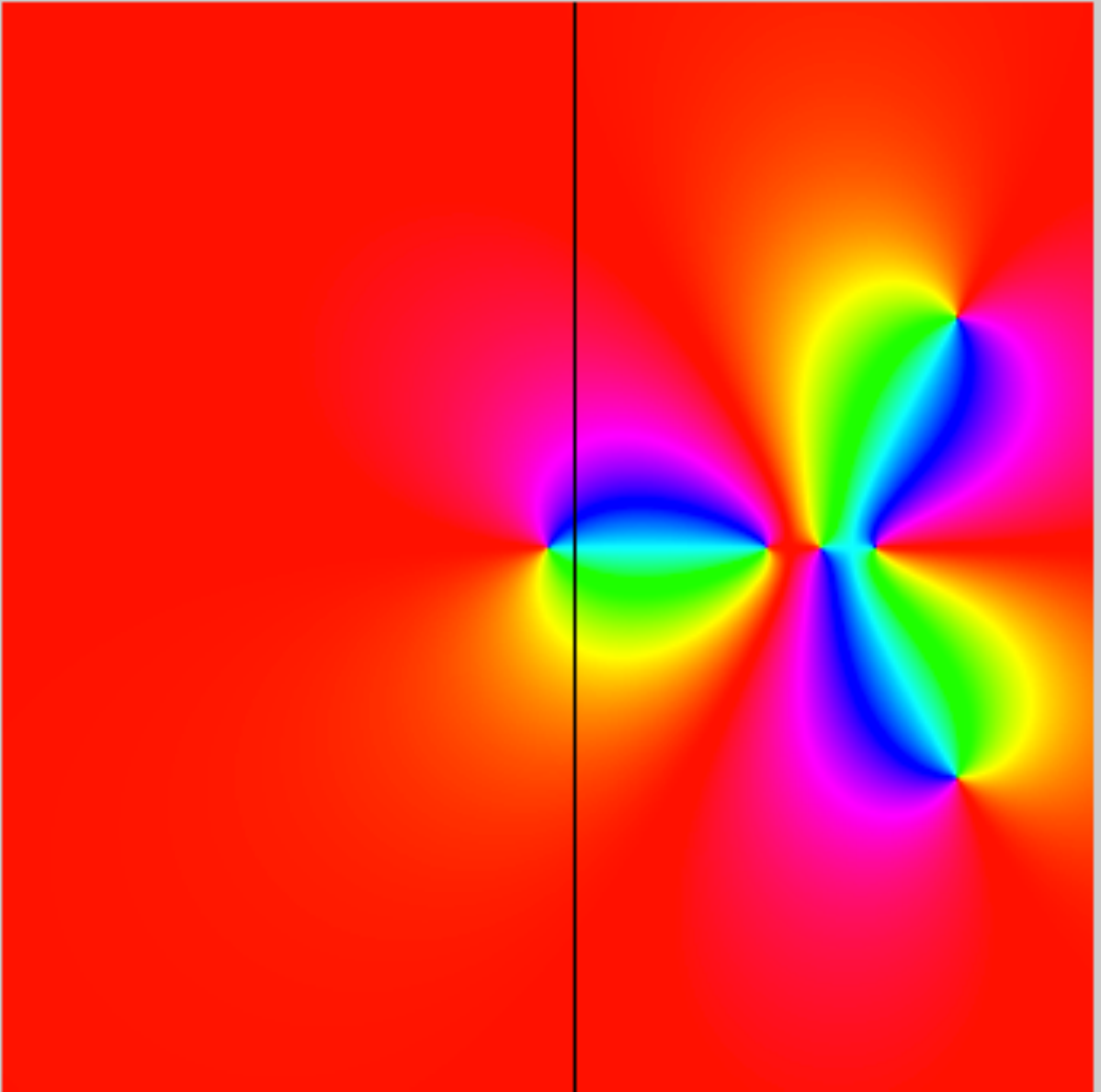}\\\vspace{1 cm}
	
	\includegraphics[width = 4 cm]{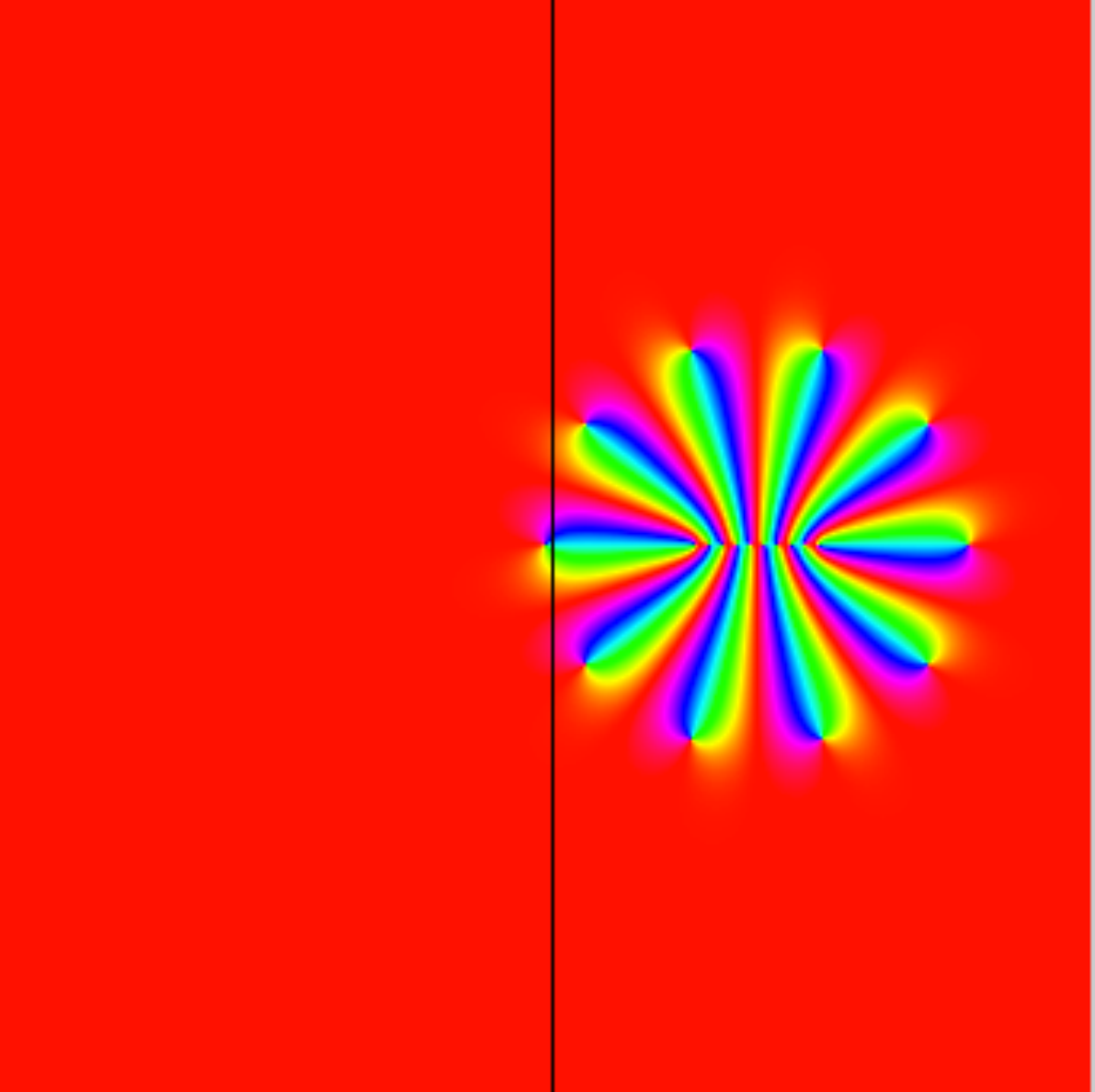}\hspace{1 cm}
	\includegraphics[width = 4 cm]{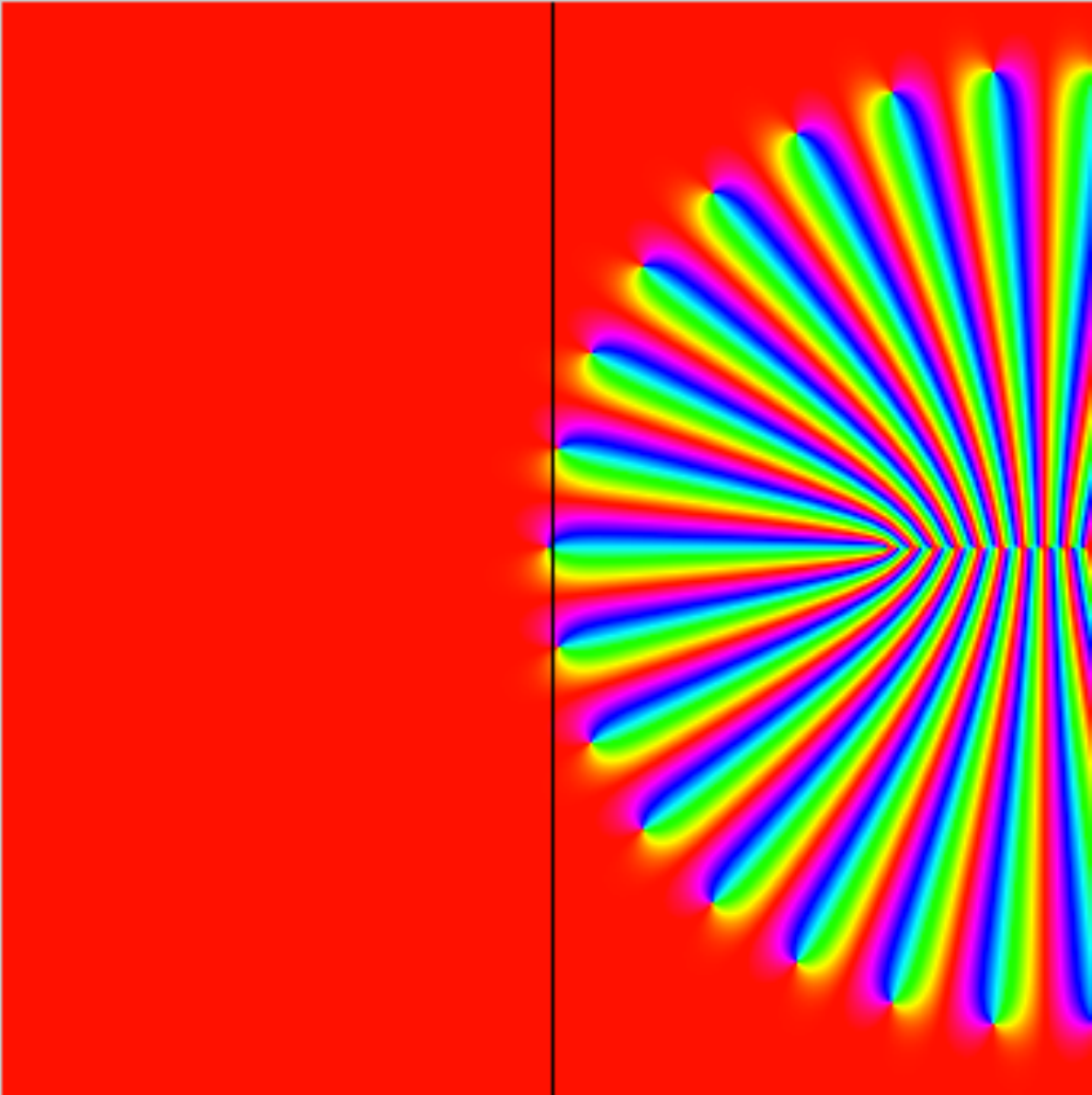}\hspace{1 cm}
	\includegraphics[width = 4 cm]{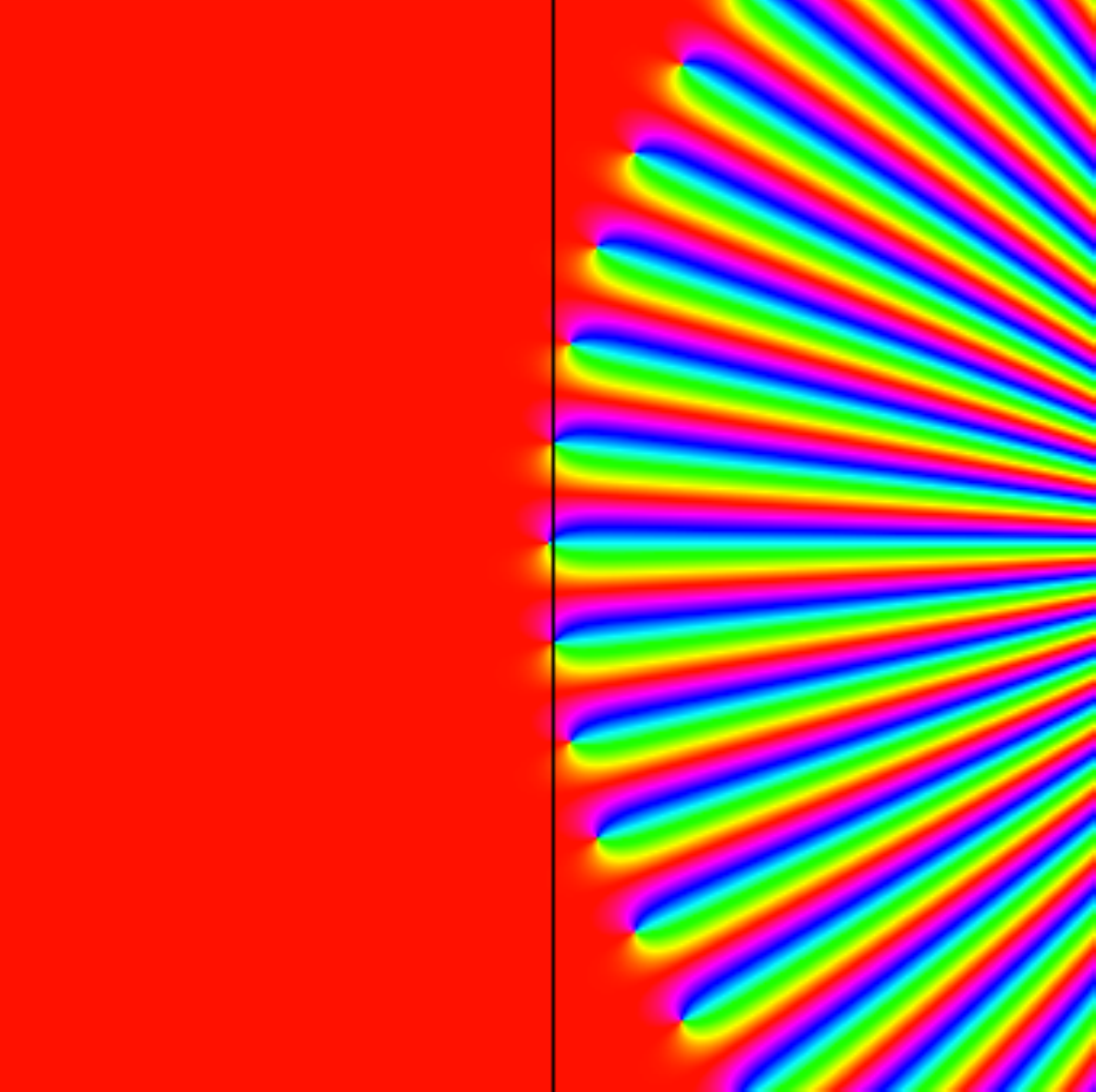}
	\caption{Phase plots of $1- f(\gamma w)$ for $\alpha= 1$, 2, 3, 10, 30 and 60. The black line indicates the set $\{z \in \C: \Re z = 1/2\}$. The regions of the plot are $\{z \in \C: \Re z \text{ and } \Im z \in [-s, s]\}$ for $s = 5$, 10, 10, 40, 50 and 50.}
	\label{fig::somePhasePlots1}
\end{figure}

Here we discuss the general branching process derived from the class of examples mentioned in the Introduction
in which the Dirichlet weights are of the form $\bfal = (\alpha, \alpha)$ with $\alpha \in \N$.
We will establish the Theorem from the Introduction.

Thanks to Lemma \ref{lem::casesWhereWorks}, we know that if $\alpha\leq 4$, as $n=2$, then the Fourier
transform $f$ of $\nu_{\gamma}$ defined in \eqref{eq::defFourierTansform} can be used to show that the rate
of convergence
in the renewal theorem is sufficiently fast for the requirements of Theorem \ref{thm::CLT}. In other words,
the applicability of Theorem \ref{thm::CLT} depends on the regularity of the characteristic $\phi$.

More generally we need to solve the equation
\[ f(\gamma(1-\theta)) = \psi(\theta) = \frac{2\Gamma(2\alpha)\Gamma(\alpha+\theta)}
{\Gamma(\alpha)\Gamma(2\alpha+\theta)} =1. \]
As $\alpha\in \N$ this is a polynomial equation and hence we seek roots of
\[ \prod_{i=0}^{\alpha-1} (\theta+\alpha+i) = \frac{2(2\alpha-1)!}{(\alpha-1)!}. \]
By letting $w=1-\theta$ the rate of convergence in the renewal theorem is given by the root of $1-f(\gamma w)$
with smallest strictly positive real part. We have computed these values numerically.

The numerical evidence shows that when $\alpha$ increases, some roots of $1- f(\gamma w)$ get close to the imaginary axis.
This phenomenon is illustrated in Figure \ref{fig::somePhasePlots1} which contains phase plots of $1-f(\gamma w)$
for different values of $\alpha$; we rescaled for convenience. To highlight this more clearly,
Figure~\ref{fig::somePhasePlots2} contains some close-ups of phase plots showing the absence or presence of
such roots of $1- f(\gamma w)$ for different values of $\alpha$. In particular, when $\alpha = 30$, the two non-zero
roots of $1-f(\gamma w)$ closest to the imaginary axis are
$$
\rho_\pm \simeq 0.9951 \pm 9.1074i;
$$
when $ \alpha = 60$, they are
$$
\rho_\pm \simeq 0.4962 \pm 9.1027i;
$$
and when $ \alpha = 80$, they are
$$
\rho_\pm \simeq 0.3718 + 9.0963i.
$$
We have in fact computed the real part of the relevant root for all values of $\alpha$ from 1 to 80 -- these are plotted
in Figure \ref{fig::somePhasePlots3}. Numerically, this establishes that $\alpha = 60$ is the smallest integer value
for which $1-f(\gamma w)$ has roots with real part $< 1/2$.

Our computations also show that for $1\leq \alpha \leq 80$ the roots of $1-f(\gamma w)$ are all simple and occur
as complex conjugate pairs except for the root at 0.

To summarise, this numerical evidence shows that the general branching process with Dirichlet weights $(\alpha, \alpha)$
admits a central limit theorem of the type described when $\alpha\leq 59$, but not when $60\leq \alpha\leq 80$.
Moreover, the monotonicity of the plot in Figure \ref{fig::somePhasePlots3} suggests that the range for which there is not
a central limit theorem extends to all $\alpha\geq 60$.

We note that we see similar results in the asymmetric case with Dirichlet weights $(\alpha_1,\alpha_2)$, $\alpha_1,
\alpha_2 \in\mathbb{N}$ with $\alpha_2\leq\alpha_1-1$. In this case the polynomial equation becomes
\[\left( \prod_{i=0}^{\alpha_2-1}(\alpha_1+\theta+i)-\frac{(\alpha_1+\alpha_2-1)!}{(\alpha_1-1)!}\right)
\prod_{i=0}^{\alpha_1-\alpha_2-1}(\alpha_2+\theta+i) = \frac{(\alpha_1+\alpha_2-1)!}{(\alpha_2-1)!}. \]
Here is a table showing for a given $\alpha_2$ the values of $\alpha_1$ below which we are in the central
limit theorem regime.

\medskip
\begin{center}
\begin{tabular}{|c|c|c|c|c|c|c|c|c|}
\hline
$\alpha_2$ & 1 & 2 & 3 & 4 & 5 & 6 & 7 & $\alpha_1-1$ \\
\hline
$\alpha_1$ & 26 & 32 & 39 & 45 & 51 & 57 & 64 & 60 \\
\hline
\end{tabular}
\end{center}

\begin{figure}
	\includegraphics[width = 4 cm]{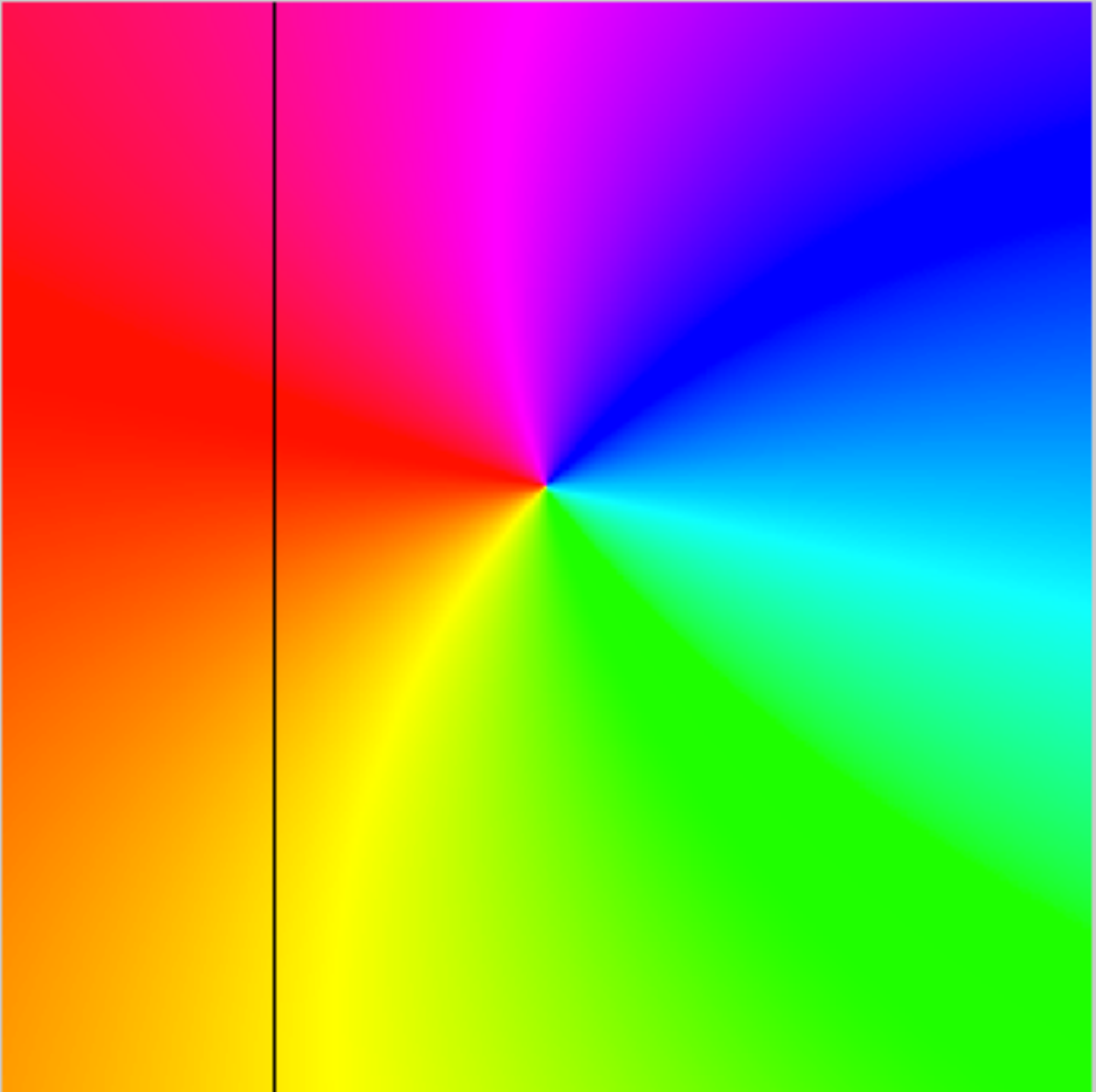}\hspace{1 cm}
	\includegraphics[width = 4 cm]{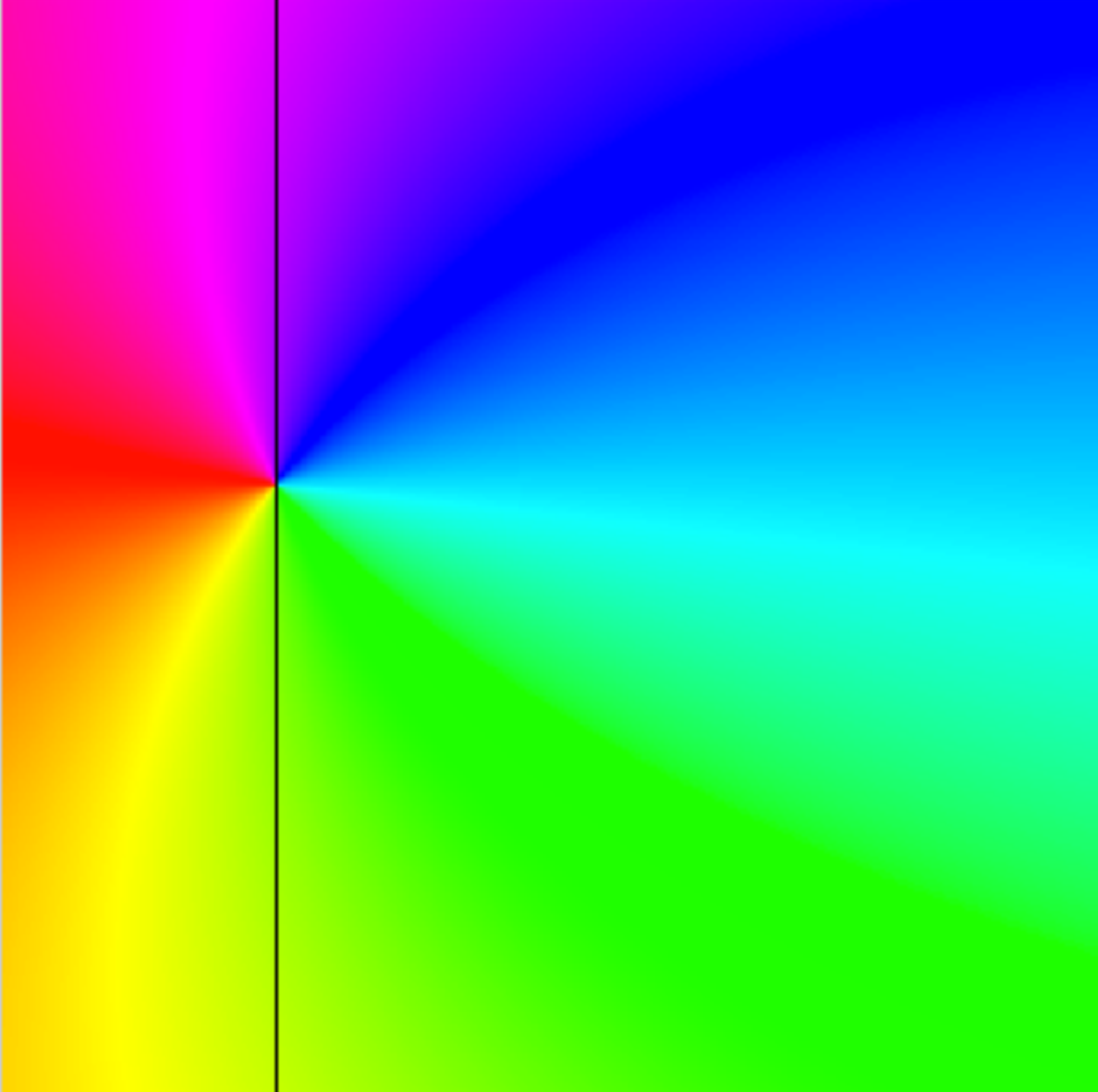}\hspace{1 cm}
	\includegraphics[width = 4 cm]{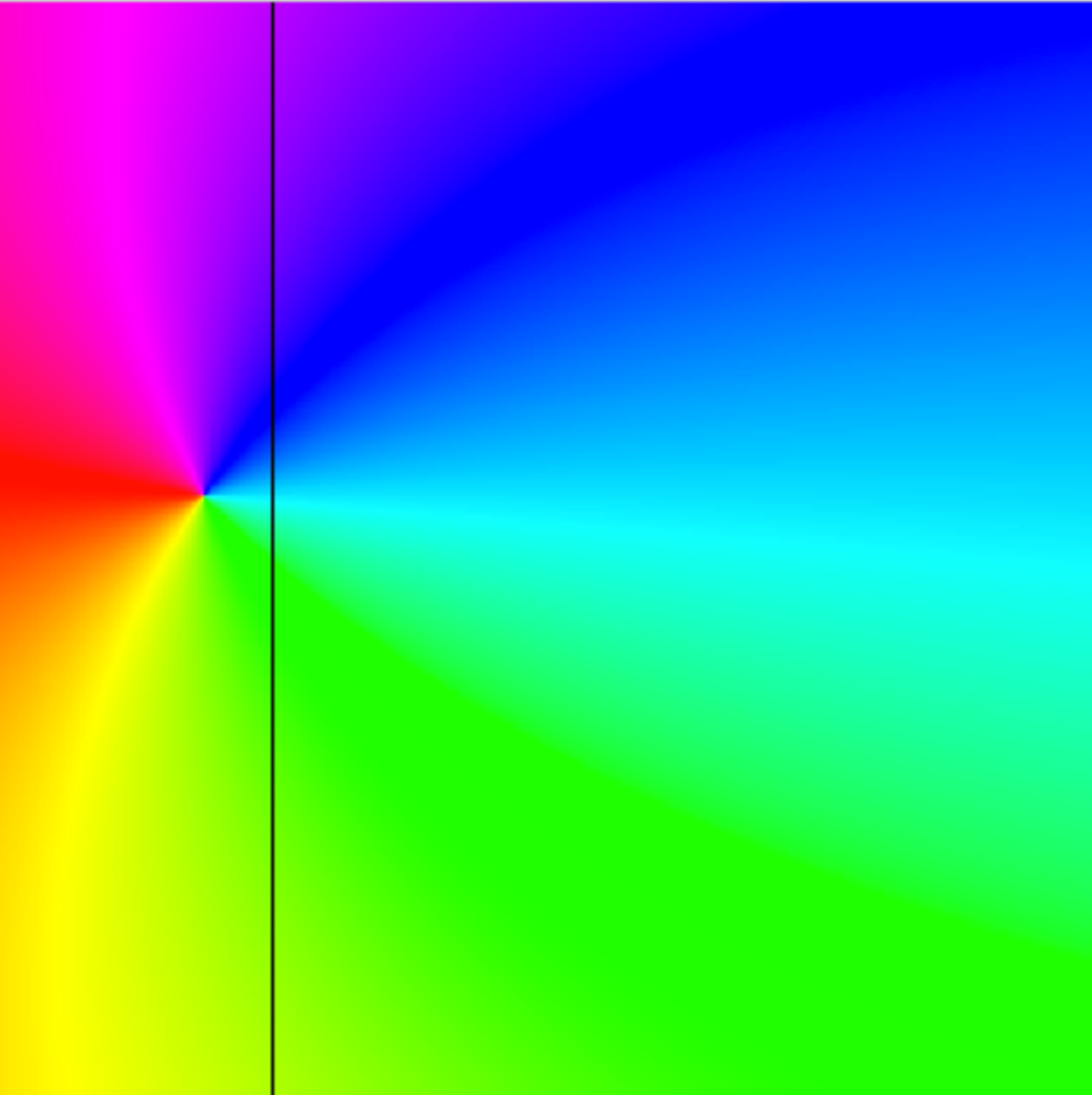}
	\caption{Phase plots of $1- f(\gamma w)$ for $\alpha= 30, 60, 80$. The black line indicates the set $\{z \in \C: \Re z = 1/2\}$. The region of the plot is $\{z \in \C: \Re z \in [0,2] \text{ and } \Im z \in [8, 10]\}$.}
	\label{fig::somePhasePlots2}
\end{figure}

\begin{figure}
	\includegraphics[width = 8 cm]{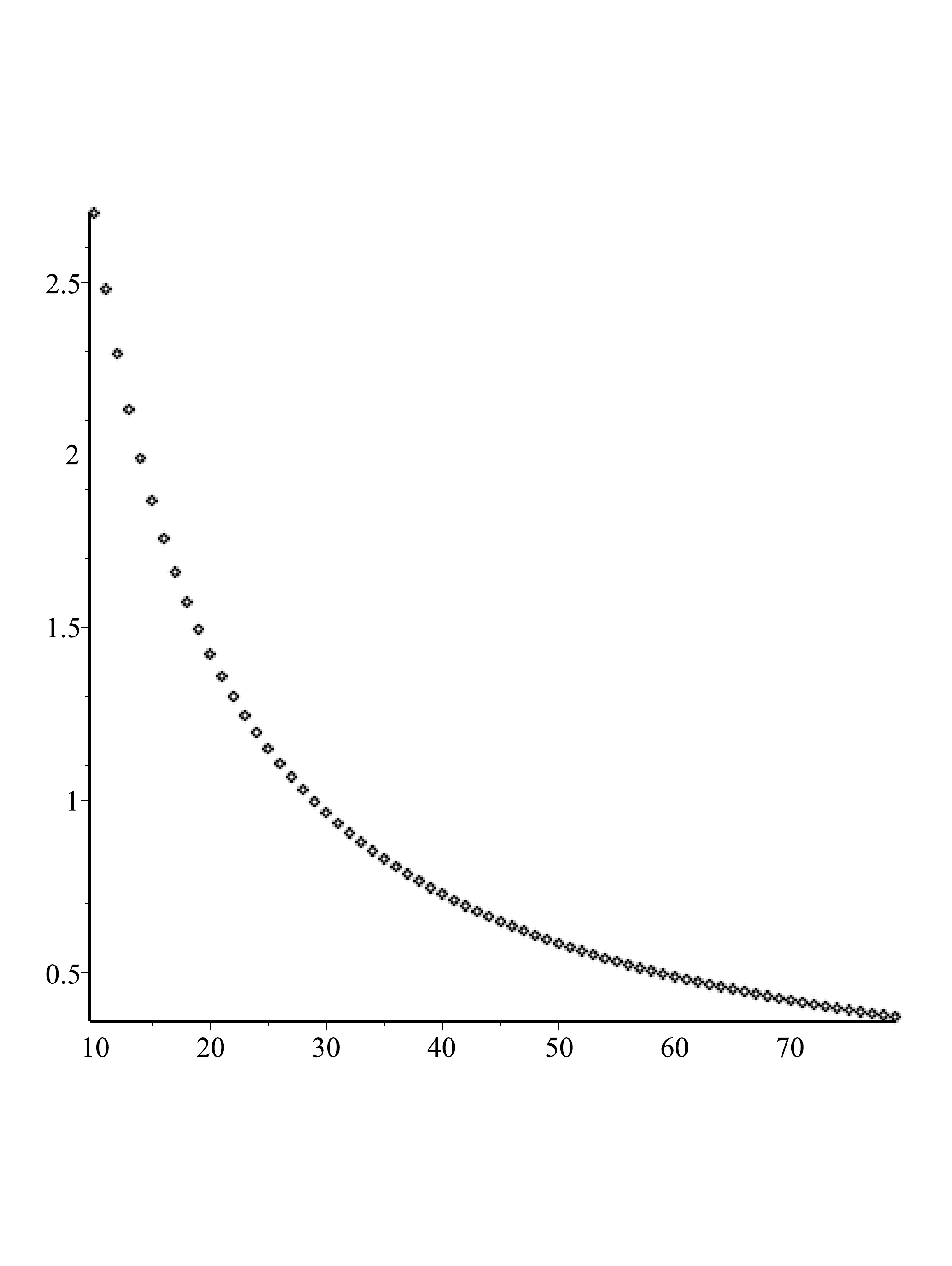}
	\caption{Plot of the real part of the root of $1- f(\gamma w)$ in $\Re z>0$ closest to the imaginary axis against $\alpha$.}
	\label{fig::somePhasePlots3}
\end{figure}

\subsection{Applications to random self-similar strings}

For the range of examples considered in Example 1 of Section~3, thanks to Lemma \ref{lem::casesWhereWorks},
we know that the Cantor set in Figure \ref{fig::cantorString} satisfies Assumption \ref{ass::convergenceRate} and
so, by Theorem \ref{thm::spectralAsympString}, the corresponding Cantor string satisfies a spectral central
limit theorem.

We now return to the second example of Section~3, which was also discussed in the Introduction.
Figure~\ref{fig::someCantorStrings} contains some pictures of statistically self-similar Cantor sets with Dirichlet
weights $(\alpha, \alpha)$ discussed in Subsection \ref{subsec::numericalExample}. The figure illustrates the fact
that the geometry of
the Cantor set becomes more rigid as $\alpha$ increases, because the corresponding Dirichlet distribution becomes
more concentrated.

\begin{proof}[Proof of Theorem~\ref{thm:ex1}]
The numerical evidence discussed in Subsection~\ref{subsec::numericalExample}
shows that Assumption~\ref{ass::convergenceRate} is satisfied for integers $\alpha \leq 59$. Thus, by Theorem~\ref{thm::spectralAsympString}, we have established parts \textit{(i)} and \textit{(ii)}
of the theorem.

For \textit{(iii)}, we start by noting if $S:=1-T^{1/\gamma}-(1-T)^{1/\gamma}$, where $T$ is a $[0,1]$-valued random variable with density
\[\frac{\Gamma(2\alpha)}{\Gamma(\alpha)^2}x^{\alpha-1}(1-x)^{\alpha-1},\]
and $\tilde{S}:=S/\pi$, then the explicit form of $\phi(t)$ yields the following distributional equality:
\[\phi(t)=\tilde{S}e^t-\lfloor \tilde{S}e^t\rfloor.\]
This is clearly bounded above by 1 for all $t\in\mathbb{R}$, and moreover, we recall from (\ref{eq::technicalBounds1}) that $\phi(t)\leq e^t$ for $t\leq0$. Taking expectations, the same is true of $\mathbf{E}\phi(t)$. Such an observation, together with the asymptotic behaviour of the renewal function (as stated at (\ref{renewal})), readily allows us to apply the double-sided renewal theorem of \cite[Theorem 5]{Karlin1955} to deduce that
\[z^\phi(t)=\int_{0}^\infty u^\phi(t-y)H(dy)\rightarrow \mu_1^{-1} \int_{-\infty}^\infty u^\phi(y)dy=:z^\phi(\infty).\]
Thus we can apply Lemma~\ref{lem::convergenceGtoZ} to obtain that
\[ z^\phi(t)-z^\phi(\infty) =\int_0^{\infty} u^\phi(t-y) G(dy)- \frac{1}{\mu_1} \int_0^{\infty} u^\phi(t+y)dy .\]
Using the bounds from \eqref{eq::technicalBounds1} again, it is straightforward to see that the second term is of
order $e^{-\gamma t}$. We now examine the first term. Using \eqref{eq::fromOfG}, we see
\begin{eqnarray*}
\int_{0}^{\infty} u^\phi(t-y) G(dy) & = & \int_{0}^t e^{-\gamma(t-y)} \mathbf{E}\phi(t-y) G(dy) \\
& = & \sum_{\Re \rho_i>0} \int_0^t e^{-\gamma(t-y)} \mathbf{E}\phi(t-y) \tilde{Q}_i(y) e^{-\rho_i y} dy.\\
\end{eqnarray*}
Define $\beta_1:=\gamma^{-1} \min_{\Re \rho_i>0}\Re \rho_i$, which by our numerical study in Example 2 of
Section~\ref{subsec::numericalExample} we know satisfies $\beta_1\in (0,1/2)$ (for $60 \leq\alpha \leq 80$). Then
\begin{eqnarray*}
{\left|\sum_{\Re \rho_i>\beta_1} \int_0^t e^{-\gamma(t-y)} \mathbf{E}\phi(t-y) \tilde{Q}_i(y) e^{-\rho_i y} dy\right|}
&\leq & c_1\sum_{\Re \rho_i>\beta_1}(1+t^{m_i-1}) e^{-\gamma t}\int_0^t e^{(\gamma-\Re {\rho_i})y}dy\\
&\leq & c_2\sum_{\Re \rho_i>\beta_1}(1+t^{m_i-1}) e^{-\gamma t} \left(1+e^{(\gamma-\Re {\rho_i})t}\right)\\
&=& o(e^{-\beta_1 t}).
\end{eqnarray*}
Without loss of generality we label the remaining pair of terms with $\rho_\pm=\gamma(\beta_1\pm i \beta_2)$, and we have that,
as all the roots are simple and come in conjugate pairs (again, for $60 \leq\alpha \leq 80$), by the remarks after Lemma~\ref{lem:simpleroots},
$\tilde{Q}_{\pm}(y)=ce^{\pm i \tilde{c}}$ for some $c,\tilde{c}$ with $c>0$. Hence
\begin{eqnarray*}
\sum_{\Re \rho_i=\beta_1} \int_0^t e^{-\gamma(t-y)} \mathbf{E}\phi(t-y) \tilde{Q}_i(y) e^{-\rho y} dy
&=&\sum_{\pm} c e^{\pm i \tilde{c}} e^{-\gamma t} \int_0^t e^{\gamma(1-\rho_{\pm})y} \mathbf{E}\phi(t-y) dy \\
&=& \sum_{\pm} c e^{\pm i \tilde{c}} e^{-\rho_{\pm}t} \int_0^t e^{-\gamma(1-\rho_{\pm})y} \mathbf{E}\phi(y) dy.
\end{eqnarray*}
As $1-\beta_1>0$ and $\phi(y)$ is a bounded function, the integrals in the above expression converge, as $t\to\infty$, to complex constants $R e^{\pm i\theta}:=\int_0^\infty e^{-\gamma(1-\rho_{\pm})y} \mathbf{E}\phi(y) dy$. It follows that
\[ z^\phi(t)-z^\phi(\infty) = 2R c \cos(\gamma \beta_2 t-\theta-\tilde{c})e^{-\gamma \beta_1 t}+ o(e^{-\gamma \beta_1 t}).\]
Now, if we suppose that $R > 0$, then the reasoning in Remark~\ref{rmk::sharpSpeedConvergence} indicates that the Cantor string does not satisfy a spectral central limit theorem for values of $\alpha\in\{60,\dots,80\}$ (recall that we have checked numerically that $\beta_1<\gamma/2$ and also $c>0$ for $\alpha$ in this range). Moreover, splitting the process as in \eqref{eq::splitApplicationCLTDirichlet} but without scaling, then taking expectations,
we can write
\begin{eqnarray*}
 \mathbf{E} Z^{\phi}(t) &=& e^{\gamma t} z^{\phi}(\infty) + e^{\gamma t}\left(z^{\phi}(t)-z^{\phi}(\infty)\right) \\
 & = & e^{\gamma t} z^{\phi}(\infty) +
 2Rc \cos(\gamma\beta_2 t-\tilde{\theta})e^{\gamma(1-\beta_1) t}+ o(e^{\gamma(1-\beta_1)t}).
\end{eqnarray*}
Rewriting in terms of the counting function we have the result for the mean counting function with $\eta(\alpha) = 1-\beta_1$
the required root of the polynomial appearing in the Theorem.

Thus to complete the proof of (iii) it remains to check that $R>0$. We will do this numerically for $\alpha\in\{60,\dots,80\}$. First, observe that for $a\in\mathbb{C}$ with $\mathfrak{R}a\in(0,1)$,
\begin{eqnarray*}
I&:=& \int_{-\infty}^\infty e^{-at}\mathbf{E}\phi(t) dt\\
&=&\mathbf{E}\int_{-\infty}^\infty e^{-at}\left(\tilde{S}e^t-\lfloor \tilde{S}e^t\rfloor\right) dt\\
&=&\mathbf{E}\sum_{n=0}^\infty\int_{\ln(n/\tilde{S})}^{\ln((n+1)/\tilde{S})} e^{-at}\left(\tilde{S}e^t-n\right)dt\\
&=&\mathbf{E}\tilde{S}^a\left(\frac{1}{1-a}+\sum_{n=1}^\infty{n^{1-a}}\left(\frac{(1+n^{-1})^{1-a}}{1-a}+\frac{(1+n^{-1})^{-a}}{a}-\frac{1}{a(1-a)}\right)\right)\\
&=&\mathbf{E}\tilde{S}^a\left(\sum_{n=0}^\infty a_n\right),
\end{eqnarray*}
where $a_0:=(1-a)^{-1}$ and, for $n\geq 1$,
\[a_n:=\frac{n^{1-a}}{a(1-a)}\left((1+n^{-1})^{-a}(1+an^{-1})-1\right).\]
Some elementary complex analysis yields
\[\left|(1+n^{-1})^{-a}-1+an^{-1}-\frac{a(a+1)}{2}n^{-2}\right|\leq 16 Mn^{-3},\qquad \forall n\geq 4,\]
where
\[M:=\max_{|z|=\frac12}|(1+z)^{-a}|\leq 2^{\mathfrak{R}a}e^{\frac{\pi}6|\mathfrak{I}a|}.\]
Hence, if $n\geq 4$, then
\begin{eqnarray*}
\lefteqn{\left|a_n-\frac{n^{-1-\mathfrak{R}a}}{2}\right|}\\
&\leq& \frac{n^{1-\mathfrak{R}a}}{|a(1-a)|}\left(\left|(1-an^{-1}+\frac{a(a+1)}{2}n^{-2})(1+an^{-1})-1-\frac{a(1-a)}{2n^2}\right|+
16Mn^{-3}|1+an^{-1}|\right)\\
&=&\frac{n^{1-\mathfrak{R}a}}{|a(1-a)|}\left(\left|\frac{a^2(a+1)}{2n^3}\right|+
16Mn^{-3}|1+an^{-1}|\right)\\
&\leq &{n^{-2-\mathfrak{R}a}}f(a),
\end{eqnarray*}
where
\[f(a):=\frac{1}{|a(1-a)|}\left(\frac{|a^2(a+1)|}2+ 2^{4+\mathfrak{R}a}e^{\frac{\pi}6|\mathfrak{I}a|}|1+a|\right).\]
Now,
\[\int_{0}^\infty e^{-at}\mathbf{E}\phi(t) dt=I-\int_{-\infty}^0e^{-at}\mathbf{E}\phi(t) dt=I-\int_{-\infty}^0e^{-at}\mathbf{E}\tilde{S}e^t dt=I-a_0\mathbf{E}\tilde{S}.\]
So, setting $a=\gamma(1-\rho_{\pm})$, we obtain that for $N\geq 3$,
\begin{eqnarray*}
\lefteqn{\left|Re^{\pm i\theta}-\mathbf{E}\tilde{S}^{a}\left(\sum_{n=1}^N a_n
+\tfrac12 \zeta(1+a)-\tfrac12 \sum_{n=1}^N n^{-(1+a)}\right)-a_0\left(\mathbf{E}\tilde{S}^{a}-\mathbf{E}\tilde{S}\right)\right|}\\
&\leq& \mathbf{E}\tilde{S}^{\mathfrak{R}a}\sum_{n=N+1}^\infty \left|a_n-\frac{n^{-1-\mathfrak{R}a}}{2}\right|\leq
\mathbf{E}\tilde{S}^{\mathfrak{R}a}\sum_{n=N+1}^\infty
{n^{-2-\mathfrak{R}a}}f(a)\leq
\mathbf{E}\tilde{S}^{\mathfrak{R}a}\frac{N^{-1-\mathfrak{R}a}f(a)}{1+\mathfrak{R}a},
\end{eqnarray*}
where $\zeta(x)=\sum_{n=1}^\infty n^{-x}$ is the usual zeta function. In particular, the above inequality allows us to compute an estimate for $Re^{\pm i\theta}$ whose error is no greater than the upper bound. For values of $\alpha\in\{60,\dots,80\}$ and $\gamma=\tfrac12$, our computations establish that $R>0$, as desired. For example, with this choice of $\gamma$, we find that if $\alpha=60$, then $R\simeq0.09703$, and if $\alpha=80$, then $R\simeq0.1056$. Note that values of $\rho_\pm$ and $R$ for all values of $\alpha\in\{60,\dots,80\}$ are presented in the Appendix below.
\end{proof}

\begin{figure}
\includegraphics[width = 4 cm]{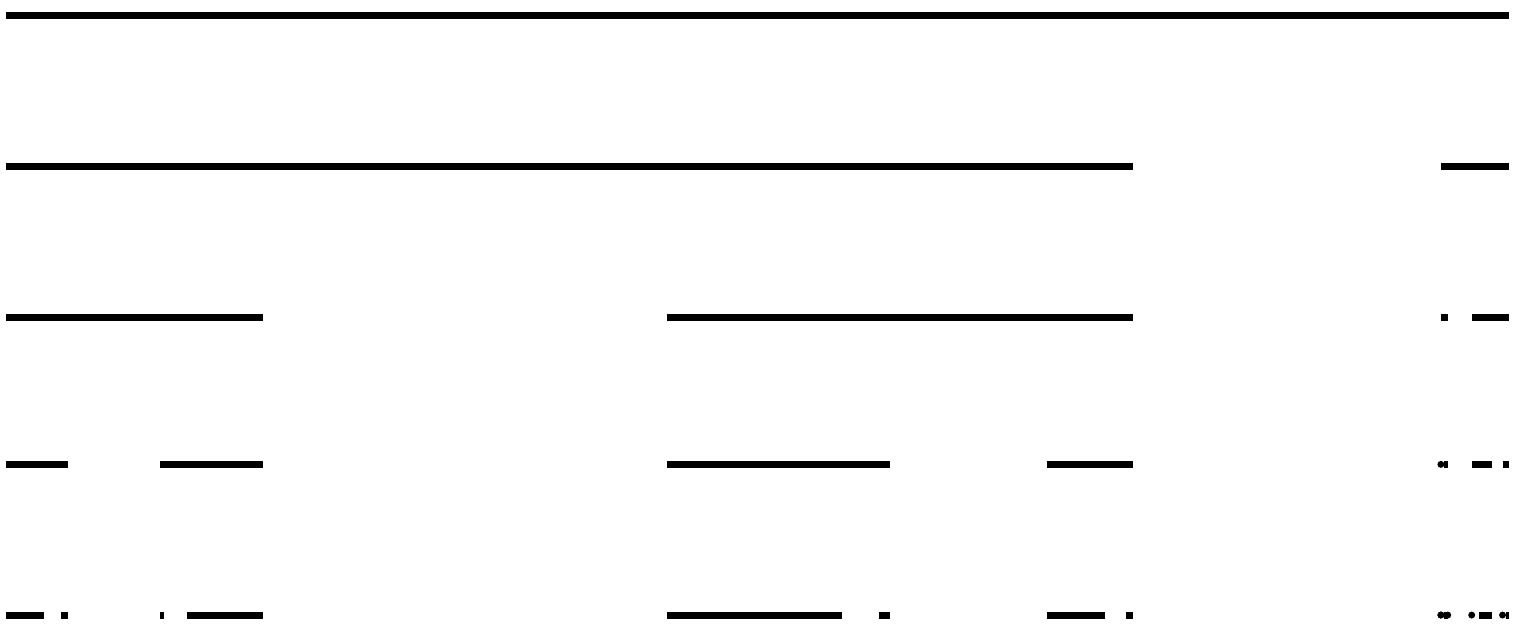}\hspace{1 cm}
\includegraphics[width = 4 cm]{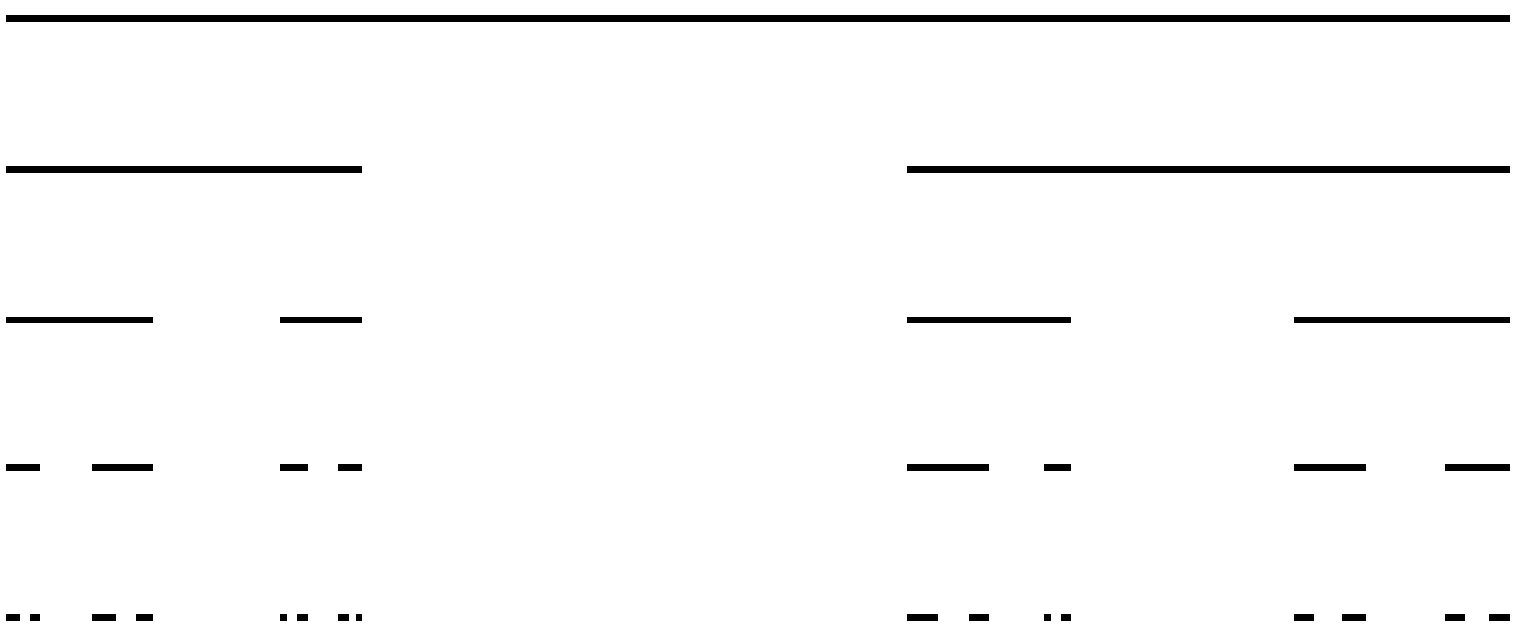}\hspace{1 cm}
\includegraphics[width = 4 cm]{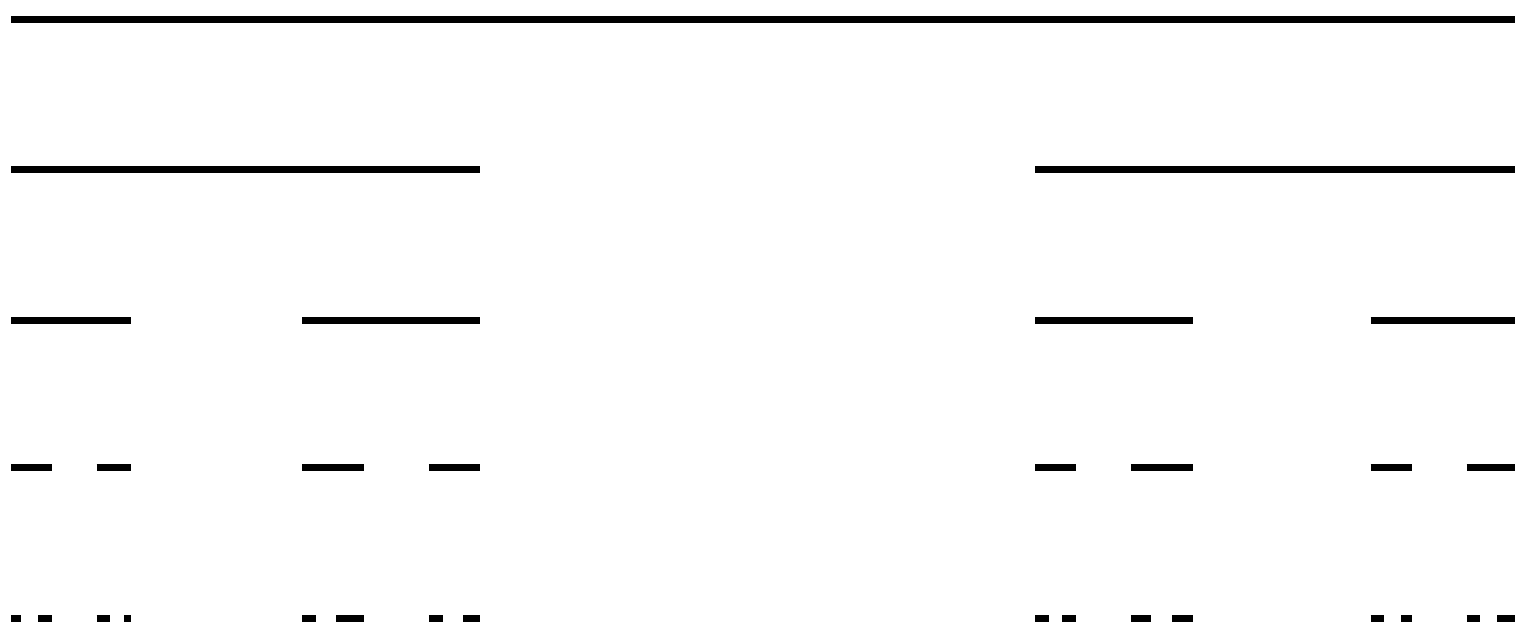}
\caption{Statistically self-similar Cantor strings for the distribution $\Dirichlet(\alpha, \alpha)$, with
$\alpha = 1$, 30 and 80 and $\gamma = 0.6$.}
\label{fig::someCantorStrings}	
\end{figure}

\section{Spectral central limit theorem for the Brownian CRT}

\subsection{Brownian CRT definition and main result}

Building on the investigations into the spectrum of the Brownian continuum random tree (CRT) undertaken in \cite{CH2008, CH2010}, in this section we apply Theorem \ref{thm::CLT} to deduce a central limit theorem for the Brownian CRT's eigenvalue counting function. The starting point for doing this is the characterisation of the Brownian CRT as a random self-similar fractal tree with $\Dirichlet(1/2, 1/2, 1/2)$ weights. (This was shown in \cite{CH2008} using a decomposition first derived in \cite{Aldous5}.)

\begin{figure}
\includegraphics[width = 10cm]{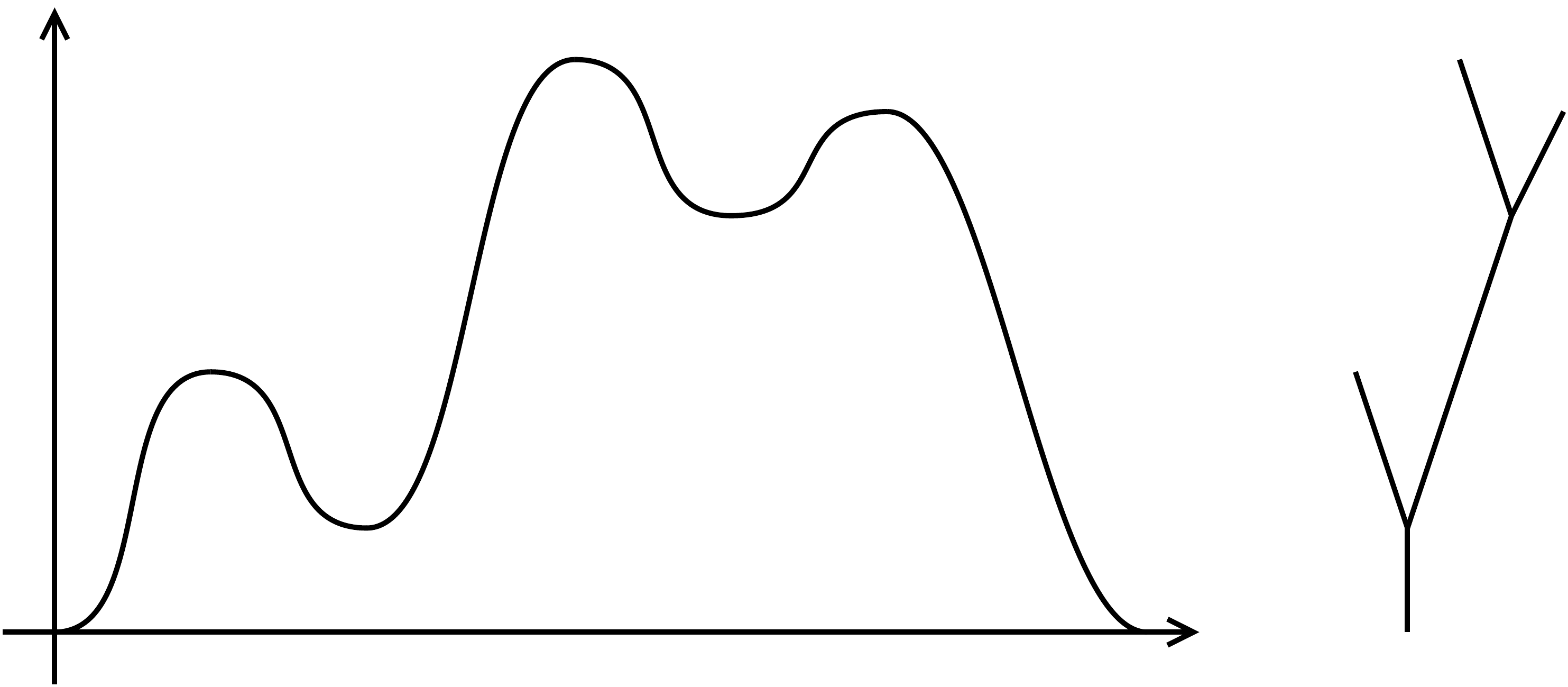}
\caption{An excursion and associated real tree.}
\label{fig::treeexc}
\end{figure}

To introduce the Brownian CRT precisely, it will be most convenient to use the now well-known connection between real trees and excursions. In particular, a function $f$ is said to be an excursion of length $\ell\in(0,\infty)$ if it belongs to $C(\mathbb{R}_+,\mathbb{R}_+)$ and also satisfies $f(x)>0$ if and only if $x\in(0,\ell)$. Given such a function, define a distance on $[0,\ell]$ by setting $d_f(x,y):=f(x)+f(y)-2\inf\{f(r):r\in[x\wedge y,x\vee y]\}$, and let $\sim_f$ be the equivalence relation arrived at by supposing $x\sim_f y$ if and only if $d_f(x,y)=0$. Subsequently, if $\mathcal{T}_f:=[0,\ell]/\sim_f$ and $d_{\mathcal{T}_f}$ is the corresponding quotient metric, it is possible to check that $(\mathcal{T}_f,d_{\mathcal{T}_f})$ is a real tree (see \cite[Definition 2.1]{LegallDuquesne} for the definition of a real tree, \cite[Theorem 2.1]{LegallDuquesne} for a proof of this fact, and Figure~\ref{fig::treeexc} for a pictorial example). Applying this construction, one may define the Brownian CRT to be the random real tree $\mathcal{T}=(\mathcal{T},d_{\mathcal{T}}):=(\mathcal{T}_{2e},d_{\mathcal{T}_{2e}})$, where $e$ is simply the Brownian excursion normalised to have unit length (see \cite[Corollary 22]{Aldous3}).

For $\mathbf{P}$-a.e.\ realisation of $\mathcal{T}$, it is possible to define naturally an associated measure and Dirichlet form as follows. Firstly, the canonical measure on $\mathcal{T}$, which will be denoted by $\mu_\mathcal{T}$, is obtained by pushing-forward Lebesgue measure on $[0,1]$ by the quotient map onto $\mathcal{T}$. This procedure yields a non-atomic Borel probability measure of full support, $\mathbf{P}$-a.s. Secondly, as a consequence of \cite[Theorem 5.4]{Kigamidendrite}, it is possible to build a local, regular, conservative Dirichlet form
$(\mathcal{E}_\mathcal{T},\mathcal{F}_\mathcal{T})$ on $L^2(\mathcal{T},\mu_\mathcal{T})$, which is related to the metric $d_\mathcal{T}$ through, for every $x\neq y\in \mathcal{T}$,
\[d_\mathcal{T}(x,y)^{-1}=\inf\{\mathcal{E}_\mathcal{T}(f,f):\:f\in\mathcal{F}_\mathcal{T},\:f(x)=0,\:f(y)=1\}.\]
The eigenvalues of the triple $(\mathcal{E}_\mathcal{T},\mathcal{F}_\mathcal{T}, \mu_\mathcal{T})$ are defined to be the numbers $\lambda$ which satisfy
\[\mathcal{E}_\mathcal{T}(f,g)=\lambda\int_\mathcal{T}fgd\mu_\mathcal{T},\hspace{20pt}\forall
g\in\mathcal{F}_\mathcal{T}\]
for some eigenfunction
$f\in\mathcal{F}_\mathcal{T}$. The corresponding eigenvalue counting function,
$N_\mathcal{T}$,  is obtained by setting
\[N_\mathcal{T}(\lambda):=\#\{\mbox{eigenvalues of
}(\mathcal{E}_\mathcal{T},\mathcal{F}_\mathcal{T}, \mu_\mathcal{T})\leq\lambda\},\]
and it is this function that will be of interest here. We note that it was checked in \cite[Section 6]{CH2008} that $N_\mathcal{T}$ is well-defined and finite for any $\lambda\in\mathbb{R}$, $\mathbf{P}$-a.s. Moreover, from \cite[Theorem 2]{CH2008} and \cite[Theorem 1.1 and Remark 1.2]{CH2010}, we know that there exists a deterministic constant $C_0\in (0,\infty)$ such that, as $\lambda\rightarrow\infty$,
\begin{equation}\label{meanconv}
\bE N_\mathcal{T}(\lambda)=C_0\lambda^{2/3}+O(1),
\end{equation}
and also, $\bP$-a.s.,
\begin{equation}\label{asconv}
\lambda^{-2/3} N_\mathcal{T}(\lambda)\rightarrow C_0.
\end{equation}
These establish second order mean behaviour, and first order almost-sure behaviour of the eigenvalue counting function. Here, we further investigate the second order distributional behaviour, applying our central limit theorem to prove the following result in particular.

\begin{theorem}\label{crtclt} There exist constants $C_0\in (0,\infty)$ and $C_1\in [0,\infty)$ such that, as $\lambda\rightarrow\infty$,
\[\frac{N_{\mathcal{T}}(\lambda)-C_0\lambda^{2/3}}{\lambda^{1/3}}\rightarrow N(0,C_1),\]
in distribution.
\end{theorem}

\begin{remark} Unfortunately we are not able to establish that the asymptotic variance $C_1$ is strictly positive, as we were in the corresponding result for fractal strings (Theorem \ref{thm::spectralAsympString}). This is due to the more complicated correlation structure of the relevant characteristics, for which we could not find suitable tools to analyse.
\end{remark}

\subsection{Self-similarity of the Brownian CRT}

As noted above, the key tool in studying the spectrum of the Brownian CRT in \cite{CH2008, CH2010} was a self-similar decomposition. We again take this recursion as our starting point, and proceed in this section to describe this in more detail. We also make the connection with the branching process framework of Section \ref{bpcltsec}.

Let $\rho\in\mathcal{T}$ be the $\sim_e$-equivalence class of $\mathcal{T}$ and $x^{(1)}$, $x^{(2)}$ be two $\mu_\mathcal{T}$-random vertices of $\mathcal{T}$. Since $\mathcal{T}$ is a real tree, there exists a unique branch-point $b^{\mathcal{T}}(\rho,x^{(1)},x^{(2)})\in\mathcal{T}$ of these three vertices. To be more precise, this is the sole element in the set $[[\rho,x^{(1)}]]\cap[[x^{(1)},x^{(2)}]]\cap[[x^{(2)},\rho]]$, where $[[x,y]]$ is the unique injective path from $x$ to $y$ in $\mathcal{T}$. Now, by the non-atomicity of $\mu_\mathcal{T}$, the vertices $\rho,x^{(1)},x^{(2)}$ are distinct almost-surely, and therefore lie in different components of $\mathcal{T}\backslash b^{\mathcal{T}}(\rho,x^{(1)},x^{(2)})$. We will label by $\mathcal{T}_1$, $\mathcal{T}_2$ and $\mathcal{T}_3$ the components containing $\rho$, $x^{(1)}$ and $x^{(2)}$, respectively. Moreover, for $i=1,2,3$, we define a metric $d_{\mathcal{T}_i}$ and probability measure $\mu_{\mathcal{T}_i}$ on
$\mathcal{T}_i$ by setting
$d_{\mathcal{T}_i}:=\Delta^{-1/2}_id_{\mathcal{T}}|_{\mathcal{T}_i\times\mathcal{T}_i}$, $\mu_{\mathcal{T}_i}(\cdot):=\Delta_i^{-1}\mu(\cdot\cap \mathcal{T}_i)$, where $\Delta_i:=\mu_\mathcal{T}(\mathcal{T}_i)$. Note that, since $\mu_\mathcal{T}$ has full-support, $\Delta_i$ is almost-surely non-zero. We also fix $\rho_1=\rho_2=\rho_3=b^{\mathcal{T}}(\rho,x^{(1)},x^{(2)})$, set $x_i^{(1)}=\rho,x^{(1)},x^{(2)}$ for $i=1,2,3$, respectively, and choose $x_i^{(2)}$ to be a $\mu_{\mathcal{T}_i}$-random vertex of $\mathcal{T}_i$ for each $i=1,2,3$. (See Figure \ref{fig::decomp}.) A minor adaptation of \cite[Theorem 2]{Aldous3} using the invariance under re-rooting of the Brownian CRT (see \cite[Section 2.7]{Aldous2}, for example) then yields the following.

\begin{figure}
\includegraphics[width = 12cm]{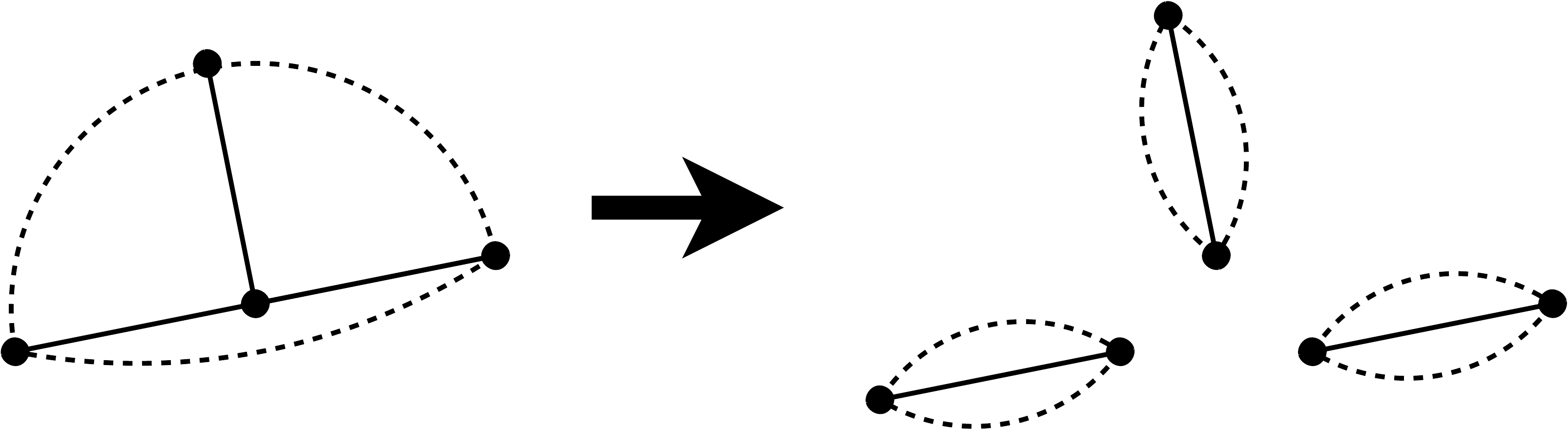}
\rput(-12.3,.3){$\rho$}
\rput(-8,1.1){$x^{(1)}$}
\rput(-10.3,3.2){$x^{(2)}$}
\rput(-8.5,2.7){$\mathcal{T}$}
\rput(-4.4,1.1){$\mathcal{T}_1$}
\rput(-1.2,1.5){$\mathcal{T}_2$}
\rput(-2.3,2.3){$\mathcal{T}_3$}
\rput(-3.2,.6){$\rho_1$}
\rput(-2.4,0.5){$\rho_2$}
\rput(-2.6,1.0){$\rho_3$}
\rput(-5.8,.3){$x^{(1)}_1$}
\rput(.2,1.0){$x^{(1)}_2$}
\rput(-3.5,3.3){$x^{(1)}_3$}
\caption{Self-similar decomposition of the continuum random tree.}
\label{fig::decomp}
\end{figure}
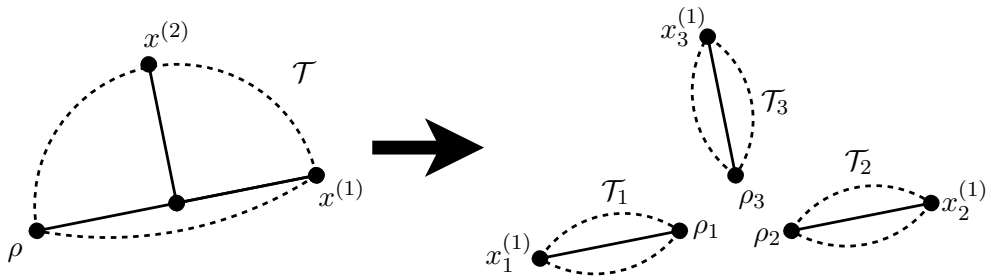

\begin{lemma}\label{split} The collections $(\mathcal{T}_i,d_{\mathcal{T}_i},\mu_{\mathcal{T}_i},\rho_i,x_i^{(1)}, x^{(2)}_i)$, $i=1,2,3$, are independent copies of $(\mathcal{T},d_{\mathcal{T}},\mu_\mathcal{T},\rho,x^{(1)},x^{(2)})$, and moreover, the entire
family of random variables is independent of $(\Delta_i)_{i=1}^3$, which has a $\Dirichlet(\tfrac{1}{2},\tfrac{1}{2},\tfrac{1}{2})$ distribution.
\end{lemma}

We will label the objects generated by applying this procedure repeatedly using a subset of the address space of sequences $I$ introduced in Section \ref{bpsubsec}. In particular, for $n\geq 0$, let $\Sigma_n:=\{1,2,3\}^n$ (using the convention that $\{1,2,3\}^0=\{\emptyset\}$), and define $\Sigma:=\cup_{m\geq
0}\Sigma_m$. For $i\in\Sigma_m, j\in\Sigma_n$, we continue to write the convolution $ij=i_1\dots i_m j_1 \dots j_n$. For $k\in\Sigma$, we denote by $|k|$ the unique integer $n$ such that $k\in\Sigma_n$. We will also write for $i\in\Sigma_m$, $i|_n=i_1\dots i_n$ for any $n\leq m$.

Returning to our inductive procedure, given $(\mathcal{T}_i,d_{\mathcal{T}_i},\mu_{\mathcal{T}_i},\rho_i,x_i^{(1)}, x^{(2)}_i)$, where $i\in \Sigma$,
we define $(\mathcal{T}_{ij},d_{\mathcal{T}_{ij}},\mu_{\mathcal{T}_{ij}},\rho_{ij},x_{ij}^{(1)}, x^{(2)}_{ij})$ and $\Delta_{ij}$, $j=1,2,3$, from $(\mathcal{T}_i,d_{\mathcal{T}_i},\mu_{\mathcal{T}_i},\rho_i,x_i^{(1)}, x^{(2)}_i)$ using the same method as
that by which $\mathcal{T}$ was decomposed above. If the  $\sigma$-algebra generated by the random variables
$(\Delta_i)_{1\leq |i|\leq n}$ is denoted  by $\mathcal{F}_n$ for each $n\in\mathbb{N}$, then Lemma \ref{split} readily yields the following corollary. As in \cite{CH2008} and \cite{CH2010}, it is this result that facilitates all that follows.

\begin{corollary} For each $n\in\mathbb{N}$,
$\{(\mathcal{T}_i,d_{\mathcal{T}_i},\mu_{\mathcal{T}_i},\rho_i,x_i^{(1)}, x^{(2)}_i)\}_{i\in  \Sigma_n}$ is an independent collection of copies of
$(\mathcal{T},d_{\mathcal{T}},\mu_{\mathcal{T}},\rho,x^{(1)}, x^{(2)})$, independent of $\mathcal{F}_n$.
\end{corollary}

To prove Theorem \ref{crtclt}, we will work with the Dirichlet eigenvalues of $(\mathcal{E}_\mathcal{T},\mathcal{F}_\mathcal{T},\mu_{\mathcal{T}})$. These are defined to be the eigenvalues of the triple $(\mathcal{E}_\mathcal{F}^D,\mathcal{F}_\mathcal{T}^D,\mu_\mathcal{T})$, where $\mathcal{E}_\mathcal{T}^D:=\mathcal{E}|_{\mathcal{F}_\mathcal{T}^D\times\mathcal{F}_\mathcal{T}^D}$ and $\mathcal{F}_\mathcal{T}^D:=\{f\in\mathcal{F}_\mathcal{T}:f(\rho)=f(x^{(1)})=0\}$. Since the corresponding eigenvalue counting function $(N^D_\mathcal{T}(\lambda))_{\lambda\in\mathbb{R}}$ satisfies
\begin{equation}\label{dnbracket}
N^D_\mathcal{T}(\lambda)\leq N_\mathcal{T}(\lambda)\leq N^D_\mathcal{T}(\lambda)+2,\hspace{20pt}\forall \lambda\in\mathbb{R},
\end{equation}
(see \cite[Lemma 19]{CH2008}), the asymptotics of $N^D_\mathcal{T}$ are indistinguishable from those of $N_\mathcal{T}$ at the level at which we are working.

We now make the connection between the eigenvalue counting function $N^D_\mathcal{T}$ on $\mathcal{T}$ and a general branching process. Suppose that, starting from the single individual $\emptyset$, each individual $i$ has three offspring, born at times $-\tfrac32\ln \Delta_{ij}$, $j=1,2,3$, after $i$ was born (so that the entire population can be indexed by the set $\Sigma$). In particular, this implies that an individual $i\in \Sigma$ has birth time $\sigma_i=-\tfrac32\ln D_i$, where $D_\emptyset:=1$ and $D_i:=\Delta_{i|1}\Delta_{i|2}\dots \Delta_{i||i|}$ for $i\in \Sigma\backslash\{\emptyset\}$. For our purposes, we do not need to define lifetimes of individuals explicitly. We do, however, define characteristics $(\phi_i)_{i\in \Sigma}$, via the formula
\begin{equation}\label{etadef}
N_i^D(e^t)=\phi_i(t)+\sum_{j=1}^3N_{ij}^D(e^t\Delta^{3/2}),
\end{equation}
where $N_i^D$ is the Dirichlet eigenvalue counting function on $(\mathcal{T}_i,d_{\mathcal{T}_i},\mu_{\mathcal{T}_i})$. Note that \cite[Lemma 19]{CH2008} implies that $\phi_i(t)\in[0,6]$ for every $t\in\mathbb{R}$, $\mathbf{P}$-a.s. Note also that the random function $\phi_i$ only depends on the progeny of $i$ (including the birth times of the offspring of $i$). Thus, we have a general branching process in the sense of Section \ref{bpsubsec}, and, in the sense of Section~\ref{exs} it has Dirichlet weights. It is easy to check that this process has Malthusian parameter equal to $\gamma=2/3$. Moreover, iterating (\ref{etadef}) (and checking that the remainder term converges to 0) allows one to deduce that the corresponding characteristic counting process
\begin{equation}\label{xchar}
Z^\phi(t)=\sum_{i\in \Sigma}\phi_i(t-\sigma_i)
\end{equation}
satisfies $Z^\phi(t)=N^D_\mathcal{T}(e^t)$ (see the proof of \cite[Lemma 3.5]{CH2010}). As before, the rescaled means
of $Z^\phi$ and $\phi$ will be written
$z^\phi(t):=e^{-\gamma t}\mathbf{E}(Z^\phi(t))$, $u^\phi(t):=e^{-\gamma t}\mathbf{E}(\phi(t))$, where we omit the
index from $\phi$ in the expectation since this is unimportant. Both of the above functions are well-defined and finite for
all $t\in\mathbb{R}$ (see \cite{CH2008}). In fact,
\begin{equation}\label{mdef}
M:=\sup_{t\in\mathbb{R}}z^\phi(t)
\end{equation}
is a finite constant (see \cite[Lemma 20]{CH2008}). Moreover, it was proved as \cite[Proposition 21]{CH2008} that
$z^\phi(t)\rightarrow z^\phi(\infty):={\int_{-\infty}^{\infty}u^\phi(t)dt}\in (0,\infty)$. (The proof that
$z^\phi(\infty)\in (0,\infty)$ was actually not included there, but this is a simple consequence of
\cite[Proposition 1.7]{Croydoncrt} and \cite[Corollary 4]{CH2008}.) We also have that, $\mathbf{P}$-a.s.,
$e^{-\gamma t}Z^\phi(t)\rightarrow z^\phi(\infty)$, see \cite[Proposition 22]{CH2008} -- as in the fractal
strings with Dirichlet weights example, the fundamental martingale is identically equal to one, and so the
limit is deterministic. Note that a simple reparameterisation of the two previous results yields the first order
parts of (\ref{meanconv}) and (\ref{asconv}).

To prove Theorem \ref{crtclt}, we introduce a rescaled centred version of the characteristic counting process. Specifically, as
 before, we set
\[\bar{Z}(t):=Z^{\bar{\zeta}}(t)=Z^\phi(t)-e^{\gamma t}z^\phi(t),\qquad\tilde{Z}(t):=e^{-\gamma t/2}\bar{Z}(t),\]
where $\bar{\zeta}$ is defined as at (\ref{eq::defCentring}). Just as (\ref{etadef}) was fundamental to demonstrating the
first order asymptotic behaviour of $N_\mathcal{T}(t)$ in the arguments of \cite{CH2008}, the recursions at
\eqref{eq::definitionTildeZ} and \eqref{eq::varianceBranchingProcess} are central to our efforts to derive the corresponding
second order behaviour via the branching process result of Theorem \ref{thm::CLT}. We note that the use of an analogous
recursion formula for providing second order bounds was already noticed in \cite{CH2010}. However, that paper was mainly
focused on the infinite variance $\alpha$-stable tree case, and did not obtain the type of detailed results that we do here for
the Brownian CRT.

\subsection{Variance convergence}

In this section, we use the renewal equation of (\ref{vrenewal}) to show that the rescaled variance $v(t):=e^{-\gamma t} \bE (\bar{Z}(t)^2)=\bE (\tilde{Z}(t)^2)$ converges as $t\rightarrow\infty$ to a finite constant. To do this, we are required to check that $v$, $r$ and $\nu_\gamma$ are suitably well-behaved, where $r$ is defined at (\ref{vdefrdef}) and $\nu_\gamma(dt):=\sum_{i=1}^3e^{-\gamma t}\bP(\sigma_i\in dt)$ -- this is the content of the next three lemmas. In the proof of the following result, we recall the function $\psi(x)=3/(1+2x)$ for $x>-{1/2}$, as introduced in \eqref{phidef}.

 \begin{lemma}\label{propsa} The function $v$ is bounded and  measurable, and $v(t)\rightarrow0$ as $t\rightarrow -\infty$.
 \end{lemma}
\begin{proof} We start by checking that $v$ is bounded for $t\geq 0$. Similarly to the proof of \cite[Lemma 5.3]{CH2010}, by appealing to \cite[Lemma 5.2]{CH2010}, it is possible to deduce that $v(t)\leq 2e^{\gamma t}(I_1+I_2+I_3)$, where
\begin{eqnarray*}
I_1 &= &\sum_{i\in\Sigma}\bE \left(e^{-2\gamma t} D_i^2(\phi_i(t-\sigma_i)- \bE(\phi_i(t-\sigma_i)|D_i))^2\right), \\
 I_2& =& \sum_{i\in\Sigma}\bE\left(D_i^2\left(\sum_{j=1}^3\Delta_{ij}z^\phi(t-\sigma_{ij})-
\bE(\Delta_{ij}z^\phi(t-\sigma_{ij})|D_i)\right)^2\right), \\
I_3& =& \sum_{i\in\Sigma}\bE \left(e^{-2\gamma t}D_i  \phi_i(t-\sigma_i)\sum_{j=1}^3 \bar{Z}_{ij}(t-\sigma_{ij}) \right).
\end{eqnarray*}

Since $\phi(t)\in[0,6]$, $I_1$ can be bounded as follows:
\begin{equation}\label{i1}
I_1 \leq 6e^{-2\gamma t}\bE \left(\sum_{i\in\Sigma}  \phi_i(t-\sigma_i)\right)= 6 e^{-2\gamma t} \bE(Z^\phi(t))=  6 e^{-\gamma t}z^\phi(t)\leq 6M e^{-\gamma t},
\end{equation}
where the first equality is a consequence of (\ref{xchar}), and $M$ is defined as at (\ref{mdef}).

For $I_2$, first observe that
\begin{equation}\label{hatmstep}
\sum_{j=1}^3\Delta_{ij}z^\phi(t-\sigma_{ij})= \sum_{j=1}^3\Delta_{ij}\hat{z}^{\phi}(t-\sigma_{ij}),
\end{equation}
where $\hat{z}^{\phi}(t):=z^\phi(t)-z^\phi(\infty)$, and the equality holds because $\sum_{j=1}^3 \Delta_j=1$. Now, by results of \cite[Section 3]{CH2010}, we have that $|\hat{z}^{\phi}(t)|\leq Ce^{-\gamma t}$ for $t\in\R$. Thus
\[{D_i\left|\sum_{j=1}^3\Delta_{ij}z^\phi(t-\sigma_{ij})-
\bE(\Delta_{ij}z^\phi(t-\sigma_{ij})|D_i)\right|}\leq Ce^{-\gamma t}\]
for some deterministic constant $C$. In particular, we have proved that
\[I_2\leq Ce^{-\gamma t} \sum_{i\in\Sigma}\bE\left(D_i\left|\sum_{j=1}^3\Delta_{ij}z^\phi(t-\sigma_{ij})-
\bE(\Delta_{ij}z^\phi(t-\sigma_{ij})|D_i)\right|\right).\]
Our next step is to show that the above sum is bounded. Writing $z^\phi(s,t):=z^\phi(s)-z^\phi(t)$, we can proceed similarly to (\ref{hatmstep}) to deduce that
\begin{eqnarray*}
\lefteqn{\bE\left|\sum_{j=1}^3 D_{ij}z^\phi(t-\sigma_{ij})-
\bE(D_{ij}z^\phi(t-\sigma_{ij})|D_i)\right|}\\
&\leq&2\bE\left(\sum_{j=1}^3 D_{ij}z^\phi(t-\sigma_{ij}, t-\sigma_i)\right).
\end{eqnarray*}
From \cite[Section 3]{CH2010}, we have for any $s\leq t$ that $z^\phi(s,t)=u^\phi(s)-u^\phi(t)-\int_s^t u^\phi(w)dw$,
and hence
\begin{eqnarray}
\lefteqn{\sum_{i\in\Sigma}\bE\left|\sum_{j=1}^3 D_{ij}z^\phi(t-\sigma_{ij})-
\bE(D_{ij}z^\phi(t-\sigma_{ij})|D_i)\right|}\nonumber\\
&\leq & 2\sum_{i\in\Sigma}\bE\left(\sum_{j=1}^3 D_{ij} u^\phi(t-\sigma_{i})\right)\label{b1}\\
&&+ 2\sum_{i\in\Sigma}\bE\left(\sum_{j=1}^3 D_{ij} u^\phi(t-\sigma_{ij})\right)\label{b2}\\
&&+  2\sum_{i\in\Sigma}\bE\left(\sum_{j=1}^3 D_{ij} \int_{t-\sigma_{ij}}^{t-\sigma_{i}} u^\phi(w)dw\right).\label{b3}
\end{eqnarray}
To bound these expressions, we will apply the following characterisation of $z^\phi(t)$:
\begin{equation}
z^\phi(t)=e^{-\gamma t} \bE\left(Z^\phi(t)\right)=e^{-\gamma t} \sum_{i\in\Sigma} \bE{(\phi_i(t-\sigma_i))}= \sum_{i\in\Sigma} \bE{(D_iu^\phi(t-\sigma_i))}.\label{mchar}
\end{equation}
Specifically, the term at (\ref{b1}) satisfies
\[2\sum_{i\in\Sigma}\bE\left(\sum_{j=1}^3 D_{ij} u^\phi(t-\sigma_{i})\right)=2\sum_{i\in\Sigma}\bE\left(D_{i} u^\phi(t-\sigma_{i})\right)=2z^\phi(t)\leq 2M.\]
Similarly, the term at (\ref{b2}) is also bounded above by $2M$. Furthermore, the term at (\ref{b3}) can be rewritten as
\[2\sum_{i\in\Sigma}\bE\left(\sum_{j=1}^3 D_{i}\Delta_j' \int_{\gamma^{-1}\ln \Delta_j'}^{0} u^\phi(t+w-\sigma_i)dw\right),\]
where $(\Delta_j')_{j=1}^3$ is a copy of $(\Delta_j)_{j=1}^3$, independent of all the other random variables of the discussion. Applying (\ref{mchar}), this can be evaluated as
\[2\bE\left(\sum_{j=1}^3 \Delta_j' \int_{\gamma^{-1}\ln \Delta_j'}^{0} z^\phi(t+w)dw\right)\leq 3M\bE\left(\sum_{j=1}^3 \Delta_j'|\ln\Delta_j'|\right)<\infty.\]
Putting these pieces together, we obtain that
\begin{equation}\label{i2}
I_2\leq Ce^{-\gamma t}
\end{equation}
for some finite constant $C$.

Finally, note that $I_3$ satisfies
\[I_3 \leq  e^{-2 \gamma t}\sum_{i\in\Sigma}\sum_{j\in\Sigma}\mathbf{E}\left(D_i\phi_i(t-\sigma_i)\phi_{ij}(t-\sigma_{ij})\right),\]
(cf. the proof of \cite[Lemma 5.3]{CH2010}). Again applying (\ref{xchar}), the boundedness of $\phi$ and Lemma \ref{lem::birthTimesSumControl}, it follows that
\begin{eqnarray}
I_3&\leq&  6e^{-2 \gamma t}\sum_{i\in\Sigma} \mathbf{E}\left(D_iZ^\phi_{i}(t-\sigma_{i})\right)\nonumber\\
&=&6e^{-\gamma t}\sum_{i\in\Sigma} \mathbf{E}\left(D_i^2z^\phi(t-\sigma_{i})\right)\nonumber\\
&\leq &6M e^{-\gamma t} \sum_{k=0}^\infty\psi(2)^k\nonumber\\
&=& C  e^{-\gamma t},\label{i3}
\end{eqnarray}
where again $M:=\sup_{t\in\mathbb{R}}z^\phi(t)$, and $C:=6M/(1-\psi(2))$ is a finite constant.

Summing (\ref{i1}), (\ref{i2}) and (\ref{i3}), we obtain that $v$ is bounded for $t\geq 0$. We now check that $v(t)$ is bounded for $t\leq 0$ and converges to $0$ as $t\rightarrow-\infty$. For this, we use the bound $\mathbf{E}(Z^\phi(t)^2)\leq Ce^{(2\gamma +\epsilon)t}$ for $t\in\R$ (cf. \cite[Lemma 4.4]{CH2010}), which implies
\[v(t)\leq  e^{-\gamma t} \left( \bE\left(Z^\phi(t)^2\right)+e^{2\gamma t}z^\phi(t)^2\right)\leq Ce^{\gamma t}\left(e^{\epsilon t}+1\right).\]
Clearly this yields the desired properties of $v(t)$. Finally, to confirm that $v$ is measurable is elementary using the fact that $Z^\phi(t)$ is monotone cadlag, $\mathbf{P}$-a.s.
\end{proof}

\begin{lemma} The function $r$, as defined at (\ref{vdefrdef}), is in $L^1(\mathbb{R})$ and $r(t)\rightarrow0$ as $|t|\rightarrow \infty$.
\end{lemma}
\begin{proof} It follows from the definition of $r$ that, similarly to the proof of Lemma \ref{propsa}, we have $|r(t)|\leq 2e^{\gamma t}(J_1+J_2+J_3)$, where
\begin{eqnarray*}
J_1 &=& e^{-2\gamma t} \mathrm{Var}(\phi(t)), \\
J_2 &=& \mathrm{Var}\left(\sum_{j=1}^3\Delta_{j}z^\phi(t-\sigma_{j})\right), \\
J_3 &=& e^{-2\gamma t} \left|\bE  \left( \phi(t)\sum_{j=1}^3  \bar{Z}_{j}(t-\sigma_{j}) \right)\right|,
\end{eqnarray*}
and we will proceed by showing that the statements of the lemma hold for $e^{\gamma t}J_i$, $i=1,2,3$. As in the previous proof, checking the measurability of the functions is elementary, and so we will restrict ourselves to finding suitable bounds for them. Firstly, we have
\[e^{\gamma t}J_1\leq e^{-\gamma t} \bE\left(\phi(t)^2\right)\leq 6  e^{-\gamma t} \bE\left(\phi(t)\right)= 6u^\phi(t).\]
That $u^\phi\in L^1(\R)$ and $u^\phi(t)\rightarrow0$ as $|t|\rightarrow \infty$ was established in \cite[Lemma 20]{CH2008}, and so the corresponding result for $e^{\gamma t}J_1$ also holds. For $e^{\gamma t}J_2$, we consider the cases $t\leq0$ and $t\geq 0$ separately. In particular, we have $e^{\gamma t}J_2\leq e^{\gamma t} M^2$, which clearly demonstrates that $e^{\gamma t}J_2\in L^1((-\infty,0])$ and $e^{\gamma t}J_2\rightarrow 0$ as $t\rightarrow -\infty$. Furthermore, defining $\hat{z}^{\phi}(t):=z^\phi(t)-z^\phi(\infty)$ as in the previous result and recalling once again that $|\hat{z}^{\phi}(t)|\leq C e^{- \gamma t}$, we are able to deduce that
\[e^{\gamma t}J_2=e^{\gamma t}\mathrm{Var}\left(\sum_{j=1}^3\Delta_{j}\hat{z}^{\phi}(t-\sigma_{j})\right)
\leq e^{\gamma t}\left(3Ce^{-\gamma t}\right)^2=Ce^{-\gamma t},\]
which confirms that $e^{\gamma t}J_2\in L^1([0,\infty))$ and $e^{\gamma t}J_2\rightarrow 0$ as $t\rightarrow\infty$. Finally, for $e^{\gamma t}J_3$ we proceed as follows:
\begin{eqnarray*}
e^{\gamma t}J_3&\leq &  3^{1/2}e^{-\gamma t}\left(\bE (\phi(t)^2)\bE\left(\sum_{j=1}^3 \bar{Z}_{j}(t-\sigma_{j})^2\right)\right)^{1/2}\\
&\leq & C e^{-\gamma t/2} \left(\bE (\phi(t))\bE\left(\sum_{j=1}^3 \Delta_{j} v(t-\sigma_{j})\right)\right)^{1/2}\\
&\leq & C u^\phi(t)^{1/2},
\end{eqnarray*}
where for the final inequality we use the fact that $v$ is bounded (Lemma \ref{propsa}). Now, from the proof of \cite[Lemma 20]{CH2008}, it can be seen that  $(u^\phi)^{1/2}\in L^1(\R)$ (and we have already noted that $u^\phi(t)\rightarrow 0$ as $|t|\rightarrow \infty$). Consequently, we have the desired result for $e^{\gamma t} J_3$. The lemma follows.
\end{proof}

\begin{lemma} The measure $\nu_\gamma$ is a non-atomic Borel probability measure on $[0,\infty)$ and also $\int_0^\infty t\nu_\gamma(dt)=1$.
\end{lemma}
\begin{proof} The proof of this lemma is straightforward and omitted.
\end{proof}

In view of the preceding three lemmas, the following result is an immediate application of the double-sided renewal theorem of \cite[Theorem 5]{Karlin1955}.

\begin{proposition} \label{varconv} The function $v$ converges as $t\rightarrow \infty$ to the finite constant $v(\infty):=\int_{-\infty}^{\infty}r(t)dt$.
\end{proposition}

\subsection{Verification of Conditions \ref{cond::integrabilityCondCLT} and \ref{cond::momentCondCLT}}

It now only remains for us to check Conditions \ref{cond::integrabilityCondCLT} and \ref{cond::momentCondCLT} before we can apply Theorem \ref{thm::CLT} to deduce the desired central limit theorem for the eigenvalue counting function of the Brownian CRT. We start by working towards an estimate for the third moment of $\tilde{Z}$, which will confirm Condition \ref{cond::momentCondCLT}, and, to this end, we use another recursion argument. This is similar to the proof of Lemma \ref{lem::thirdMomentPositiveTimes}, but more involved due to the lack of a uniform bound for $\bar{\zeta}$. Specifically, iterating \eqref{z3w}, we deduce that for any $k\in\mathbb{N}$
\[\bar Z(t)^3 = \sum_{|i|<k}W_i(t-\sigma_i) + \sum_{i\in \Sigma_k}\bar Z_i(t-\sigma_i)^3.\]
The following lemma establishes that the expectation of the remainder term here converges to 0 as $k\rightarrow \infty$.

\begin{lemma}\label{yrem} For each $t\in \mathbb{R}$,
\[\lim_{k\rightarrow\infty}\bE\left(\sum_{i\in \Sigma_k}\left|\bar{Z}_i(t-\sigma_i)\right|^3\right)=0.\]
\end{lemma}
\begin{proof} By Cauchy-Schwarz and Lemma \ref{propsa},
\begin{eqnarray}
\bE\left(\left|\bar Z(t)\right|^3\right)&\leq & \bE\left(\left|\bar Z(t)\right|\left(Z^\phi(t)^2+e^{2\gamma t} z^\phi(t)^2\right)\right)\nonumber\\
&\leq & Ce^{\gamma t/2}\left(\left(\bE(Z^\phi(t)^4)\right)^{1/2}+e^{2\gamma t}M^2\right).\label{ysplit}
\end{eqnarray}
Applying the characterisation of $Z^{\phi}(t)$ at (\ref{xchar}), we have that
\[\bE(Z^{\phi}(t)^4)= \sum_{i,j,k,l\in \Sigma}\bE\left(\phi_i(t-\sigma_i)\phi_j(t-\sigma_j)\phi_k(t-\sigma_k)\phi_l(t-\sigma_l)\right).\]
Since
\begin{equation}\label{etabound}
\phi_i(t)\leq 6\bone_{\{t\geq -\ln \delta_i\}}\leq 6e^{\theta \gamma t}\delta_i^{\theta\gamma},
\end{equation}
where $\delta_i$ is defined to be the diameter of the metric space $(\mathcal{T}_i,d_{\mathcal{T}_i})$, which is a random variable with a finite positive moments of all orders (see proof of \cite[Lemma 20]{CH2008}), it follows that, for any $\theta,\epsilon>0$,
\begin{eqnarray}
{\bE(Z^{\phi}(t)^4)}&\leq& Ce^{4\theta\gamma t}\sum_{i,j,k,l\in \Sigma} \bE\left(D_i^\theta D_j^\theta D_k^\theta D_l^\theta \delta_i^{\theta\gamma}\delta_j^{\theta\gamma}\delta_k^{\theta\gamma}\delta_l^{\theta\gamma}\right)\nonumber\\
&\leq & C e^{4\theta\gamma t}\sum_{i,j,k,l\in \Sigma} \bE\left(D_i^{\theta(1+\epsilon)}D_j^{\theta(1+\epsilon)}D_k^{\theta(1+\epsilon)}D_l^{\theta(1+\epsilon)}\right)^{1/(1+\epsilon)}.\label{x4bound}
\end{eqnarray}

\begin{figure}
\begin{center}
 \scalebox{0.4}{\includegraphics{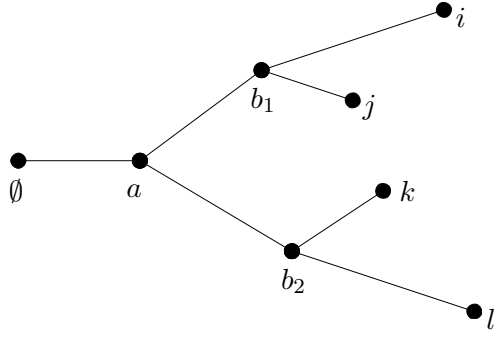}}
 \rput(-6.25,1.7){$\emptyset$}
 \rput(-4.7,1.7){$a$}
 \rput(-3,2.9){$b_1$}
 \rput(-2.6,.5){$b_2$}
 \rput(-.4,4){$i$}
 \rput(-1.6,2.8){$j$}
 \rput(-1.1,1.7){$k$}
 \rput(0,0){$l$}
\end{center}
 \caption{A possible configuration of $i,j,k,l$.}\label{config}
\end{figure}

Now, suppose $\Sigma$ is viewed as a graph tree with edges between $i|_{|i|-1}$ and $i$ for each $i\in \Sigma\backslash \{\emptyset\}$, and the subtree of $\Sigma$ spanning $i,j,k,l$ (and the root $\emptyset$) has shape as shown in Figure \ref{config}, where we assume that $a,b_1,b_2,i,j,k,l$ are distinct. It is then straightforward to check from the independence structure of $(D_i)_{i\in \Sigma}$ that $\bE(D_i^{\theta(1+\epsilon)}D_j^{\theta(1+\epsilon)}D_k^{\theta(1+\epsilon)}D_l^{\theta(1+\epsilon)})$ is bounded above by
\begin{eqnarray*}
&\bE\left(D_a^{4\theta(1+\epsilon)}\right)
\bE\left(\frac{D_{b_1}^{2\theta(1+\epsilon)}}{D_{b_1|_{|a|+1}}^{2\theta(1+\epsilon)}}\right)
\bE\left(\frac{D_{b_2}^{2\theta(1+\epsilon)}}{D_{b_2|_{|a|+1}}^{2\theta(1+\epsilon)}}\right)&\\
&\times\bE\left(\frac{D_{i}^{\theta(1+\epsilon)}}{D_{i|_{|b_1|+1}}^{\theta(1+\epsilon)}}\right)
\bE\left(\frac{D_{j}^{\theta(1+\epsilon)}}{D_{j|_{|b_1|+1}}^{\theta(1+\epsilon)}}\right)
\bE\left(\frac{D_{k}^{\theta(1+\epsilon)}}{D_{k|_{|b_2|+1}}^{\theta(1+\epsilon)}}\right)
\bE\left(\frac{D_{l}^{\theta(1+\epsilon)}}{D_{l|_{|b_2|+1}}^{\theta(1+\epsilon)}}\right),&
\end{eqnarray*}
which is equal to
\small
\begin{equation}\label{3terms}
\left(\frac{\psi(4\theta(1+\epsilon))}{3}\right)^{|a|}
\left(\frac{\psi(2\theta(1+\epsilon))}{3}\right)^{|b_1|+|b_2|-2|a|-2}
\left(\frac{\psi(\theta(1+\epsilon))}{3}\right)^{|i|+|j|+|k|+|l|-2|b_1|-2|b_2|-4},
\end{equation}
\normalsize
where we again recall $\psi(x)=3/(1+2x)$ for $x>-1/2$. Since $\psi(\theta)<1$ for any $\theta>1$ and  $3^{\epsilon/(1+\epsilon)}\psi(\theta(1+\epsilon))^{1/(1+\epsilon)}\rightarrow \psi(\theta)$ as $\epsilon\rightarrow0$, if we are given any $\theta>1$, then it is possible to choose $\epsilon>0$ such that $3 ({\psi(\theta(1+\epsilon))}/{3})^{{1}/{1+\epsilon}}<1$. By summing (\ref{3terms}) over all suitable $a,b_1,b_2,i,j,k,l$ for such a choice of $\theta$ and $\epsilon$, it follows that the terms of the form considered contribute at most the finite amount
\small
\[\left(\frac{1}{1-3^{\frac{\epsilon}{1+\epsilon}}\psi(4\theta(1+\epsilon))^{\frac{1}{1+\epsilon}}}\right)
\left(\frac{3}{1-3^{\frac{\epsilon}{1+\epsilon}}\psi(2\theta(1+\epsilon))^{\frac{1}{1+\epsilon}}}\right)^2
\left(\frac{3}{1-3^{\frac{\epsilon}{1+\epsilon}}\psi(\theta(1+\epsilon))^{\frac{1}{1+\epsilon}}}\right)^4\]
\normalsize
to the sum at (\ref{x4bound}). For other configurations of $i,j,k,l$, it is possible to proceed similarly, and consequently prove that, for any $\epsilon>0$, $\bE(Z^\phi(t)^4)\leq Ce^{(4\gamma+\epsilon)t}$.

Returning to (\ref{ysplit}), the bound of the previous paragraph implies $\bE(|\bar{Z}(t)|^3)\leq Ce^{5\gamma t/2}(1\vee e^{\epsilon t})$, and so
\[\bE\left(\sum_{i\in \Sigma_k}\left|\bar{Z}_i(t-\sigma_i)\right|^3\right)
\leq  C e^{5\gamma t/2}(1\vee e^{\epsilon t}) \bE\left(\sum_{i\in \Sigma_k}D_i^{5/2}\right)\leq C\psi(5/2)^k\]
which converges to 0 as $k\rightarrow \infty$.
\end{proof}

The first main result of this section is the following, which establishes that Condition \ref{cond::momentCondCLT} holds in the present setting.

\begin{proposition}\label{moment3} We have that $\sup_{t\in\mathbb{R}}\bE(|\tilde{Z}(t)|^3)<\infty$.
\end{proposition}
\begin{proof} As a result of the previous lemma, we have that $\bar Z(t)^3 = \sum_{i\in\Sigma}W_i(t-\sigma_i)$. Hence, from the definition of $W$, we deduce that $\bE(|\bar{Z}(t)|^3)\leq \bE(K_1)+\bE(K_2)+\bE(K_3)+\bE(K_4)$, where $K_1$, $K_2$, $K_3$, $K_4$ are defined to be the terms appearing in equations (\ref{eq::firstBitThridMoment}) to (\ref{eq::fourthBitThirdMoment}) respectively, and it will be our goal to show that $e^{-t}\bE(K_i)$ is bounded for $i=1,2,3,4$.

Applying the bound for $\phi$ at (\ref{etabound}) and the estimate $|\hat{z}^{\phi}(t)|=|z^\phi(t)-z^\phi(\infty)|\leq Ce^{-\gamma t}$ (as well as recalling that $z^\phi$ is a bounded function), it is straightforward to deduce the existence of a deterministic constant $C$ such that, $\bP$-a.s., $|\bar\zeta_i(t)|\leq C(1\wedge(e^{\gamma t}(1+\delta_i^{\gamma})))$. This bound implies $|\bar\zeta_i(t)|=|\bar\zeta_i(t)|^{1/2}|\bar\zeta_i(t)|^{1/2}\leq Ce^{\gamma t/2}(1+\delta_i^{\gamma/2})$, and so $e^{-t}\bE(K_1)$ is bounded above by
\[ C e^{-t}\bE\left(\sum_{i\in \Sigma} e^{t-\sigma_i}(1+\delta_i)\right)
= C \sum_{i\in \Sigma}\bE\left(D_i^{3/2}\right)\bE(1+\delta_i)
=C \sum_{k=0}^\infty \psi(3/2)^k,\]
which is finite, because $\psi(3/2)<1$.

Secondly, we proceed similarly to obtain that
\begin{eqnarray*}
e^{-t}\bE(K_2)&\leq & Ce^{-t} \bE\left(\sum_{i\in \Sigma}e^{\gamma(t-\sigma_i)}(1+\delta_i^\gamma)\sum_{j=1}^3   |\bar{Z}_{ij}(t-\sigma_{ij})|\right)\\
&\leq &Ce^{-t/3} \bE\left(\sum_{i\in \Sigma}D_i\sum_{j=1}^3
\bE\left( (1+\delta_{i1}^\gamma+\delta_{i2}^\gamma+\delta_{i3}^\gamma)\bar{Z}_{ij}(t-\sigma_{ij})\vline \mathcal{F}_{|i|+1}\right)\right)\\
&\leq&Ce^{-t/3} \bE\left(\sum_{i\in \Sigma}D_i\sum_{j=1}^3  \bE\left( \bar{Z}_{ij}(t-\sigma_{ij})^2\vline \mathcal{F}_{|i|+1}\right)^{1/2}\right)\\
&\leq&Ce^{-t/3} \bE\left(\sum_{i\in \Sigma}D_i\sum_{j=1}^3
\bE(e^{(t-\sigma_{ij})/3})\right)\\
&\leq&C \sum_{k=0}^\infty \psi(3/2)^k,
\end{eqnarray*}
where the third inequality is a conditional Cauchy-Schwarz estimate (we also apply the fact that the moments of $\delta_i$ are finite), and to deduce the fourth we use Lemma \ref{propsa}.

For the third term, we start by observing that, similarly to (\ref{ysplit}), $\bE(|\bar{Z}(t)|^3)$ is bounded above by
\[\bE\left(\left|\bar{Z}(t)\right|^{7/4}\left(Z^\phi(t)^{5/4}+e^{5\gamma t/4}z^\phi(t)^{5/4}\right)\right)\leq  Ce^{7\gamma t/8}\left(\left(\bE(Z^\phi(t)^{10})\right)^{1/8}+ e^{5\gamma t/4}M^{5/4}\right).\]
By making the obvious extensions to the argument applied in the proof of Lemma \ref{yrem}, it is possible to check that, for any $\epsilon>0$, $\bE(Z^\phi(t)^{10})\leq Ce^{(10\gamma+\epsilon)t}$, and hence $\bE(|\bar{Z}(t)|^3)\leq Ce^{17t/12}(e^{\epsilon t}\vee 1)$. For any $a\in[0,1]$, we also have that $|\bar\zeta_i(t)|=|\bar\zeta_i(t)|^{a}|\bar\zeta_i(t)|^{1-a}\leq Ce^{(1-a)\gamma t}(1+\delta_i^{(1-a)\gamma})$. Putting these bounds together yields
\begin{eqnarray*}
e^{-t}\bE(K_3)&\leq & Ce^{-t}\bE\left(\sum_{i\in \Sigma}e^{(1-a)\gamma(t-\sigma_i)}(1+\delta_i^{(1-a)\gamma})\sum_{j,k=1}^3  |\bar{Z}_{ij}(t-\sigma_{ij})\bar{Z}_{ik}(t-\sigma_{ik})|\right)\\
&\leq&Ce^{-(1-(1-a)\gamma)t}\bE\left(\sum_{i\in \Sigma}D_i^{1-a}
\sum_{j,k=1}^3    \bE\left(|\bar{Z}_{ij}(t-\sigma_{ij})|^3\vline \mathcal{F}_{|i|+1}\right)^{1/3}\right.\\
&&\hspace{130pt}\left.\times{\vphantom{\sum_{k=1}^3 }}\bE\left(|\bar{Z}_{ik}(t-\sigma_{ik})|^3\vline \mathcal{F}_{|i|+1}\right)^{1/3}\right)\\
&\leq &Ce^{-(1-(\frac{29}{12}-a)\gamma)t}(e^{\gamma\epsilon t}\vee 1)\bE\left(\sum_{i\in \Sigma}D_i^{\frac{29}{12}-a}\right)\\
&=&Ce^{-(1-(\frac{29}{12}-a)\gamma)t}(e^{\gamma\epsilon t}\vee 1)\sum_{k=0}^\infty \psi\left(\frac{29}{12}-a\right)^k,
\end{eqnarray*}
where the second inequality is an application of H\"{o}lder (and we bound the $\delta_i$ term similarly to how this was controlled when estimating $K_2$ above). If $a=\frac{11}{12}$, then for $t\leq 0$ we obtain from this that $e^{-t}K_3\leq C\sum_{k=0}^\infty\psi(3/2)^k<\infty$. If $a=1$, then it is possible to choose $\epsilon$ small enough so that the above bound implies, for $t\geq 0$, $e^{-t}K_3\leq C\sum_{k=0}^\infty\psi(17/12)^k<\infty$.

Finally, we can proceed as in the proof of Lemma \ref{lem::birthTimesSumControl} to deduce that $e^{-t}\bE(K_4)\leq C \sum_{k=0}^\infty \psi(3/2)^k$. The additional input needed to do this is provided by Lemma \ref{propsa} again. This completes the proof of the proposition.
\end{proof}

From the proof of the previous result, we have that $|\bar{\zeta}_i(t)|\leq C$ for some deterministic constant $C$. Hence we can deduce Condition \ref{cond::integrabilityCondCLT} by applying the same argument as that used to establish Lemma \ref{lem::checkIntegrabilityCondition}. We simply state the conclusion.

\begin{proposition}\label{cond26crt} For every $\epsilon \in (0, 1/2)$,
\[e^{-\gamma t/2} \sum_{\sigma_i \leq \epsilon t} \bar \zeta_i(t- \sigma_i) \to 0,\]
in probability as $t \to \infty$.
\end{proposition}

To complete the proof of Theorem \ref{crtclt}, note that, by definition and (\ref{dnbracket}),
\[\left|\frac{N_\mathcal{T}(\lambda)-\bE N_\mathcal{T}(\lambda)}{\lambda^{1/3}}-\tilde{Z}(\ln \lambda)\right|\leq 2\lambda^{-1/3}.\]
Hence Propositions \ref{varconv}, \ref{moment3} and \ref{cond26crt} allow us to apply Theorem \ref{thm::CLT} to deduce the result with
\[C_0:=z^\phi(\infty)\equiv \int_{-\infty}^{\infty}u^\phi(t)dt\in(0,\infty),\hspace{20pt}C_1:=v(\infty)\equiv
\int_{-\infty}^{\infty}z(t)dt\in[0,\infty).\]

\section*{Appendix}

The following table contains the approximate values of $\rho_\pm$ and $R$ for different values of $\alpha$ with $\gamma=\tfrac12$, as required in the proof of Theorem~\ref{thm:ex1}.

\begin{center}
\begin{tabular}{|c|c|c|}
  \hline
	$\alpha$		&		$\rho_\pm$		&		$R$		\\ \hline\hline
$	59	$	&	$	0.495347 \pm 9.10306i	$	&	$	0.0964835	$	\\ \hline
$	60	$	&	$	0.503788 \pm 9.1027 i	$	&	$	0.0970307	$	\\ \hline
$	61	$	&	$	0.511952 \pm 9.10235i	$	&	$	0.0975642	$	\\ \hline
$	62	$	&	$	0.519852 \pm 9.10199i 	$	&	$	0.0980839	$	\\ \hline
$	63	$	&	$	0.527501 \pm 9.10164i	$	&	$	0.0985906	$	\\ \hline
$	64	$	&	$	0.534909 \pm 9.1013i	$	&	$	0.0990848	$	\\ \hline
$	65	$	&	$	0.54209 \pm  9.10096i	$	&	$	0.0995668	$	\\ \hline
$	66	$	&	$	0.549052 \pm 9.10062i	$	&	$	0.100037	$	\\ \hline
$	67	$	&	$	0.555805 \pm 9.10028i	$	&	$	0.100496	$	\\ \hline
$	68	$	&	$	0.56236 \pm 9.09995i	$	&	$	0.100945	$	\\ \hline
$	69	$	&	$	0.568724 \pm 9.09963i	$	&	$	0.101382	$	\\ \hline
$	70	$	&	$	0.574906 \pm  9.09931i	$	&	$	0.10181	$	\\ \hline
$	71	$	&	$	0.580913 \pm 9.09899i	$	&	$	0.102228	$	\\ \hline
$	72	$	&	$	0.586753 \pm 9.09867i	$	&	$	0.102636	$	\\ \hline
$	73	$	&	$	0.592432 \pm 9.09836i	$	&	$	0.103034	$	\\ \hline
$	74	$	&	$	0.597958 \pm 9.09806i	$	&	$	0.103425	$	\\ \hline
$	75	$	&	$	0.603335 \pm 9.09776i	$	&	$	0.103806	$	\\ \hline
$	76	$	&	$	0.608571 \pm 9.09746i	$	&	$	0.10418	$	\\ \hline
$	77	$	&	$	0.613671 \pm 9.09717i	$	&	$	0.104545	$	\\ \hline
$	78	$	&	$	0.618639 \pm 9.09688i	$	&	$	0.104902	$	\\ \hline
$	79	$	&	$	0.623482 \pm 9.09659i	$	&	$	0.105252	$	\\ \hline
$	80	$	&	$	0.628203 \pm 9.09631i	$	&	$	0.105594	$	\\ \hline
\end{tabular}
\end{center}

\section*{Acknowledgments}

We would like to thank Mohsin Javed for help with the numerics and Charles Stone for providing useful references. The first author was supported by a Berrow Foundation scholarship and the Swiss National Science Foundation.

\end{document}